\documentclass[oneside,english,a4wide]{amsart}

\usepackage{graphicx}
\usepackage{float} 
\usepackage{subfigure}
\usepackage[]{caption2} 
\renewcommand{\thesubfigure}{(\roman{subfigure})}
\makeatletter \renewcommand{\@thesubfigure}{\thesubfigure \space}
\renewcommand{\p@subfigure}{} \makeatother

\usepackage[latin9]{inputenc}
\usepackage[a4paper]{geometry}
\usepackage{mathabx}
\usepackage{esvect}
\usepackage{babel}
\usepackage{mathtools}
\usepackage{amstext}
\usepackage{amsthm}
\usepackage{amssymb}
\usepackage{enumerate}
\usepackage{esint}
\usepackage{graphicx}
\usepackage{bbm}
\usepackage[all]{xy}
\usepackage{mathrsfs}
\usepackage[unicode=true,
bookmarks=true,bookmarksnumbered=true,bookmarksopen=true,bookmarksopenlevel=2,
breaklinks=false,pdfborder={0 0 1},backref=false,colorlinks=false]
{hyperref}
\hypersetup{
	colorlinks,linkcolor=blue,anchorcolor=blue,citecolor=blue}
\usepackage{breakurl}
\usepackage{autobreak}

\usepackage[svgnames]{xcolor} 

\usepackage{graphicx} 

\usepackage{amsmath}

\makeatletter
\numberwithin{equation}{section}
\numberwithin{figure}{section}
\theoremstyle{plain}
\newtheorem{thm}{\protect\theoremname}[section]
\theoremstyle{plain}
\newtheorem{prop}[thm]{\protect\propositionname}
\newtheorem{proposition}[thm]{Proposition}

\theoremstyle{definition}
\newtheorem{defn}[thm]{\protect\definitionname}
\theoremstyle{plain}
\newtheorem{lem}[thm]{\protect\lemmaname}
\newtheorem{lemma}[thm]{Lemma}
\theoremstyle{remark}
\newtheorem{rem}[thm]{\protect\remarkname}

\theoremstyle{definition}
\newtheorem{definition}[thm]{Definition}
\newtheorem*{claim*}{Claim}
\newtheorem{example}[thm]{Example}

\theoremstyle{remark}
\newtheorem{remark}[thm]{Remark}
\theoremstyle{definition}
\newtheorem*{defn*}{\protect\definitionname}

\usepackage{babel}
\usepackage{amscd}

\makeatother

\providecommand{\definitionname}{Definition}
\providecommand{\lemmaname}{Lemma}
\providecommand{\propositionname}{Proposition}
\providecommand{\remarkname}{Remark}
\providecommand{\theoremname}{Theorem}

\newcommand{\Rmnum}[1]{\expandafter\@slowromancap\romannumeral #1@}
\makeatother

\newcommand{\8}{\infty}

\newcommand{\la}{\langle}
\newcommand{\ra}{\rangle}

\newcommand{\be}{\begin{eqnarray*}}
	\newcommand{\ee}{\end{eqnarray*}}
\newcommand{\beq}{\begin{equation}}
	\newcommand{\eeq}{\end{equation}}
\newcommand{\beqn}{\begin{equation*}}
	\newcommand{\eeqn}{\end{equation*}}
\newcommand{\bs}{\begin{split}}
	\newcommand{\es}{\end{split}}

\begin{document}

	
	
	\title[Complex median method and Schatten class membership of commutators]{Complex median method and Schatten class membership of commutators}
	
	
	\author[Zhenguo Wei]{Zhenguo Wei}
	\address{Laboratoire de Math{\'e}matiques, Universit{\'e} de Bourgogne Franche-Comt{\'e}, 25030 Besan\c{c}on Cedex, France}
	\email{zhenguo.wei@univ-fcomte.fr}

	\author[Hao Zhang]{Hao Zhang}
	\address{
	Department of Mathematics, University of Illinois Urbana-Champaign, USA}
	\email{hzhang06@illinois.edu}
	\date{}
	\maketitle
	
	\begin{abstract}
	This article is devoted to the study of the Schatten class membership of commutators involving singular integral operators. We utilize martingale paraproducts and Hyt\"{o}nen's dyadic martingale technique to obtain sufficient conditions on the weak-type and strong-type Schatten class membership of commutators in terms of Sobolev spaces and Besov spaces respectively. 
	
	We also establish the complex median method, which is applicable to complex-valued functions. We apply it to get the optimal necessary conditions on the weak-type and strong-type  Schatten class membership of commutators associated with non-degenerate kernels. This resolves the problem on the characterization of the weak-type and strong-type  Schatten class membership of commutators.
	
	Our new approach is built on Hyt\"{o}nen's dyadic martingale technique and the complex median method. Compared with all the previous ones, this new one is more powerful in several aspects: $(a)$ it permits us to deal with more general singular integral operators with little smoothness; $(b)$ it allows us to deal with commutators with complex-valued kernels; $(c)$ it turns out to be powerful enough to deal with the weak-type and strong-type  Schatten class of commutators in a universal way.
	\end{abstract}
	
 \tableofcontents{}
	
	\section{Introduction}\label{Introduction}
In harmonic analysis, commutators involving singular integral and multiplication operators are generalizations of Hankel type operators.  Hankel operators were first studied by Hankel in \cite{Hankel}. Since then they have become an important class of operators. Later, Nehari  characterized the boundedness of Hankel operators on the Hardy space $H_2(\mathbb{T})$ in terms of the $BMO$ space in \cite{Neha}, and Hartman discussed their compactness by the $VMO$ space in \cite{Hart}. In \cite{V1}, Peller obtained the strong-type Schatten class criterion of Hankel operators by Besov space for $1\leq p<\8$, while the case $0<p<1$ was investigated by Peller in \cite{V2} and Semmes in \cite{SS}, respectively. The weak-type Schatten class of Hankel operators can be seen in \cite[Theorem 4.4, Chapter 6]{PVV2003}.

The boundedness, compactness and Schatten class membership of commutators have also been studied extensively. In \cite{CRW}, Coifman, Rochberg and Weiss showed the boundedness of commutators with regards to the $BMO$ space on $\mathbb{R}^n$, which yields a new characterization of $BMO$. Soon after their work, Uchiyama sharpened one of their results and showed the compactness of commutators by virtue of the $CMO$ space in \cite{Uchi}. The strong-type Schatten class membership of commutators was developed by Janson and Wolff in terms of Besov spaces in \cite{JW}. Afterwards, Janson and Peetre established a fairly general framework to investigate the boundedness and strong-type Schatten class of commutators in \cite{JPe}. As for the weak-type Schatten class of commutators, it was studied in \cite{RSe} and \cite{SEFD2017}. It is worth pointing out that the weak-type Schatten class is deeply related to noncommutative geometry, and the interested reader is referred to \cite{Alain1997,MSX2019,MSX2020} for more details.

In addition, Petermichl discovered an explicit representation formula for the one-variable Hilbert transform as an average of dyadic shift operators to investigate Hankel operators with matrix symbol in \cite{SPe}. This new idea turns out to be very useful and powerful to deal with problems related to singular integral operators, such as Hankel operators and weighted inequalities. Nazarov, Treil and Volberg in  \cite{NTV} invented another type of dyadic representation of singular integral operators and used it to solve the $T(1)$ and the $T(b)$ theorems on non-homogeneous spaces. Later, Hyt\"{o}nen refined in an essential way the method of Nazarov, Treil and Volberg to establish a novel and beautiful dyadic representation, and finally settled the well-known $A_2$ conjecture \cite{TH1}. Hyt\"{o}nen's dyadic representation will be a paramount tool of our arguments.

We would like to highlight that the aforementioned different kinds of dyadic representation of singular integral operators \cite{SPe,NTV,TH1} are all based on dyadic operators. Meanwhile, as Hankel type operators, dyadic operators serve as crucial tools in harmonic analysis.

Nowadays, the study of the Schatten class of commutators attracts much attention and is going through vast development. We refer the interested reader to \cite{FLL2023,FLLX2024,FLMSZ2024,GLW2023,LLW2023,LLW2024,LLWW2024}. Motivated by these remarkable works, we aim to investigate the weak-type and strong-type Schatten class of commutators in this paper.

To present our main results, we first provide the setup for singular integral operators. Let $T\in B(L_2(\mathbb{R}^n))$ be a singular integral operator with a kernel $K(x, y)$,  i.e. for any $f\in L_2(\mathbb{R}^n)$
\[T(f)(x)=\int_{\mathbb{R}^n}K(x, y)f(y)dy, \quad x\notin \mathrm{supp}f.\]
We assume that $K(x, y)$ is defined for all $x\neq y$ on $\mathbb{R}^n\times \mathbb{R}^n$ and satisfies the following standard kernel estimates:
\begin{equation}\label{standard}
	\begin{cases}
		\displaystyle|K(x,y)|\le \frac{C}{|x-y|^n},\\
		\displaystyle|K(x,y)-K(x',y)|+|K(y,x)-K(y,x')|\le \frac{C |x-x'|^\alpha}{|x-y|^{n+\alpha}}
	\end{cases}
\end{equation}
for all $x,x',y\in\mathbb{R}^n$ with $|x-y|>2|x-x'|>0$ and some fixed $\alpha\in (0, 1]$ and constant $C>0$. 

In particular, if for any $x\neq y$ 
\begin{equation}\label{phi2333}
	K(x, y)= \phi(x-y),
\end{equation}
where $\phi$ is homogeneous of degree $-n$ with mean value zero on the unit sphere, then $T$ is called a Calder\'{o}n-Zygmund transform. 

In the sequel,  $T:L_2(\mathbb{R}^n)\to L_2(\mathbb{R}^n)$ will always be assumed to satisfy the above standard estimates (\ref{standard}) and to be bounded. The celebrated David-Journé $T(1)$ theorem in \cite{DaJou} asserts that for any singular integral operator $T$ satisfying (\ref{standard}), $T$ is bounded on $L_2(\mathbb{R}^n)$ if and only if $T(1)$ and $T^*(1)$ both belong to $BMO(\mathbb{R}^n)$ and $T$ satisfies the weak boundedness property. We recall that $BMO(\mathbb{R}^n)$ is the space consisting of all locally integrable complex-valued functions $b$ such that
$$  \|b\|_{BMO(\mathbb{R}^n)}= \sup _{\substack{Q \subset \mathbb{R}^n \\ Q \operatorname{cube}}} \bigg(\frac{1}{m(Q)} \int_Q\Big|b-\big(\frac{1}{m(Q)} \int_Q b\ d m\big)\Big|^2 d m\bigg)^{1/2}<\infty, $$
where $m$ is Lebesgue measure on $\mathbb{R}^n$.

Assume that $b$ is a locally integrable complex-valued function, and let $M_b$ be the pointwise multiplication by $b$. The commutator is defined by $C_{T, b}=[T, M_b]=T M_b - M_b T$, that is for any $f\in L_2(\mathbb{R}^n)$,
$$  C_{T, b}(f)=T(b\cdot f)-b\cdot T(f).  $$

Suppose that $1<p<\infty$. The homogeneous Besov space $\pmb{B}_p(\mathbb{R}^n)$ is defined as the completion of all locally integrable complex-valued functions $b$ satisfying
\begin{equation}\label{funeq}
	\|b\|_{\pmb{B}_p(\mathbb{R}^n)}=	\bigg(\int_{\mathbb{R}^n\times\mathbb{R}^n}\frac{|b(x)-b(y)|^p}{|x-y|^{2n}}dxdy\bigg)^{\frac{1}{p}}<\infty,
\end{equation}  
with respect to the semi-norm  $\|\cdot \|_{\pmb{B}_p(\mathbb{R}^n)}$. If $p>n$, $\pmb{B}_p(\mathbb{R}^n)$ coincides with the classical Besov space of parameters $(p,p,n/p)$, namely the space $\Lambda_\alpha^{p,q}(\mathbb{R}^n)$ in \cite[Chapter V, \S 5]{ST}. Let $\dot{W}^1_p(\mathbb{R}^n)$ be the homogeneous Sobolev space consisting of all locally integrable complex-valued functions $b$ whose distributional gradient satisfies $\nabla b\in L_p(\mathbb{R}^n)$ endowed with the norm $\| \nabla b\|_{L_p(\mathbb{R}^n)}$.

The study of the Schatten class $S_p$ of commutators can be tracked back to Peller \cite{V1,V2} in the context of complex analysis. As for the Euclidean setting, this dates back to Janson and Wolff \cite{JW}. They obtained the following theorem ({\cite[Theorem 1]{JW}}), which shows a striking difference between the one dimensional and higher dimensional cases: $p=n$ becomes a cut-point for $n\ge 2$. This is now called the Janson-Wolff phenomenon.
\begin{thm}\label{thm1.5}
	Let $T$ be a Calder\'{o}n-Zygmund transform with a kernel $\phi$ defined in \eqref{phi2333}. Assume that $\phi$ is $C^{\infty}(\mathbb{R}^n)$ except at the origin and not identically zero. 
	
	Suppose that $n \geq 2$ and $0<p<\infty$. For $0<p\leq n$, $ C_{T, b}\in S_p(L_2(\mathbb{R}^n))$ if and only if $b$ is constant. For $p>n$, $ C_{T, b}\in S_p(L_2(\mathbb{R}^n))$ if and only if $b\in \pmb{B}_p(\mathbb{R}^n)$.
\end{thm}

For the cut-point $p=n\geq 2$, we have the following remarkable sharp quantitative characterization of weak-type Schatten class $S_{n,\infty}$ membership, which is first explicitly stated by Lord, McDonald, Sukochev and Zanin in \cite[Theorem 1]{SEFD2017}, and they also gave a complete and different proof of it. However, note that Rochberg and Semmes' innovative work \cite{RSe} describes the $S_{n,\infty}$-norm of commutators by some special type of oscillation spaces. Then these oscillation spaces are shown to be equivalent to the Sobolev spaces $\dot{W}^1_n(\mathbb{R}^n)$ by Rochberg and Semmes for one direction and  Connes, Sullivan and Teleman \cite[Appendix]{CST1994} for the other; see also Frank's recent paper \cite{FRANK} for an alternative proof.

\begin{thm}\label{thmSucochev}
	Suppose that $n\geq 2$ and $b\in L_\8(\mathbb{R}^n)$. Let $R_j$ $(1\leq j\leq n)$ be the Riesz transforms on $\mathbb{R}^n$. Then $ C_{R_j, b}\in S_{n,\infty}(L_2(\mathbb{R}^n))$ if and only if $b\in \dot{W}^1_n(\mathbb{R}^n)$. Moreover, in this case
	$$ \| C_{R_j, b}\|_{S_{n,\infty}(L_2(\mathbb{R}^n))}\approx_n\|b\|_{\dot{W}^1_n(\mathbb{R}^n)}. $$
\end{thm}

Note that in Connes' quantized differential calculus language, a commutator $C_{R_j, b}$ can be interpreted as the quantized differential $\dj{} b$ of $b$. Theorem \ref{thmSucochev} is intimately related to the Dixmier trace formula for $\dj{}b$. Indeed, Lord, McDonald, Sukochev and Zanin first proved  this trace formula, then derived Theorem \ref{thmSucochev}. With regard to this, McDonald, Sukochev and Xiong obtained the analogue of Theorem \ref{thmSucochev} for quantum tori and quantum Euclidean spaces and gave the underlying Dixmier trace formula (see \cite{MSX2019,MSX2020}). This is the first truly noncommutative examples on this subject.

\

The following theorems are two of our main results. They largely extend Theorem \ref{thm1.5} and Theorem \ref{thmSucochev}. Moreover, our approach is completely different from all previous ones, and it is applicable to more general settings including the noncommutative case and homogeneous type spaces (see later paragraphs for more discussions). This new approach is one of our main novelties. Recall that $\alpha$ is the fixed constant in (\ref{standard}). We refer to Section \ref{schattenconv} for the definition of non-degenerate kernels.

	\begin{thm}\label{corollary1.8}
	Let $T\in B(L_2(\mathbb{R}^n))$ be a singular integral operator with kernel $K(x,y)$ satisfying \eqref{standard}. Suppose that $b$ is a locally integrable complex-valued function.\\
	$\mathrm{(1)}$ Sufficiency: Suppose that $n< p<\infty$, $0<\alpha\le 1$ when $n\ge 2$, or $2\le p<\infty$, $0<\alpha\le 1$ when $n= 1$, or $1< p<2$, $1/2\le \alpha\le 1$ when $n=1$.  If $b\in \pmb{B}_p(\mathbb{R}^n)$, then $ C_{T, b}\in S_p(L_2(\mathbb{R}^n))$ and 
	$$ \| C_{T, b}\|_{S_p(L_2(\mathbb{R}^n))}\lesssim_{n, p,T}\big(1+\|T(1)\|_{BMO(\mathbb{R}^n)}+\|T^*(1)\|_{BMO(\mathbb{R}^n)}\big)\|b\|_{\pmb{B}_p(\mathbb{R}^n)}. $$
	$\mathrm{(2)}$ Necessity: Suppose that $1<p<\infty$ and $0<\alpha\leq 1$. Assume in addition that $K(x, y)$ is non-degenerate. If $C_{T, b}\in S_p(L_2(\mathbb{R}^n))$, 
	then $b\in \pmb{B}_p(\mathbb{R}^n)$ and
	\begin{equation*}
		\|b\|_{\pmb{B}_p(\mathbb{R}^n)}\lesssim_{n,p,T} \|C_{T, b}\|_{S_p(L_2(\mathbb{R}^n))}.
	\end{equation*}
	In particular, when $n\geq 2$ and $0<p\le n$, if $C_{T, b}\in S_p(L_2(\mathbb{R}^n))$, then $b$ is constant.\\
	$\mathrm{(3)}$ Suppose that $n< p<\infty$, $0<\alpha\le 1$ when $n\ge 2$, or $2\le p<\infty$, $0<\alpha\le 1$ when $n= 1$, or $1< p<2$, $1/2\le \alpha\le 1$ when $n=1$. Assume that $K(x, y)$ is non-degenerate. Then 
	$ C_{T, b}\in S_p(L_2(\mathbb{R}^n))$ if and only if $b\in \pmb{B}_p(\mathbb{R}^n)$. Moreover, in this case
	$$ \| C_{T, b}\|_{S_p(L_2(\mathbb{R}^n))}\approx_{n,p,T} 	\|b\|_{\pmb{B}_p(\mathbb{R}^n)}.  $$
	In addition, when $n\geq 2$ and $0<p\leq n$, then $ C_{T, b}\in S_p(L_2(\mathbb{R}^n))$ if and only if $b$ is constant.
\end{thm}

\begin{thm}\label{corollary1.10}
	Suppose that $n\geq 2$, $0<\alpha\le 1$ and $T\in B(L_2(\mathbb{R}^n))$ is a singular integral operator with kernel $K(x,y)$ satisfying \eqref{standard}. Suppose that $b$ is a locally integrable complex-valued function.\\
	$\mathrm{(1)}$ Sufficiency: If $b\in \dot{W}^1_n(\mathbb{R}^n)$, then $ C_{T, b}\in S_{n,\infty}(L_2(\mathbb{R}^n))$ and 
	$$ \| C_{T, b}\|_{S_{n,\infty}(L_2(\mathbb{R}^n))}\lesssim_{n, T}\big(1+\|T(1)\|_{BMO(\mathbb{R}^n)}+\|T^*(1)\|_{BMO(\mathbb{R}^n)}\big)\|b\|_{\dot{W}^1_n(\mathbb{R}^n)}. $$
	$\mathrm{(2)}$ Necessity: Assume in addition that $K(x, y)$ is non-degenerate. If $C_{T, b}\in S_{n,\infty}(L_2(\mathbb{R}^n))$, then $b\in\dot{W}^1_n(\mathbb{R}^n)$ and
	\begin{equation*}
		\|b\|_{\dot{W}^1_n(\mathbb{R}^n)}\lesssim_{n,T} \|C_{T, b}\|_{S_{n,\infty}(L_2(\mathbb{R}^n))}.
	\end{equation*}
	$\mathrm{(3)}$ Assume that $K(x, y)$ is non-degenerate. Then $C_{T, b}\in S_{n,\infty}(L_2(\mathbb{R}^n))$ if and only if $b\in\dot{W}^1_n(\mathbb{R}^n)$. Moreover, in this case
	\begin{equation*}
		\|C_{T, b}\|_{S_{n,\infty}(L_2(\mathbb{R}^n))}\approx_{n,T} 	\|b\|_{\dot{W}^1_n(\mathbb{R}^n)}.
	\end{equation*}
\end{thm}

The proofs of the necessity parts of Theorem \ref{corollary1.8} and Theorem \ref{corollary1.10} in most of the recent special cases in the literature are based on the so-called real median method, this prevents one from considering complex-valued kernels or functions (see later discussions for more details on this point). To overcome this obstacle, we invent the complex median method that is stated as follows. This theorem will allow us to treat the necessity parts of both Theorem \ref{corollary1.8} and Theorem \ref{corollary1.10} in essentially a same way. It is perhaps the most original and most innovative result of this article.

\begin{thm}\label{divideS}
	Let $(\Omega,\mathcal{F},\mu)$ be a measure space. Let $I\in \mathcal{F}$ be of finite measure, and $b$ be a measurable function on $I$.
	Then there exist two orthogonal lines $L_1$ and $L_2$ in $\mathbb{C}$ which divide $\mathbb{C}$ into four closed quadrants $T_1$, $T_2$, $T_3$, $T_4$ such that
	\begin{equation*}
		\mu(\{x\in I:b(x)\in T_i\})\ge \frac{1}{16}\mu(I),\quad i\in \{1,2,3,4\}.
	\end{equation*}
\end{thm}

Let us make comments on Theorem \ref{corollary1.8}, Theorem \ref{corollary1.10} and Theorem \ref{divideS}, and compare them with the existing results. First, we would like to emphasize that Theorem \ref{corollary1.8} is more general than Theorem \ref{thm1.5} because it deals with commutators involving general singular integral operators, while \cite{JPe} and \cite{JW} focus on Calder\'{o}n-Zygmund transforms with $C^\infty$ kernels. For the same reason, Theorem \ref{corollary1.10} is also more general than Theorem \ref{thmSucochev}, while \cite{FLLX2024,FLMSZ2024,SEFD2017} handle the Riesz transforms, and \cite{LXY2024} focuses on Calder\'{o}n-Zygmund transforms with $C^\infty$ kernels. In addition, the smoothness of associated kernels is not required both in Theorem \ref{corollary1.8} and Theorem \ref{corollary1.10}.

Theorem \ref{corollary1.8} directly implies Theorem \ref{thm1.5} for $p>n\geq 2$. In particular, when $T$ is a convolution operator whose kernel is smooth and homogeneous of degree $-n$, we give an alternative proof of Theorem \ref{thm1.5} for $p>n\geq 2$ by virtue of dyadic martingale techniques from the perspective of `probability method'. Note that if $T$ is the Hilbert transform (so $n=1$), Theorem \ref{corollary1.8} becomes Peller's celebrated theorem; recently, Pott and Smith \cite{PS} gave a different proof of the sufficiency of Theorem 1.3 in this special case by using the dyadic representation of the Hilbert transform established by Petermichl in \cite{SPe}. Note that Petermichl's representation is much easier than Hyt\"{o}nen's dyadic representation for general singular integral operators that is one of the key techniques in our proofs of the sufficiency of Theorem \ref{corollary1.8} and Theorem \ref{corollary1.10}.

Regarding the estimate of the weak-type Schatten class membership of commutators, it should be pointed out that the condition $b\in L_\infty(\mathbb{R}^n)$ must be required to prove Theorem \ref{thmSucochev} in \cite{SEFD2017}, and their approach is based on the technique of double operator integrals and the theory of pseudodifferential operators. However, Theorem \ref{thmSucochev} is a special case of our Theorem \ref{corollary1.10}. Similarly to Theorem \ref{corollary1.8}, we still use the dyadic martingale techniques to prove Theorem \ref{corollary1.10}. Besides, we only require that $b$ be a locally integrable complex-valued function. More importantly, we extend it to general singular integral operators, instead of just Riesz transforms or Calder\'{o}n-Zygmund transforms.

Therefore, Theorem \ref{corollary1.8} and Theorem \ref{corollary1.10} give a complete picture of the strong-type and weak-type Schatten class membership of commutators involving general singular integral operators for $n\geq 2$; in the one-dimensional case, our method does not allow one to treat the case $0<p\leq1$ while Peller's theorem holds for all $p>0$. Besides, Theorem \ref{corollary1.8} and Theorem \ref{corollary1.10} have also been extended to the semicommutative setting which can be found in \cite{WZ2024,WZ2024non}.

We now comment on Theorem \ref{divideS}. The concept of `median' for real-valued functions was first given by Carleson in \cite{CL1980}. Fujii applied this concept `median' to the theory of weighted norm inequalities in \cite{FN1991}. Lerner improved their work to establish the celebrated median decomposition in \cite{LAK10}. We would like to stress that Lerner's median decomposition is very powerful and has been used to demonstrate many beautiful and important results on weighted norm inequalities. Besides, Lerner, Ombrosi and Rivera-R\'{\i}os employed `median' to estimate the lower bound of the boundedness of commutators in \cite{LOR}. This is now called the real median method.

However, no matter what the median decomposition or the median method, the concept `median' could be available only for real-valued functions before our work. In addition, it is well-acknowledged that the real median method is only applicable to real-valued functions, and unfortunately cannot be applied to complex-valued ones. We extend the definition of `median' to complex-valued functions with the help of our new complex median method. Thus this allows us to deal with singular integral operators with complex-valued kernels. This is one of the main novelties of this article.

\

Now we illustrate the methodologies of the proofs of Theorem \ref{corollary1.8} and Theorem \ref{corollary1.10}. Note that the sufficiency parts do not need the non-degeneracy of kernels. The proofs of the sufficiency parts of Theorem \ref{corollary1.8} and Theorem \ref{corollary1.10} are both based on martingale paraproducts, and the dyadic representation of singular integral operators developed by Hyt\"{o}nen in \cite{TH1} and \cite{TH2}. This is quite surprising as the estimates of the strong-type and weak-type Schatten classes can be unified in the same way. To the best of our knowledge, this is the first time that Hyt\"{o}nen's dyadic representation is utilized to estimate Schatten class of commutators involving general singular integral operators. Compared with the dyadic representation of the Hilbert transform used in \cite{PS}, Hyt\"{o}nen's dyadic representation is much more complicated because of the appearance of dyadic shifts with high complexity and martingale paraproducts. This obviously gives rise to some essential difficulties to overcome. We would like to stress that because of such high complexity of dyadic shift operators, more technical works on martingale paraproducts and related commutators must be done. This is also one novelty of this paper.

 We will show the sufficiency of Theorem \ref{corollary1.8} and Theorem \ref{corollary1.10} in a unified way for $p\geq 2$, $0<\alpha\le 1$ and $n\geq 1$. As for the case $n=1$, the situation is much more delicate, and our method can still work well for $1< p<2$, $1/2\le \alpha\le 1$ after some non-trivial technical modifications of the proof for $p\geq 2$. At the time of this writing, we do not know whether the sufficiency part of Theorem \ref{corollary1.8} holds for $1<p<2$, $0<\alpha< 1/2$ when $n=1$. 
 
 However it should be noted that the reason why the estimate of the strong-type and weak-type Schatten classes in \cite{V1, FLLX2024} can be achieved for $p<2$ is that these papers dealt with strong smoothness assumptions on the kernels of singular integral operators. But in our setting, we just require the standard estimates \eqref{standard}. Indeed, when $n=1$, the sufficiency of Theorem \ref{corollary1.8} still holds for all $1<p<\8$ if the kernel of $T$ satisfies $1/2\leq\alpha\leq1$. This means in this case that the kernel $K(x,y)$ is very regular in some sense. But this does not imply that $K(x, y)$ is smooth. So Theorem \ref{corollary1.8} is also more general than all the previous known results even when $n=1$ except the case $0<p\leq 1$.
 
 The proofs of the necessity parts of Theorem \ref{corollary1.8} and Theorem \ref{corollary1.10} are more delicate. In fact, in order to get lower bounds of  $\| C_{T, b}\|_{S_p(L_2(\mathbb{R}^n))}$ and  $\| C_{T, b}\|_{S_{n,\infty}(L_2(\mathbb{R}^n))}$ in terms of the Besov norm or Sobolev norm of $b$, the kernel $K(x, y)$ associated with $T\in B(L_2(\mathbb{R}^n))$ should not be very small. So we deal with ``non-degenerate'' kernels, which is inspired by Hyt\"{o}nen's recent work \cite{TH3}. Our main ingredient of the proofs for the necessity parts of Theorem \ref{corollary1.8} and Theorem \ref{corollary1.10} is the complex median method, i.e. Theorem \ref{divideS}. We will use this new complex median method to give a new proof of some results in \cite{TH3}, which is about to appear in our subsequent paper \cite{WZ3}.

We also would like to stress that the method in \cite{JPe} which heavily relies on Schur multipliers fails in our current case due to the lack of the smoothness condition of the kernels. For the same reason, Rochberg and Semmes' method in \cite{RSe} does not work either. Our approach turns out to be very generally applicable since the smoothness of kernels associated with singular integral operators is not required. Hence, this approach can be utilized to commutators on spaces of homogeneous spaces, which is about to appear in subsequent papers \cite{FWZ2024,WXZZ2024fi}. In addition, our method also works well for the noncommutative setting, which is already done in  \cite{WZ2024}.


\

From the previous description, we see that our approach is completely different from all previous used methods. Moreover, it just requires some regularity of kernels associated with singular integral operators. Since little smoothness of kernels is required, our method can be employed to establish the Schatten class of commutators on spaces of homogeneous type both in the classical and semicommutative settings, which will appear in our future work \cite{FWZ2024,WXZZ2024fi}. We would like to stress that in the framework of spaces of homogeneous type, most of known results about Schatten-Lorentz class properties of Hankel operators and commutators will be covered. This also reflects that our approach is very powerful and generally applicable to most cases.

A summary of the main techniques and the contents seems to be in order. The remaining of this paper is devoted to the proofs of Theorem \ref{corollary1.8}, Theorem \ref{corollary1.10} and Theorem \ref{divideS}. Section \ref{pre2} presents notation and background, such as $d$-adic martingales, martingale paraproducts and martingale Besov spaces. In Section \ref{key lemma}, we will give some vital lemmas which will be used in the proofs of our main results. More specifically, we describe the strong-type and weak-type Schatten class of commutators involving martingale paraproducts, which is of independent interest (see Proposition \ref{T0est} for more details). Then we introduce
Hyt\"{o}nen's dyadic representation and weak Besov spaces in Section \ref{repre}.

 In Section \ref{Application 3}, by virtue of the Schatten class of martingale paraproducts established in Section \ref{key lemma}, and by Hyt\"{o}nen's dyadic representation for general singular integral operators, we will show the sufficiency of Theorem \ref{corollary1.8}. Our approach differs from all the previous ones in similar situations as it is the first time that Hyt\"{o}nen's dyadic representation is utilized to estimate the weak-type and strong-type Schatten norm. This approach has also been used on spaces of homogeneous type \cite{FWZ2024,WXZZ2024fi}.
  
 In Section \ref{proofdivide}, we invent the complex median method.  As is well-known, the real median method turns out to be very powerful to deal with the lower bound of the boundedness and Schatten class of commutators (see \cite{LOR,TH3,FLL2023,DGKLWY2021} and references therein). However, all previous investigations of the real median method into lower bounds of commutators $[T,M_b]$ require that the kernel of $T$ and $b$ be real-valued functions. With the help of our new complex median method,  we treat complex-valued kernels. This constitutes perhaps one of the most important ideas of this article. We will also use this complex median method to investigate the Schatten-Lorentz class of commutators on spaces of homogeneous type, which is about to appear in subsequent papers \cite{FWZ2024,WXZZ2024fi}. Then we show the necessity part of Theorem \ref{corollary1.8} in Section \ref{schattenconv} by the complex median method.

In Section \ref{WEAK}, we study the weak-type Schatten class of commutators. We prove Theorem \ref{corollary1.10}, which concerns the endpoint estimates of the weak-type Schatten class of commutators, still by virtue of martingale paraproducts, Hyt\"{o}nen's dyadic representation and the complex median method. Indeed, we will prove some stronger results, which generalize the sufficiency part and the necessity part of Theorem \ref{corollary1.10}. Finally, as a byproduct, by the real interpolation, we obtain the Schatten-Lorentz class of commutators.

Throughout this paper, we will use the following notation: $A\lesssim B$ (resp. $A\lesssim_\varepsilon B$) means that $A\le CB$ (resp. $A\le C_\varepsilon B$) for some absolute positive constant $C$ (resp. a positive constant $C_\varepsilon$ depending only on $\varepsilon$). $A\approx B$ or $A\approx_\varepsilon B$ means that these inequalities as well as their inverses hold.  For $1\leq p\leq \infty$, denote by $p'=\frac{p}{p-1}$ the conjugate exponent of $p$.



\bigskip

\section{Preliminaries}\label{pre2}

In this section, we provide notation and background that will be used in this paper.


\subsection{$d$-adic martingales}\label{commumar}

Let $d\ge 2$ be a fixed integer. We are particularly interested in $d$-adic martingales since it is closely related to dyadic martingales on Euclidean spaces. In this subsection, we give a general definition of $d$-adic martingales. Afterwards we will present an orthonormal basis of Haar wavelets for $d$-adic martingales, which will be used to represent martingale paraproducts and to define martingale Besov spaces for $d$-adic martingales (to be defined in Subsection \ref{sec2.3}).

Let $\Omega$ be a measure space endowed with a $\sigma$-finite measure $\mu$. Assume that in $\Omega$, there exists a family of measurable sets $I_{n,k}$ for $n, k\in \mathbb{Z}$ satisfying the following properties:
\begin{enumerate}
	\item  $I_{n,k}$  are pairwise disjoint for any $k$ if $n$ is fixed;
	\item $\cup_{k\in \mathbb{Z}}I_{n,k}=\Omega$ for every $n$;
	\item $I_{n,k}=\cup_{q=1}^d I_{n+1,kd+q-1}$ for any $n, k$, so each $I_{n,k}$ is a union of $d$ disjoint subsets $I_{n+1,kd+q-1}$;
	\item $\mu(I_{n,k})=d^{-n}$ for any $n, k$.
\end{enumerate}
Then $I_{n,k}$ are called $d$-adic intervals, and let $\mathcal{ D}$ be the family of all such $d$-adic intervals. We say that $\mathcal{D}$ is a $d$-adic system on $\Omega$. For $I\in \mathcal{D}$, let $\tilde{I}$ be the parent interval of $I$, and $I{(j)}$ the $j$-th subinterval of $I$, namely
\[(I_{n,k}){(j)}=I_{n+1,kd+j-1},\quad \forall n,k\in \mathbb{Z},1\le j\le d.\] 
Denote by $\mathcal{D}_n$ the collection of $d$-adic intervals of length $d^{-n}$ in $\mathcal{D}$. Given $I\in \mathcal{D}$, let $\mathcal{D}(I)$ be the collection of $d$-adic intervals contained in $I$, and $\mathcal{D}_n(I)$ the intersection of $\mathcal{D}_n$ and $\mathcal{D}(I)$.
For each $n\in \mathbb{Z}$, denote by $\mathcal{F}_n$ the $\sigma$-algebra generated by the $d$-adic intervals $I_{n,k}$, $\forall k\in \mathbb{Z}$. Denote by $\mathcal{F}$ the $\sigma$-algebra generated by all $d$-adic intervals for all $I_{n,k}$, $\forall n,k\in\mathbb{Z}$. 

Then $(\mathcal{F}_n)_{n\in \mathbb{Z}}$ is a filtration associated with the measure space $(\Omega, \mathcal{F}, \mu)$. Denote by $L^{\rm{loc}}_1(\Omega)$ the family of all locally integrable complex-valued functions $g$ on $\Omega$, that is, $g\in L_1(I_{n, k})$ for all $n, k\in \mathbb{Z}$. For a locally integrable complex-valued function $g\in L^{\rm{loc}}_1(\Omega)$, the sequence $(g_n)_{n\in\mathbb{Z}}$ is called a $d$-adic martingale, where
$$g_n= \mathbb{E}(g|\mathcal{F}_n)=\sum_{k=-\8}^{\8}\frac{\mathbbm{1}_{I_{n,k}}}{\mu(I_{n,k})}\int_{I_{n,k}}g\ d\mu.  $$
The martingale differences are defined as $d_n g=g_n-g_{n-1}$ for any $n\in\mathbb{Z}$. We also denote $g_n$ by $\mathbb{E}_n(g)$ ($n\in\mathbb{Z}$) as usual.

\begin{definition}\label{haar}
	Let $\omega=e^{\frac{2\pi \mathrm{i}}{d}}$ (here $\mathrm{i} $ is the imaginary number). For any $I=I_{n,k}\in \mathcal{D}$, define
	\[h_I^i=d^{n/2}\sum_{j=0}^{d-1}  \omega^{i{(j+1)}}\mathbbm{1}_{I_{n+1,kd+j}}, \quad \forall \ 1\leq i\leq d-1,\]
	and $h_I^0=d^{n/2}\mathbbm{1}_I$.
\end{definition}
Then $\{h_I^i\}_{I\in\mathcal{D},1\leq i\leq d-1}$ is an orthonormal basis on $L_2(\Omega)$ because $\forall g\in L_2(\Omega)$
\begin{equation}
	g=\sum_{k=-\8}^{\8} d_kg=\sum_{k=-\8}^{\8} \biggl(\sum_{|I|=d^{-k+1}}\sum_{i=1}^{d-1}h_I^i\langle h_I^i,g\rangle\biggr).
\end{equation}   
We call $\{h_I^i\}_{I\in\mathcal{D},1\leq i\leq d-1}$  the system of Haar wavelets. Note that for any $1\leq i, j\leq d-1$,
\begin{equation}\label{formula2.2}
	h_I^i\cdot h_I^j=\mu(I)^{-1/2} h_I^{\overline{i+j}},
\end{equation}
where $\overline{i+j}$ is the remainder in $[1, d]$ modulo $d$. The equality (\ref{formula2.2}) is vital in our proofs of Lemma \ref{TLambdab}.

\begin{example}\label{example1.4.3}
	A natural example of $d$-adic martingales is where $\Omega=\mathbb{R}$, $\mu=m$ and $I_{n,k}$ are defined as follows
	\[ I_{n,k}=[kd^{-n},(k+1)d^{-n}), \quad \forall n,k\in\mathbb{Z}.\]
\end{example}
\begin{example}	
	For $d=2^n$ and $\Omega=\mathbb{R}^n$, define
	\[ \mathcal{ D}_k=\{2^{-k}([0,1)^n+q):q\in\mathbb{Z}^n\}, \quad \forall k\in \mathbb{Z}.\]
	Then $\mathcal{D}=\{2^{-k}([0,1)^n+q):k\in\mathbb{Z},q\in\mathbb{Z}^n\} $ is the family of all $2^n$-adic intervals. Indeed, this is the dyadic filtration on $\mathbb{R}^n$.
\end{example}
\noindent$\mathit{\textbf{Convention}}$. In the sequel, for simplicity of notation, we will always assume that $\Omega=\mathbb{R}$ as this does not change the $d$-adic martingale structure. Denote also by $|I|$ the length $m(I)$ of $I\in\mathcal{ D}$.

\subsection{Martingale paraproducts and martingale Besov spaces}\label{sec2.3}
We need to employ martingale paraproducts to show Theorem \ref{corollary1.8} and Theorem \ref{corollary1.10}. At first, we introduce martingale paraproducts. Given a $d$-adic martingale $b=(b_k)_{k\in\mathbb{Z}}\in L_2(\mathbb{R})$, recall that the martingale paraproduct with symbol $b$ is defined as
\begin{equation*}
	\pi_b(f)=\sum_{k=-\infty}^{\infty}d_kb \cdot f_{k-1}, \quad \forall f=(f_k)_{k\in\mathbb{Z}}\in L_2(\mathbb{R}),
\end{equation*}
where $d_kb=b_k-b_{k-1}$ for any $k\in\mathbb{Z}$. Now we calculate $\pi_b$ in terms of Haar wavelets. Let $b\in  L_1^{\rm{loc}}(\mathbb{R})$. For $f\in L_2(\mathbb{R})$, we have 
\begin{equation}\label{calcu}
	\begin{aligned}
		\pi_b(f){}&=\sum_{k=-\infty}^{\infty}d_kb\cdot f_{k-1}\\&=\sum_{k=-\infty}^{\infty}\biggl(\sum_{|I|=d^{-k+1}}\sum_{i=1}^{d-1}\langle h_I^i,b\rangle h_I^i\biggr)\biggl(\sum_{|I|=d^{-k+1}}\biggl\langle \frac{\mathbbm{1}_I}{|I|},f\biggr\rangle\mathbbm{1}_I\biggr)\\&=\sum_{I\in \mathcal{D}}\sum_{i=1}^{d-1}\langle h_I^i,b\rangle \biggl\langle \frac{\mathbbm{1}_I}{|I|},f\biggr\rangle h^i_I,
	\end{aligned}
\end{equation}
The adjoint operator of $\pi_b$ is given by $\forall f\in L_2(\mathbb{R})$
\begin{equation}\label{pistar}
	\begin{aligned}
		\pi_b^*(f)&=\sum_{k\in\mathbb{Z}} \mathbb{E}_{k-1}(d_k\overline{b}\cdot d_kf) \\
		&=\sum_{I\in \mathcal{D}}\sum_{i=1}^{d-1}\overline{\langle h_I^i, b\rangle} \langle h_I^i,f\rangle\frac{\mathbbm{1}_I}{|I|}\\
		&=\sum_{I\in \mathcal{D}}\sum_{i=1}^{d-1}\langle b,h_I^i\rangle \langle h_I^i,f\rangle\frac{\mathbbm{1}_I}{|I|}. 
	\end{aligned}
\end{equation}

As mentioned before, we use Haar wavelets to define martingale Besov spaces for $d$-adic martingales.
\begin{definition}\label{mbs1}
	The martingale Besov space $\pmb{B}_p^d(\mathbb{R})$ $(0<p<\8)$ of $d$-adic martingales is the space of all locally integrable complex-valued functions $b$ such that
	\begin{equation}\label{e1.2}
		\|b\|_{\pmb{B}_p^d(\mathbb{R})}=\biggl(\sum_{I\in \mathcal{D}}\sum_{i=1}^{d-1}\frac{|\langle h_I^i,b\rangle|^p}{|I|^{p/2}}\biggr)^{1/p}<\infty.
	\end{equation}
\end{definition}

\begin{rem}\label{prop2.100}
	When $1\leq p<\8$, $\pmb{B}_p^d(\mathbb{R})$ is a Banach space. When $0<p<1$, $\pmb{B}_p^d(\mathbb{R})$ is a quasi-Banach space.
\end{rem}

The following two propositions give equivalent descriptions of $\pmb{B}_p^d(\mathbb{R})$.
\begin{proposition}\label{bbpdrmdp}
For any $0<p<\8$ and $b\in \pmb{B}_p^d(\mathbb{R})$,
	\[\|b\|_{\pmb{B}_p^d(\mathbb{R})}\approx_{d, p}\biggl(\sum_{k=-\infty}^\infty d^{k}\|d_kb\|_{L_p(\mathbb{R})}^p\biggr)^{1/p}.\]
\end{proposition}
\begin{proof}
	$\forall k\in\mathbb{Z}$, one has
	\begin{equation*}
		\begin{aligned}
			\|d_kb\|_{L_p(\mathbb{R})}^p{}&=\sum_{I\in\mathcal{D}_{k-1}}\bigg\|\sum_{i=1}^{d-1}h_I^i\cdot \langle h_I^i,b\rangle  \bigg\|_{L_p(\mathbb{R})}^p
			=d^{-k}\cdot\sum_{I\in\mathcal{D}_{k-1}}\frac{1}{|I|^{p/2}}\sum_{j=0}^{d-1}\bigg|\sum_{i=1}^{d-1}\omega^{i(j+1)} \langle h_I^i,b\rangle\bigg|^p.
		\end{aligned}
	\end{equation*}
 Let
	\begin{equation*}
		\xi_I=(\langle h_I^1,b\rangle, \langle h_I^2,b\rangle, \cdots, \langle h_I^{d-1},b\rangle)^\top,\quad \tilde{\xi}_I=(\langle h_I^1,b\rangle, \langle h_I^2,b\rangle, \cdots, \langle h_I^{d-1},b\rangle, 0)^\top
	\end{equation*}
	and
	\begin{equation*}
		B_I=\sum\limits_{i=1}^{d-1}\sum\limits_{j=0}^{d-1}\omega^{i(j+1)}e_{j+1,i},\quad \tilde{B}_I=\sum\limits_{i=1}^{d}\sum\limits_{j=0}^{d-1}\omega^{i(j+1)}e_{j+1,i}.
	\end{equation*}
	Then
	\begin{equation*}
		\begin{aligned}
			\sum_{j=0}^{d-1}\bigg|\sum_{i=1}^{d-1}\omega^{i(j+1)} \langle h_I^i,b\rangle\bigg|^p=\|B_I\xi_I\|^p_{\ell_p^{d-1}} =\|\tilde{B}_I\tilde{\xi}_I\|^p_{\ell_p^{d}}\approx_{d, p} \|\tilde{\xi}_I\|^p_{\ell_p^d}=\sum_{i=1}^{d-1}|\langle h_I^i,b\rangle|^p.
		\end{aligned}
	\end{equation*}
	It implies that
	\begin{equation*}
		\begin{aligned}
			\biggl(\sum_{k=-\infty}^\infty d^{k}\|d_kb\|_{L_p(\mathbb{R})}^p\biggr)^{1/p}\approx_{d, p}  \biggl(\sum_{I\in \mathcal{D}}\sum_{i=1}^{d-1}\biggl(\frac{|\langle h_I^i,b\rangle|}{|I|^{1/2}}\biggr)^p\biggr)^{1/p}=\|b\|_{\pmb{B}_p^d(\mathbb{R})}.
		\end{aligned}
	\end{equation*}
	This finishes the proof.
\end{proof}
	
\begin{prop}\label{equbbk}
	Let $1\le p<\infty$. Suppose that $b$ is a locally integrable complex-valued function and $b\in \pmb{B}_p^{d}(\mathbb{R})$. Then
	\begin{equation*}
			\|b\|_{\pmb{B}_p^d(\mathbb{R})} \approx_{d,p} \left(\sum_{k\in\mathbb{Z}} d^k \|b-b_k\|^p_{L_p(\mathbb{R})}\right)^{1/p}.
	\end{equation*}
\end{prop}

\begin{proof}
	On the one hand, since $b\in \pmb{B}_p^{d}(\mathbb{R})$, from Proposition \ref{bbpdrmdp} we know that
	$$\biggl(\sum_{k\in\mathbb{Z}} d^{k}\|d_kb\|_{L_p(\mathbb{R})}^p\biggr)^{1/p}<\infty.$$
	Then from the Minkowski inequality we obtain
	\begin{equation*}
		\begin{aligned}
			{}&\biggl(\sum_{k\in\mathbb{Z}}  \Big(\sum_{j=1}^\infty d^{-j/p}\big\| d^{(j+k)/p}\cdot d_{j+k}b\big\|_{L_p(\mathbb{R})}\Big)^p\biggr)^{1/p}\\
			&\le \sum_{j=1}^\infty  d^{-j/p} \biggl(\sum_{k\in\mathbb{Z}} \big\| d^{(j+k)/p}\cdot d_{j+k}b\big\|^p_{L_p(\mathbb{R})}\biggr)^{1/p}\\
			&=\sum_{j=1}^\infty  d^{-j/p} \biggl(\sum_{k\in\mathbb{Z}} d^k \|d_{k}b\|^p_{L_p(\mathbb{R})}\biggr)^{1/p}\\
			&=\frac{1}{d^{1/p}-1}\biggl(\sum_{k\in\mathbb{Z}} d^k \|d_{k}b\|^p_{L_p(\mathbb{R})}\biggr)^{1/p}<\infty.
		\end{aligned}
	\end{equation*}
	Besides, for any $k\in\mathbb{Z}$, notice that
	\begin{equation*}
		\begin{aligned}
			\|b-b_k\|_{L_p(\mathbb{R})}{}&\le  \sum_{j=k+1}^\infty \| d_jb\|_{L_p(\mathbb{R})}\\
			&=\sum_{j=1}^\infty\| d_{j+k}b\|_{L_p(\mathbb{R})}\\
			&=d^{-k/p}\cdot\sum_{j=1}^\infty d^{-j/p} \big\| d^{(j+k)/p}\cdot d_{j+k}b\big\|_{L_p(\mathbb{R})}.
		\end{aligned}
	\end{equation*}
	Thus 
	\begin{equation*}
		\begin{aligned}
			\biggl(\sum_{k\in\mathbb{Z}} d^k \|b-b_k\|^p_{L_p(\mathbb{R})}\biggr)^{1/p}{}&\le \biggl(\sum_{k\in\mathbb{Z}}  \Big(\sum_{j=1}^\infty d^{-j/p}\big\| d^{(j+k)/p}\cdot d_{j+k}b\big\|_{L_p(\mathbb{R})}\Big)^p\biggr)^{1/p}\\
			&\le \frac{1}{d^{1/p}-1}\biggl(\sum_{k\in\mathbb{Z}} d^k \|d_{k}b\|^p_{L_p(\mathbb{R})}\biggr)^{1/p}.
		\end{aligned}
	\end{equation*}
	On the other hand, notice that for any $k\in\mathbb{Z}$,
	\begin{equation*}
		\|d_kb\|_{L_p(\mathbb{R})}=\|\mathbb{E}_k(b-b_{k-1})\|_{L_p(\mathbb{R})}\le \|b-b_{k-1}\|_{L_p(\mathbb{R})},
	\end{equation*}
	which implies that
	\begin{equation*}
		\sum_{k\in\mathbb{Z}} d^k \|d_kb\|^p_{L_p(\mathbb{R})} \le d\cdot\sum_{k\in\mathbb{Z}} d^k \|b-b_k\|^p_{L_p(\mathbb{R})}.
	\end{equation*}
	The proof is completed.
\end{proof}

Chao and Peng described the Schatten class membership of martingale paraproducts by virtue of the martingale Besov spaces. They showed the following theorem (see {\cite[Theorem 3.1]{CP}}):
\begin{thm}\label{lem2.1}
	Let $0<p<\infty$ and $b$ be a locally integrable complex-valued function. Then $\pi_b\in S_p(L_2(\mathbb{R}))$ if and only if $b\in \pmb{B}_p^d(\mathbb{R})$.
\end{thm}

 In addition, we refer the interested reader to  \cite{CL1,CP} for more information on the the boundedness and compactness of martingale paraproducts for $d$-adic martingales.
 
 We would also like to mention that  \'{E}ric Ricard and the second named author gave a new proof of the Schatten class of martingale paraproducts based on Schur multipliers in \cite{ERZH2024}.

\subsection{The Schatten-Lorentz classes}
Given $0<p<\infty$ and $0<q\le\infty$, let $\ell_{p,q}(\Lambda)$ denote the usual Lorentz space on an index set $\Lambda$. $\ell_{p,q}$ is a quasi-Banach space and admits an equivalent norm $\| \cdot \|_{\tilde{\ell}_{p,q}}$ if $1<p<\infty$ and $1\le q\le\infty$. 
Now we introduce the definition of the weak-type and strong-type Schatten class. We define the Schatten-Lorentz class by virtue of the Lorentz space $\ell_{p,q}$. Let $H$ be a separable Hilbert space. A compact operator $A\in \mathcal{K}(H)$ has a singular value decomposition:
\begin{equation}\label{SVD}
	A(\xi)=\sum_{n}\lambda_n \langle e_n,\xi\rangle f_n, \quad\forall \xi\in H,
\end{equation}
where $\{\lambda_n\}_{n}\subset \mathbb{C}$, and $\{e_n\}_n$, $\{f_n\}_n$ are two orthonormal bases of $H$. For $0<p,q\le\infty$, the Schatten-Lorentz class $S_{p,q}(H)$ is defined as follows:
\begin{equation*}
	S_{p,q}(H)=\bigg\{A\in \mathcal{K}(H):\|A\|_{S_{p,q}(H)}:=\|\{\lambda_n\}_{n\in\mathbb{Z}}\|_{\ell_{p,q}}<\infty\bigg\}.
\end{equation*}
When $0<p<\infty$ and $q=\infty$, $S_{p,\infty}(H)$ is the weak-type Schatten class. When $0<p=q\leq \infty$, $S_{p}(H):=S_{p,p}(H)$ is namely the usual Schatten $p$-class, which is also called the strong-type Schatten class.

We introduce some properties for Schatten-Lorentz classes $S_{p,q}(H)$ which can be found in \cite{BS1977}. For any $0<p,q\le\infty$, $S_{p,q}(H)$ is a quasi-Banach space. When $1<p<\infty$ and $1\le q< \infty$, the dual space $(S_{p,q}(H))^*$ is isomorphic to $S_{p',q'}(H)$. In addition, if $1<p<\infty$ and $1\le q\le \infty$, $\tilde{S}_{p,q}(H)$ admits an equivalent norm $\|\cdot\|_{\tilde{S}_{p,q}(H)}$ induced by $\| \cdot \|_{\tilde{\ell}_{p,q}}$.

\begin{rem}\label{Spqquasi}
	If $1<p<\infty$ and $1\le q\le \infty$, then for any $A_i\in S_{p,q}(H)$,
	\begin{equation*}
		\begin{aligned}
			\Big\|\sum_{i}A_i\Big\|_{{S}_{p,q}}\le \Big\|\sum_{i}A_i\Big\|_{\tilde{S}_{p,q}}\le \sum_{i}\|A_i\|_{\tilde{S}_{p,q}}\lesssim_{p} \sum_{i}\|A_i\|_{{S}_{p,q}}.
		\end{aligned}
	\end{equation*}
\end{rem}
The following vital lemma is from \cite{RSe}.
\begin{lem}\label{WEAKnonNWOpre1}
	Let $0<p<\infty$ and $0<q\le\infty$. Let $\mathcal{D}$ be the $d$-adic system defined in Subsection \ref{commumar}. Assume that $\{e_{I}\}_{I\in\mathcal{D}}$ and $\{f_{I}\}_{I\in\mathcal{D}}$ are two function sequences in $L_2(\mathbb{R})$ satisfying $\mathrm{supp}e_{I}, \,\mathrm{supp}f_{I}\subseteq I$, and
	$\|e_{I}\|_{\infty},\,\|f_{I}\|_{\infty}\le |I|^{-\frac{1}{2}}$, or more generally, $\|e_{I}\|_{r},\,\|f_{I}\|_{r}\le |I|^{\frac{1}{r}-\frac{1}{2}}$ for some $r>2$. Define
	\begin{equation*}
		A(f)=\sum_{I\in \mathcal{D}} \lambda_{I} \langle e_{I},f\rangle f_{I},\quad \forall f\in L_2(\mathbb{R}),
	\end{equation*}
	where $\{\lambda_{I}\}_{I\in\mathcal{D}}\subset \mathbb{C}$.
	Then
	\begin{equation*}
		\begin{aligned}
			\|A\|_{S_{p,q}(L_2(\mathbb{R}))}\lesssim_{d,p,q}	\big\|\{\lambda_{I}\}_{I\in\mathcal{D}}\big\|_{\ell_{p,q}}.
		\end{aligned}
	\end{equation*}
In particular, when $0<p=q<\infty$,
	\begin{equation*}
		\begin{aligned}
			\|A\|_{S_{p}(L_2(\mathbb{R}))}\lesssim_{d,p}	\big\|\{\lambda_{I}\}_{I\in\mathcal{D}}\big\|_{\ell_{p}}.
		\end{aligned}
	\end{equation*}
\end{lem}

From Lemma \ref{WEAKnonNWOpre1}, one immediately has the following proposition involving the Schatten-Lorentz class of martingale paraproducts.
\begin{proposition}\label{WEAKlem2.1}
	Let $0<p<\infty$ and $0< q\le\infty$. Let $	\pi_b$ be the martingale paraproduct associated with the $d$-adic system. If $\{|I|^{-1/2}|\langle h_I^i,b\rangle|\}_{I\in\mathcal{D},1\le i\le d-1}\in \ell_{p,q}$, then $	\pi_b\in S_{p,q}(L_2(\mathbb{R}))$ and
	\begin{equation*}
		\|\pi_b\|_{S_{p,q}(L_2(\mathbb{R}))}\lesssim_{d,p,q}\Big\|\big\{|I|^{-1/2}|\langle h_I^i,b\rangle|\big\}_{I\in\mathcal{D},1\le i\le d-1}\Big\|_{\ell_{p,q}}.
	\end{equation*}
\end{proposition}
\begin{proof}
	Fix $1\le i\le d-1$. For any $I\in\mathcal{D}$, let 
	\begin{equation*}
		e_{I}=\frac{\mathbbm{1}_{I}}{|I|^{1/2}},\quad f_{I}=h_I^i,\quad \text{and}\quad \lambda_{I}=\frac{\langle h_I^i,b\rangle}{|I|^{1/2}}.
	\end{equation*}
	One checks that $\mathrm{supp}e_{I}, \,\mathrm{supp}f_{I}\subseteq I$ and
	$\|e_{I}\|_{\infty},\,\|f_{I}\|_{\infty}\le |I|^{-\frac{1}{2}}$. Hence from Remark \ref{Spqquasi} and Lemma \ref{WEAKnonNWOpre1}, we have
	\begin{equation*}
		\|\pi_b\|_{S_{p,q}(L_2(\mathbb{R}))}\lesssim_{d,p,q}	\Big\|\big\{|I|^{-1/2}|\langle h_I^i,b\rangle|\big\}_{I\in\mathcal{D},1\le i\le d-1}\Big\|_{\ell_{p,q}}.
	\end{equation*}
	The proof is finished.
\end{proof}

\begin{rem}
	Note that when $0<p=q<\infty$, Proposition \ref{WEAKlem2.1} directly implies the sufficiency of Theorem \ref{lem2.1}. But the necessity of Theorem \ref{lem2.1} is more delicate.
\end{rem}

The following two lemmas will be needed for the proofs
of the sufficiency parts of Theorem \ref{corollary1.8} and Theorem \ref{corollary1.10}. Let $(e_i)_{i\in \mathbb{N}}$ be the standard orthonormal basis on $\ell_2$. For any $A\in S_{p,q}(\ell_2)$, denote by $A_{i, j}$ the $(i, j)$-th entry defined as
$$ A_{i, j}= \la e_i, A e_j\ra.  $$
The following lemma is from \cite[Lemma 6.9]{WZ2024}.
\begin{lem}\label{lem5.1.9AAApre}
	Let $1< p<\infty$ and $1\le q\le\infty$. Suppose that $(A_\gamma)_{\gamma\in \Gamma}$ is a net of operators in $S_{p.q}(\ell_2)$. Let $A\in S_{p,q}(\ell_2)$. If for any $i, j\in \mathbb{N}$,
	$$ \lim_{\gamma} (A_\gamma)_{i, j}=A_{i, j},  $$
	then
	$$ \|A\|_{S_{p,q}(\ell_2)} \leq \sup_{\gamma}  \|A_\gamma\|_{S_{p,q}(\ell_2)}. $$
\end{lem}
\begin{proof}
	For any projection $\rho\in B(\ell_2)$ with finite rank, one has for any $B\in S_{p',q'}(\ell_2)$
	$$  \text{Tr}(\rho A\rho B) =\lim_\gamma  \text{Tr} (\rho A_\gamma \rho B).  $$
	By duality, this implies that
	$$  \|\rho A\rho \|_{S_{p,q}(\ell_2)} \leq \sup_\gamma  \|\rho A_\gamma\rho \|_{S_{p,q}(\ell_2)}.   $$
	Therefore, we have
	\begin{equation*}
		\begin{aligned}
			\|A\|_{S_{p,q}(\ell_2)}&=\sup_{\substack{\rho^2=\rho=\rho* \\ \text{finite rank}}}\|\rho A\rho \|_{S_{p,q}(\ell_2)} \\
			&\leq \sup_{\substack{\rho^2=\rho=\rho* \\ \text{finite rank}}}\sup_\gamma  \|\rho A_\gamma\rho \|_{S_{p,q}(\ell_2)}\\
			&\leq \sup_{\gamma}  \|A_\gamma\|_{S_{p,q}(\ell_2)},
		\end{aligned}
	\end{equation*}
	as desired.
\end{proof}
Then following lemma is an extension of \cite[Lemma 4.3]{PS}.
\begin{lem}\label{weakconver}
	Let $1< p<\infty$ and $1\le q\le\infty$. Suppose that $(T_\gamma)_{\gamma\in \Gamma}$ is a bounded net of operators in $B(L_p(\mathbb{R}^n))$ for each $p\in (1, \8)$. Assume that $(T_\gamma)_{\gamma\in \Gamma}$ converges to $T\in B(L_p(\mathbb{R}^n))$ with respect to the weak operator topology for each $p\in (1, \8)$. If $1< p<\8$ and $b\in L_p(\mathbb{R}^n)$, then
	$$  \| [T,M_b]\|_{S_{p,q}(L_2(\mathbb{R}^n))} \leq \sup_\gamma  \| [T_\gamma,M_b]\|_{S_{p,q}(L_2(\mathbb{R}^n))}. $$
	In particular, when $1<p=q<\infty$,
	$$  \| [T,M_b]\|_{S_{p}(L_2(\mathbb{R}^n))} \leq \sup_\gamma  \| [T_\gamma,M_b]\|_{S_{p}(L_2(\mathbb{R}^n))}. $$
\end{lem}
\begin{proof}
	For any finite cubes $I, J\subset \mathbb{R}^n $, one has
	$$  \lim_\gamma\big\la \mathbbm{1}_I, [T_\gamma,M_b]( \mathbbm{1}_J)\big\ra= \big\la \mathbbm{1}_I, [T,M_b]( \mathbbm{1}_J)\big\ra. $$
	Hence, the desired result follows from Lemma \ref{lem5.1.9AAApre}.
\end{proof}

\bigskip

\section{Some key lemmas}\label{key lemma}
This section focuses on the estimate of Schatten class of several dyadic operators associated with the dyadic system $\mathcal{D}$. Later, these estimates will be used in the proofs of Theorem \ref{corollary1.8} and Theorem \ref{corollary1.10}.
\begin{lem}\label{TLambdab}
	Assume that $0< p<\infty$ and $0< q\le\infty$. For any $d$-adic martingale $ f=(f_k)_{k\in\mathbb{Z} }\in L_2(\mathbb{R})$, we define
	\[\varLambda_b(f)=\sum_{k\in\mathbb{Z}}d_kb\cdot d_kf.\]
	If $\big\{|I|^{-1/2}|\langle h_I^i,b\rangle|\big\}_{I\in\mathcal{D},1\le i\le d-1}\in \ell_{p,q}$, then $\varLambda_b\in S_{p,q}(L_2(\mathbb{R}))$  and 
	$$ \|\varLambda_b\|_{S_{p,q}(L_2(\mathbb{R}))} \lesssim_{d,p,q} \Big\|\big\{|I|^{-1/2}|\langle h_I^i,b\rangle|\big\}_{I\in\mathcal{D},1\le i\le d-1}\Big\|_{\ell_{p,q}}. $$
	Particularly, if $b\in\pmb{B}_p^d(\mathbb{R})$, then $\varLambda_b\in S_p(L_2(\mathbb{R}))$ and 
	$$ \|\varLambda_b\|_{S_p(L_2(\mathbb{R}))} \lesssim_{d, p} \|b\|_{\pmb{B}_p^d(\mathbb{R})}. $$
\end{lem}	

\begin{proof}
    We write ${\varLambda}_b$ as follows:
	\begin{equation}\label{WEAKLambdasum}
		\begin{aligned}
			{\varLambda}_b(f){}
			&=\sum_{k\in\mathbb{Z}}\sum_{I\in\mathcal{D}_{k-1}}\sum_{i=1}^{d-1}\langle h_I^i,b\rangle h_I^i\cdot\sum_{J\in\mathcal{D}_{k-1}}\sum_{j=1}^{d-1}\langle h_J^j,f\rangle h_J^j\\
			&=\sum_{i,j=1}^{d-1}\sum_{k\in\mathbb{Z}}\sum_{I\in\mathcal{D}_{k-1}}\langle h_I^i,b\rangle \langle h_I^j,f\rangle h_I^i h_I^j=:\sum_{i,j=1}^{d-1}A_{ij}.
		\end{aligned}
	\end{equation}
	Now fix $1\le i,j\le d-1$. In terms of $A_{st}$, for any given $I\in\mathcal{D}$, let 
	\begin{equation*}
		e_{I}=h_I^j,\quad  f_{I}=|I|^{1/2}h_I^i h_I^j\quad \text{and} \quad\lambda_{I}=\frac{\langle h_I^i,b\rangle}{|I|^{1/2}}.
	\end{equation*}
	Note that $\mathrm{supp}e_{I}, \,\mathrm{supp}f_{I}\subseteq I$ and
	$\|e_{I}\|_{\infty},\,\|f_{I}\|_{\infty}\lesssim |I|^{-\frac{1}{2}}$. 
	Then from Lemma \ref{WEAKnonNWOpre1}, we have
	\begin{equation*}
		\begin{aligned}
			\|A_{ij}\|_{S_{p,q}(L_2(\mathbb{R}))}{}&\lesssim_{d,p,q}	\big\|\{\lambda_{I}\}_{I\in\mathcal{D}}\big\|_{\ell_{p,q}}
			=\Big\|\big\{|I|^{-1/2}|\langle h_I^i,b\rangle|\big\}_{I\in\mathcal{D}}\Big\|_{\ell_{p,q}}.
		\end{aligned}
	\end{equation*}
	Hence from Remark \ref{Spqquasi},
	$$ \|\varLambda_b\|_{S_{p,q}(L_2(\mathbb{R}))} \lesssim_{d,p,q} \Big\|\big\{|I|^{-1/2}|\langle h_I^i,b\rangle|\big\}_{I\in\mathcal{D},1\le i\le d-1}\Big\|_{\ell_{p,q}}. $$
	Particularly, when $0<p=q<\infty$,
		$$ \|\varLambda_b\|_{S_p(L_2(\mathbb{R}))} \lesssim_{d, p}\Big\|\big\{|I|^{-1/2}|\langle h_I^i,b\rangle|\big\}_{I\in\mathcal{D},1\le i\le d-1}\Big\|_{\ell_{p,p}}=\|b\|_{\pmb{B}_p^d(\mathbb{R})}. $$
\end{proof}
We define the $d$-adic martingale $BMO$ space as follows:
\begin{defn}
	The martingale $BMO$ space of $d$-adic martingale denoted by $BMO^d(\mathbb{R})$ is the space of all locally integrable complex-valued functions $b$ such that
	\begin{equation}\label{BMOd}
		\|b\|_{BMO^d(\mathbb{R})}= \sup _{n\in\mathbb{Z}}\ \bigg\|\mathbb{E}_n\bigg(\sum_{k=n+1}^{\8}|d_k b|^2\bigg)\bigg\|_\8^{1/2} <\infty. 
	\end{equation} 
\end{defn}

In order to prove the sufficiency parts of Theorem \ref{corollary1.8} and Theorem \ref{corollary1.10}, we give the following proposition, which concerns the Schatten-Lorentz class of commutators involving martingale paraproducts and the pointwise multiplication operator $M_b$.
\begin{proposition}\label{T0est}
	Let $0< p<\infty$ and $0< q\le\infty$. Suppose that $a\in BMO^d(\mathbb{R})$.
	If $\big\{|I|^{-1/2}|\langle h_I^i,b\rangle|\big\}_{I\in\mathcal{D},1\le i\le d-1}\in\ell_{p,q}$, then $[\pi_a,M_b]$ and $[\pi^*_a,M_b]$ both belong to $S_{p,q}(L_2(\mathbb{R}))$. Moreover,
	\begin{equation*}
		\|[\pi_a,M_b]\|_{S_{p,q}(L_2(\mathbb{R}))}\lesssim_{d,p,q} \|a\|_{BMO^{d}(\mathbb{R})} \Big\|\big\{|I|^{-1/2}|\langle h_I^i,b\rangle|\big\}_{I\in\mathcal{D},1\le i\le d-1}\Big\|_{\ell_{p,q}}
	\end{equation*}
	and
	\begin{equation*}
		\|[\pi_a^*,M_b]\|_{S_{p,q}(L_2(\mathbb{R}))}\lesssim_{d,p,q} \|a\|_{BMO^{d}(\mathbb{R})} \Big\|\big\{|I|^{-1/2}|\langle h_I^i,b\rangle|\big\}_{I\in\mathcal{D},1\le i\le d-1}\Big\|_{\ell_{p,q}}.
	\end{equation*}
Particularly, if $b\in \pmb{B}_p^{d}(\mathbb{R})$, then $[\pi_a,M_b]$ and $[\pi^*_a,M_b]$ both belong to $S_p(L_2(\mathbb{R}))$. Moreover,
\begin{equation*}
	\|[\pi_a,M_b]\|_{S_p(L_2(\mathbb{R}))}\lesssim_{d,p} \|a\|_{BMO^d(\mathbb{R})} \|b\|_{\pmb{B}_p^{d}(\mathbb{R})}
\end{equation*}
and
\begin{equation*}
	\|[\pi_a^*,M_b]\|_{S_p(L_2(\mathbb{R}))}\lesssim_{d,p} \|a\|_{BMO^d(\mathbb{R})} \|b\|_{\pmb{B}_p^{d}(\mathbb{R})}.
\end{equation*}
\end{proposition} 
\begin{proof}
	Without loss of generality, we suppose that $\|a\|_{BMO^{d}(\mathbb{R})}=1$. Let
	\begin{equation}\label{R_b}
		R_b(f)=\sum_{k\in\mathbb{Z}}b_{k-1}\cdot d_kf, \quad \forall f\in L_2(\mathbb{R}).
	\end{equation}
	Note that for $b, f\in L_2(\mathbb{R})$,  $M_b(f)=\pi_b(f)+\varLambda_b(f)+R_b(f)$. Thus
	$$ 	[\pi_{{a}},M_b]=[\pi_{{a}},\pi_b]+[\pi_{{a}},\varLambda_b]+[\pi_{{a}},R_b].$$
	For any $f\in L_2(\mathbb{R})$,
	\begin{equation*}
		\begin{aligned}
			[\pi_{a}, R_b](f){}&=\pi_{a} (R_b(f))-R_b (\pi_{a}(f))\\
			&=\sum_{k\in\mathbb{Z}} d_ka\cdot\mathbb{E}_{k-1} \biggl(\sum_{j\in\mathbb{Z}} b_{j-1}\cdot d_jf \biggr)- \sum_{k\in\mathbb{Z}} b_{k-1} \cdot d_k \biggl(\sum_{j\in\mathbb{Z}} d_ja\cdot f_{j-1}\biggr)\\
			&=\sum_{k\in\mathbb{Z}} d_ka\cdot \biggl(\sum_{j\le k-1}b_{j-1}\cdot d_jf\biggr)-\sum_{k\in\mathbb{Z}} b_{k-1}\cdot d_ka\cdot f_{k-1}\\
			&=\sum_{k\in\mathbb{Z}} d_ka\cdot \biggl(\sum_{j\le k-1}b_{j-1}\cdot d_jf- b_{k-1}\cdot f_{k-1}\biggr)\\
			&=-\sum_{k\in\mathbb{Z}} d_ka\cdot \biggl( \sum_{j\le k-1}d_jb\cdot d_jf\biggr)-\sum_{k\in\mathbb{Z}} d_ka\cdot \biggl(\sum_{j\le k-1}d_jb\cdot f_{j-1}\biggr)\\
			&=-\sum_{k\in\mathbb{Z}} d_ka\cdot \biggl( \sum_{j\le k-1}d_jb\cdot d_jf\biggr)-\pi_{a}(\pi_b(f))\\
			&=: -\varPsi_{a,b}(f)-\pi_{a}(\pi_b(f)).
		\end{aligned}
	\end{equation*}
	Thus
	\begin{equation}\label{piarb}
		[\pi_{a}, R_b]=-\varPsi_{a,b}-\pi_{a}\pi_b.
	\end{equation}
	By \eqref{piarb}, one has
	$$ 	[\pi_{{a}},M_b]=[\pi_{{a}},\pi_b]+[\pi_{{a}},\varLambda_b]-\varPsi_{a,b}-\pi_{a}\pi_b=\pi_b\pi_{a}+[\pi_{{a}},\varLambda_b]-\varPsi_{a,b}.$$
	From Proposition \ref{WEAKlem2.1} and Lemma \ref{TLambdab}, we know that
	\begin{equation*}
		\begin{aligned}
			\|\pi_b\|_{S_{p,q}(L_2(\mathbb{R}))}\lesssim_{d,p,q}\Big\|\big\{|I|^{-1/2}|\langle h_I^i,b\rangle|\big\}_{I\in\mathcal{D},1\le i\le d-1}\Big\|_{\ell_{p,q}}
		\end{aligned}
	\end{equation*}
and
\begin{equation*}
	\|\varLambda_b\|_{S_{p,q}(L_2(\mathbb{R}))}\lesssim_{d,p,q}\Big\|\big\{|I|^{-1/2}|\langle h_I^i,b\rangle|\big\}_{I\in\mathcal{D},1\le i\le d-1}\Big\|_{\ell_{p,q}}.
\end{equation*}
	Since $\pi_a$ is bounded on $L_2(\mathbb{R})$ (see \cite{CL1}), one has
	\begin{equation*}
		\|\pi_b\pi_a\|_{S_{p,q}(L_2(\mathbb{R}))}\lesssim_{d, p,q} \|a\|_{BMO^{d}(\mathbb{R})}\Big\|\big\{|I|^{-1/2}|\langle h_I^i,b\rangle|\big\}_{I\in\mathcal{D},1\le i\le d-1}\Big\|_{\ell_{p,q}},
	\end{equation*}
	and 
	\begin{equation*}
		\|[\pi_{{a}},\varLambda_b]\|_{S_{p,q}(L_2(\mathbb{R}))}\lesssim_{d, p,q} \|a\|_{BMO^{d}(\mathbb{R})}\Big\|\big\{|I|^{-1/2}|\langle h_I^i,b\rangle|\big\}_{I\in\mathcal{D},1\le i\le d-1}\Big\|_{\ell_{p,q}}.
	\end{equation*}
	Besides, note that
	\begin{equation*}
		\begin{aligned}
			\varPsi_{a,b}(f){}
			&=\sum_{k\in\mathbb{Z}}\sum_{I\in\mathcal{D}_{k-1}}\sum_{l=1}^{d-1}\langle h_I^l,a\rangle h_I^l \cdot \biggl(\sum_{m\le k-1}\sum_{L\in \mathcal{D}_{m-1}}\sum_{i=1}^{d-1}\langle h_L^i,b\rangle h_L^i \cdot \sum_{Q\in \mathcal{D}_{m-1}}\sum_{j=1}^{d-1}\langle h_Q^j,f\rangle h_Q^j\biggr)\\
			&=\sum_{k\in\mathbb{Z}}\sum_{I\in\mathcal{D}_{k-1}}\sum_{l=1}^{d-1}\langle h_I^l,a\rangle h_I^l \cdot \biggl(\sum_{m\le k-1}\sum_{L\in \mathcal{D}_{m-1}}\sum_{i,j=1}^{d-1}\langle h_L^i,b\rangle \langle h_L^j,f\rangle h_L^ih_L^j\biggr)\\
			&=\sum_{i,j=1}^{d-1}\sum_{s=1}^{d}\sum_{L\in\mathcal{D}}\sum_{I\subseteq L(s)}\sum_{l}\langle h_I^l,a\rangle h_I^l\cdot \langle h_L^i,b\rangle \langle h_L^j,f\rangle   h_L^i h_L^j\\
			&=\sum_{i,j=1}^{d-1}\sum_{s=1}^{d}\sum_{L\in\mathcal{D}}|L|^{-1/2}\langle h_L^i,b\rangle \cdot \langle h_L^j,f\rangle \bigg(\sum_{I\subseteq L(s)}\sum_{l}\langle h_I^l,a\rangle h_I^l h_L^{\overline{i+j}}\bigg)\\
			&=:\sum_{i,j=1}^{d-1}\sum_{s=1}^{d}A_{i,j,s}.
		\end{aligned}
	\end{equation*}
	Now fix $1\le i,j\le d-1$ and $1\le s\le d$. In terms of $A_{i,j,s}$, for any given $L\in\mathcal{D}$, let  $L\in\mathcal{D}_t$, where $t\in\mathbb{Z}$. Define 
	\begin{equation*}
		e_{L}=h_L^j,\quad  {f}_{L}=\sum_{I\subseteq L(s)}\sum_{l}\langle h_I^l,a\rangle h_I^l h_L^{\overline{i+j}} \quad\text{and} \quad\lambda_{L}=|L|^{-1/2}\langle h_L^i,b\rangle.
	\end{equation*}
	Then
	\begin{equation*}
		A_{i,j,s}=\sum_{L\in\mathcal{D}}\lambda_{L}\langle e_L,f\rangle f_L.
	\end{equation*}
	Note that $\mathrm{supp}e_{L}, \,\mathrm{supp}f_{L}\subseteq L$ and
	$\|e_{L}\|_{\infty}\lesssim |L|^{-\frac{1}{2}}$. Besides, note that $h_L^j$ is constant on $L(s)$ and $|h_L^j(x)|=|L|^{-1/2},\forall x\in L(s)$. Then
	\begin{equation*}
		\begin{aligned}
			{f}_{L}{}&=h_L^{\overline{i+j}}\cdot\sum_{I\subseteq L(s)}\sum_{l}\langle h_I^l,a\rangle h_I^l 
			=h_L^{\overline{i+j}}\cdot \sum_{k\ge t+2}d_ka\cdot\mathbbm{1}_{L(s)}=h_L^{\overline{i+j}}\cdot(a-a_{t+1})\mathbbm{1}_{L(s)}.
		\end{aligned}
	\end{equation*}
	By the John-Nirenberg inequality,
	\begin{equation*}
		\|a\|_{BMO^{d}(\mathbb{R})}\approx \sup_{n\in\mathbb{Z}}\big\|\mathbb{E}_n|a-a_n|^4\big\|_\infty^{1/4}=\sup_{n\in\mathbb{Z}}\sup_{J\in\mathcal{D}_n}\frac{1}{|J|^{1/4}}\bigg(\int_J|a(x)-a_n(x)|^4dx\bigg)^{1/4}.
	\end{equation*}
	It implies that
	\begin{equation*}
		\begin{aligned}
			\|{f}_L\|_{L_4(\mathbb{R})}&=\frac{1}{|L|^{1/2}}\bigg(\int_{L(s)}|a(x)-a_{t+1}(x)|^4dx\bigg)^{1/4}{}\\
			&\lesssim \frac{|L(s)|^{1/4}}{|L|^{1/2}}\|a\|_{BMO^{d}(\mathbb{R})}\\
			&=\frac{|L|^{-1/4}}{2^{n/4}}\|a\|_{BMO^{d}(\mathbb{R})}.
		\end{aligned}
	\end{equation*}
	Then from Lemma \ref{WEAKnonNWOpre1}, we have
	\begin{equation*}
		\|A_{i,j,s}\|_{S_{p,q}(L_2(\mathbb{R}))}\lesssim_{d,p,q}	\|a\|_{BMO^{d}(\mathbb{R})}\Big\|\big\{|I|^{-1/2}|\langle h_I^i,b\rangle|\big\}_{I\in\mathcal{D}}\Big\|_{\ell_{p,q}}.
	\end{equation*}
  Thus from Remark \ref{Spqquasi},
	\begin{equation*}
		\begin{aligned}
			\|\varPsi_{a,b}\|_{S_{p,q}(L_2(\mathbb{R}))}
			\lesssim_{d,p,q}\|a\|_{BMO^{d}(\mathbb{R})}\Big\|\big\{|I|^{-1/2}|\langle h_I^i,b\rangle|\big\}_{I\in\mathcal{D},1\le i\le d-1}\Big\|_{\ell_{p,q}}.
		\end{aligned}
	\end{equation*}
	Therefore, from Remark \ref{Spqquasi} we deduce that
	\begin{equation*}
		\begin{aligned}
			\|[\pi_a,M_b]\|_{S_{p,q}(L_2(\mathbb{R}))}{}&\lesssim \|\pi_b\pi_{a}\|_{S_{p,q}(L_2(\mathbb{R}))}+\|[\pi_{{a}},\varLambda_b]\|_{S_{p,q}(L_2(\mathbb{R}))}
			+\|\varPsi_{a,b}\|_{S_{p,q}(L_2(\mathbb{R}))}\\
			&\lesssim_{d,p,q}  \|a\|_{BMO^d(\mathbb{R})}\Big\|\big\{|I|^{-1/2}|\langle h_I^i,b\rangle|\big\}_{I\in\mathcal{D},1\le i\le d-1}\Big\|_{\ell_{p,q}}.
		\end{aligned}
	\end{equation*}
	It is easy to verify that 
	\begin{equation*}
		[\pi_a^*,M_b]^*=-[\pi_a,M_{\overline{b}}].
	\end{equation*}
Hence this implies
	\begin{align*}
		\|[\pi_a^*,M_b]\|_{S_{p,q}(L_2(\mathbb{R}))}&=\|[\pi_a,M_{\overline{b}}]\|_{S_{p,q}(L_2(\mathbb{R}))}\\
	&\lesssim_{d,p,q}  \|a\|_{BMO^d(\mathbb{R})}\Big\|\big\{|I|^{-1/2}|\langle h_I^i,b\rangle|\big\}_{I\in\mathcal{D},1\le i\le d-1}\Big\|_{\ell_{p,q}}.
	\end{align*}
Particularly, when $0<p=q<\infty$, we obtain
\begin{equation*}
	\|[\pi_a,M_b]\|_{S_p(L_2(\mathbb{R}))}\lesssim_{d,p} \|a\|_{BMO^d(\mathbb{R})} \|b\|_{\pmb{B}_p^{d}(\mathbb{R})},
\end{equation*}
and
\begin{equation*}
	\|[\pi_a^*,M_b]\|_{S_p(L_2(\mathbb{R}))}\lesssim_{d,p} \|a\|_{BMO^d(\mathbb{R})} \|b\|_{\pmb{B}_p^{d}(\mathbb{R})}.
\end{equation*}
This completes the proof.
\end{proof}

\bigskip

\section{Hyt\"{o}nen's dyadic representation and weak Besov spaces}\label{repre}
We will introduce the key ingredient: the dyadic representation of singular integral operators established by Hyt\"{o}nen in \cite{TH1} and \cite{TH2}. This representation enables the reduction to the $d$-adic martingale setting. Then we introduce the weak Besov spaces by virtues of several dyadic systems. Note that the weak Besov spaces will be used to describe the Schatten-Lorentz class of commutators.
\subsection{Hyt\"{o}nen's dyadic representation}
Now we introduce the dyadic system on $\mathbb{R}^n$. Recall that the standard system of dyadic cubes is 
\[\mathcal{D}^0=\big\{2^{-k}([0,1)^n+q):k\in\mathbb{Z},q\in\mathbb{Z}^n\big\}.\]
Let $\mathcal{ D}^0_k=\big\{2^{-k}([0,1)^n+q):q\in\mathbb{Z}^n\big\}$ for any $k\in \mathbb{Z}$. Define $\ell(I)=2^{-k}$ and $|I|= 2^{-nk}$ if $I\in \mathcal{ D}^0_k$. Let $\omega=(\omega_j)_{j\in\mathbb{Z}}\in(\{0,1\}^n)^\mathbb{Z}$ and define
\begin{equation}\label{Idot}
	I\dot{+}\omega=I+\sum_{j:2^{-j}<\ell(I)}2^{-j}\omega_j.
\end{equation}
Note that if $\ell(I)=2^{-k}$, $I\cap I\dot{+}\omega\neq \emptyset$ unless a coordinate of $\sum_{j:2^{-j}<\ell(I)}2^{-j}\omega_j$ is exactly $2^{-k}$. Then set
\[\mathcal{D}^\omega=\big\{I\dot{+}\omega:I\in\mathcal{D}^0\big\},\]
which is obtained by translating the standard system. Indeed, in particular, if $\omega=(\omega_j)_{j\in\mathbb{Z}}\in(\{0,1\}^n)^\mathbb{Z}$ such that $\exists j_0\in \mathbb{Z}$, $\forall j\leq j_0$, $\omega_j=0$, then
$$  \mathcal{D}^\omega=\big\{I+\sum_{j}2^{-j}\omega_j:I\in\mathcal{D}^0\big\}. $$
Now for any $k\in\mathbb{Z}$, $\mathcal{ D}_k^\omega=\big\{I\dot{+}\omega:I\in\mathcal{D}_k^0\big\}$ is the family of all dyadic cubes with volume $2^{-nk}$. The following lemma implies that $\mathcal{D}^\omega$ is still a dyadic system.
\begin{lem}
	For any $\omega=(\omega_j)_{j\in\mathbb{Z}}\in(\{0,1\}^n)^\mathbb{Z}$, $\mathcal{D}^\omega$ on $\mathbb{R}^n$ is a dyadic system. More precisely, for any $k\in\mathbb{Z}$,
	\begin{equation*}
		\mathcal{ D}_k^\omega=\mathcal{D}^0_k+\sum\limits_{j>k}2^{-j}\omega_j.
	\end{equation*}
\end{lem}
\begin{proof}
	By direct calculation,
	\begin{align*}
		\mathcal{ D}_k^\omega&=\big\{I\dot{+}\omega:I\in\mathcal{D}_k^0\big\}\\
		&=\left\{I+\sum\limits_{j>k}2^{-j}\omega_j:I\in\mathcal{D}_k^0\right\}.
	\end{align*}
	Note that 
	$$	\mathcal{ D}_k^\omega= \mathcal{D}^0_k+\sum\limits_{j>k}2^{-j}\omega_j=\mathcal{D}^0_k+\sum\limits_{j\geq k}2^{-j}\omega_j=(\mathcal{D}^0+\sum\limits_{j\geq k}2^{-j}\omega_j)_k, $$
	and
	$$ \mathcal{ D}_{k-1}^\omega=\mathcal{D}^0_{k-1}+\sum\limits_{j>k-1}2^{-j}\omega_j=\mathcal{D}^0_{k-1}+\sum\limits_{j\geq k}2^{-j}\omega_j=(\mathcal{D}^0+\sum\limits_{j\geq k}2^{-j}\omega_j)_{k-1}, $$
	where $(\mathcal{D}^0+\sum\limits_{j\geq k}2^{-j}\omega_j)_k$ and $(\mathcal{D}^0+\sum\limits_{j\geq k}2^{-j}\omega_j)_{k-1}$ are the collections of cubes in $\mathcal{D}^0+\sum\limits_{j\geq k}2^{-j}\omega_j$ with the corresponding volumes $2^{-nk}$ and $2^{-n(k-1)}$ respectively. This implies that every cube in $ \mathcal{ D}_{k-1}^\omega$ is a union of $2^n$ disjoint cubes in $ \mathcal{ D}_{k}^\omega$. Hence, $\mathcal{D}^\omega$ is a dyadic system.
\end{proof}

We refer the reader to \cite{TH2008} for more details on $\mathcal{D}^\omega$. For any $I\in\mathcal{ D}^\omega$, let $\mathcal{D}^\omega(I)$ be the collection of cubes in $\mathcal{D}^\omega$ contained in $I$, and $\mathcal{D}_k^\omega(I)$ the intersection of $\mathcal{D}_k^\omega$ and $\mathcal{D}^\omega(I)$ for any $k\in\mathbb{Z}$. In addition, we assign to the parameter set $(\{0,1\}^n)^\mathbb{Z}$ the natural probability measure, that is the infinite tensor product of the uniform probability measure $\sum\limits_{\omega\in \{0,1\}^n}2^{-n}\delta_{\omega}$. Here $\delta_{\omega}$ is the Dirac measure. Denote by $\mathbb{E}_\omega$ the expectation on $(\{0,1\}^n)^\mathbb{Z}$.

For convenience, denote the set $\{0,1\}^n\backslash \{0\}$ by $\{0,1\}^n_0$. Now for any given cube $I=I_1\times\cdots\times I_n\in\mathcal{D}^\omega$, let $H_{I_i}^0=|I_i|^{-1/2}\mathbbm{1}_{I_i}$ and $H_{I_i}^1=|I_i|^{-1/2}(\mathbbm{1}_{I_{i\ell}}-\mathbbm{1}_{I_{ir}})$, where $\mathbbm{1}_{I_{i\ell}}$ and $\mathbbm{1}_{I_{ir}}$ are the left and right halves of $I_i$ for $1\le i\le n$. For any $ \eta\in\{0,1\}^n_0$, we denote by $H_I^\eta$ the function on the cube $I=I_1\times\cdots\times I_n$ which is the product of the one-variable functions:
\[H_I^\eta(x)=H_{I_1\times\cdots\times I_n}^{(\eta_1,\cdots,\eta_n)}(x_1,\cdots,x_n)=\prod_{i=1}^nH_{I_i}^{\eta_i}(x_i).\]
Hence $\{H_I^\eta\}_{I\in\mathcal{D}^\omega,\eta\in\{0,1\}^n_0}$ form an orthonormal basis of $L_2(\mathbb{R}^n)$.

\

Let $i,j\in \mathbb{N}\cup\{0\}$. For a fixed dyadic system $\mathcal{D}^\omega$, the dyadic shift with parameters $i,j$ is an operator of the form
\[S_\omega^{ij}(f)=\sum_{K\in\mathcal{D}^\omega}A_K^{ij}(f),\quad\quad A_K^{ij}(f)=\sum_{\substack{I,J\in\mathcal{D}^\omega;I,J\subseteq K\\ \ell(I)=2^{-i}\ell(K)\\\ell(J)=2^{-j}\ell(K)}}\sum_{\xi,\eta\in\{0,1\}^n_0}a_{IJK}^{\xi\eta}\langle H_I^\xi,f\rangle H_J^\eta,\]
with coefficients $a_{IJK}^{\xi\eta}$ satisfying
\begin{equation}\label{aijkxieta}
	|a_{IJK}^{\xi\eta}|\le\frac{\sqrt{|I||J|}}{|K|}.
\end{equation}
From the definition of $S_\omega^{ij}$, we see that $I$ and $J$ are the $i$-th and $j$-th generation of $K$ respectively. So when $i$ or $j$ is very large, the dyadic shift $S_\omega^{ij}$ has very high complexity. This is the main difficulty when dealing with $S_\omega^{ij}$. We  also have the following properties:
\begin{enumerate}
	\item The map $S_\omega^{ij}:L_2(\mathbb{R}^n)\to L_2(\mathbb{R}^n)$  is bounded with norm at most one,
	\item $S_\omega^{ij}$ is of weak type $(1,1)$ with norm $O(i)$,
	\item For $1<p<\8$
	$$ \|S_\omega^{ij}\|_{L_p(\mathbb{R}^n)\rightarrow L_p(\mathbb{R}^n)}\lesssim_{n, p} i+j.   $$
\end{enumerate} 
The reader is referred to \cite{TH1} and \cite{TH2} for more information.

Recall that in this paper, $T:L_2(\mathbb{R}^n)\to L_2(\mathbb{R}^n)$ is always assumed to be bounded and its kernel satisfies the estimates (\ref{standard}). The following is the dyadic representation of singular integral operators discovered by Hyt\"{o}nen in \cite{TH1} and \cite{TH2} (see {\cite[Theorem 3.3]{TH2}}).
\begin{thm}\label{CZdec}
	Let $T$ be a bounded singular integral operator.
	Then $T$ has a dyadic expansion, say for $f,g\in L_2(\mathbb{R}^n)$,
	\begin{equation}\label{T}
		\begin{aligned}
			\langle g,T(f)\rangle{}&=C_1(T)\mathbb{E}_\omega\sum_{i,j=0\atop \max\{i,j\}>0}^\infty \tau(i,j)\langle g,S_\omega^{ij}(f)\rangle+C_2(T)\mathbb{E}_\omega\langle g,S_\omega^{00}(f)\rangle\\&\quad+\mathbb{E}_\omega \langle g,\pi_{T(1)}^\omega (f)\rangle+\mathbb{E}_\omega \langle g,(\pi_{T^*(1)}^\omega)^* (f)\rangle,
		\end{aligned}
	\end{equation}
	where $S_\omega^{ij}$ is the dyadic shift of parameters $(i,j)$ on the dyadic system $\mathcal{D}^\omega$, $\pi_b^\omega$ is the dyadic martingale paraproduct on the dyadic system $\mathcal{D}^\omega$ associated with the $BMO$-function $b\in\{T(1),T^*(1)\}$, $C_1(T)$, $C_2(T)$ are positive constants depending on $T$, and $\tau(i,j)$ satisfies 
	\begin{equation*}
		0\leq \tau(i,j)\lesssim (1+\max\{i,j\})^{2(n+\alpha)}2^{-\alpha \max\{i,j\}}.
	\end{equation*}
\end{thm}


The dyadic system $\mathcal{ D}^\omega$ on $\mathbb{R}^n$ can be regarded as the $2^n$-adic system by our definition of $d$-adic martingales (see Subsection \ref{commumar}). So we can define the martingale Besov space $\pmb{B}_p^{\omega, 2^n}(\mathbb{R}^n)$ on $\mathbb{R}^n$ by virtue of $H_I^\eta$ similarly as in Definition \ref{mbs1}. More precisely, $\pmb{B}_p^{\omega, 2^n}(\mathbb{R}^n)$ $(0<p<\8)$ associated with the dyadic martingale $\mathcal{ D}^\omega$ on $\mathbb{R}^n$ is the space of all locally integrable complex-valued functions $b$ such that
\begin{equation}\label{Bpw2nrnm}
	\|b\|_{\pmb{B}_p^{\omega, 2^n}(\mathbb{R}^n)}=\biggl(\sum_{I\in \mathcal{D}^\omega}\sum_{\eta\in\{0,1\}^n_0}\frac{|\langle H_I^\eta,b\rangle|^p}{|I|^{p/2}}\biggr)^{1/p}<\infty.
\end{equation}

In addition, Theorem \ref{lem2.1}, Lemma \ref{TLambdab}, Proposition \ref{T0est} also hold for the dyadic system on $\mathbb{R}^n$ with $d=2^n$ since our proof only depends on the martingale structure, and does not depend on the dimension of the Euclidean space.

\subsection{Weak Besov spaces}
Before introducing the weak Besov spaces, we give the following lemma which is explicitly stated in \cite{CJM2013}. In fact, it is due to Mei \cite{Mei4} in the case of $\mathbb{T}^n$; Mei made a remark \cite[Remark 7]{Mei4} for $\mathbb{R}^n$. However, Conde \cite{CJM2013} noted that Mei's remark is not correct, and finally found the following right substitution.
\begin{lem}\label{Domegan1}
	There exist $n+1$ dyadic systems $\mathcal{D}^{\omega(1)},\mathcal{D}^{\omega(2)},\cdots,\mathcal{D}^{\omega(n+1)}$ in $\mathbb{R}^n$, where $\omega(i)\in(\{0,1\}^n)^\mathbb{Z}$ for all $ 1\le i\le n+1$, such that for any cube $B\subset\mathbb{R}^n$, there exists some cube $Q\in \bigcup\limits_{i=1}^{n+1} \mathcal{D}^{\omega(i)}$ satisfying
	\begin{equation*}
		B\subseteq Q\subseteq c_n B,
	\end{equation*}
	where $c_n$ only depends on $n$. Moreover, $n+1$ is the optimal number of such dyadic systems.
\end{lem}

Suppose that $b$ is a locally integrable complex-valued function. For simplicity, for any given cube $Q\subset\mathbb{R}^n$, we let
\begin{equation*}
	\begin{aligned}
		MO_1(b;Q)=\frac{1}{|Q|}\int_Q \Big|b(x)-\big\langle \frac{\mathbbm{1}_Q}{|Q|},b\big\rangle\Big| dx
	\end{aligned}
\end{equation*}
and 
\begin{equation*}
	\begin{aligned}
		MO_2(b;Q)=\bigg(\frac{1}{|Q|}\int_{Q} \Big|b(x)-\big\langle \frac{\mathbbm{1}_{Q}}{|Q|},b\big\rangle\Big|^2 dx\bigg)^{1/2}.
	\end{aligned}
\end{equation*}

Now we give the definition of weak Besov spaces.
Let $0<p<\infty$, $0<q\le \infty$. Let $\mathcal{D}^{\omega(1)}, \mathcal{D}^{\omega(2)}, \cdots, \mathcal{D}^{\omega(n+1)}$ be $n+1$ dyadic systems as in Lemma \ref{Domegan1}. The weak Besov space $\pmb{WB}_{p,q}(\mathbb{R}^n)$ is the space of all locally integrable complex-valued functions $b$ such that
\begin{equation*}
	\|b\|_{\pmb{WB}_{p,q}(\mathbb{R}^n)}=\sum_{i=1}^{n+1}\Big\|\big\{MO_1(b;Q)\big\}_{Q\in\mathcal{D}^{\omega(i)}}\Big\|_{\ell_{p,q}}<\infty.
\end{equation*}

The following several lemmas will be used in our later proofs.

\begin{lem}\label{Domegan2}
	Let $\mathcal{D}^{\omega(1)},\mathcal{D}^{\omega(2)},\cdots,\mathcal{D}^{\omega(n+1)}$ be $n+1$ dyadic systems as in Lemma \ref{Domegan1}. For any cube $B\subset \mathbb{R}^n$ of length $2^{-k}$ with $k\in\mathbb{Z}$, let $Q\in\bigcup\limits_{i=1}^{n+1} \mathcal{D}^{\omega(i)}$ such that
	\begin{equation*}
		B\subseteq Q\subseteq c_n B.
	\end{equation*}
	Then for any given $\omega\in(\{0,1\}^n)^\mathbb{Z}$, $Q$ contains only a finite number of dyadic cubes in $\mathcal{D}_k^\omega$, and this number only depends on $n$.
\end{lem}
\begin{proof}
	Fix $k\in\mathbb{Z}$. From Lemma \ref{Domegan1}, we have
	\begin{equation*}
		\ell(Q)\le c_n \ell(B)=c_n2^{-k}.
	\end{equation*}
	This implies that $Q$ contains at most $\lfloor c_n\rfloor^n$ dyadic cubes in $\mathcal{D}_k^\omega$.
\end{proof}

\begin{lem}\label{IHbxiaolemma}
	Let $0<p<\infty$, $0<q\le\infty$ and $\omega\in(\{0,1\}^n)^\mathbb{Z}$. If $b\in\pmb{WB}_{p,q}(\mathbb{R}^n)$, then for any $\xi\in\{0,1\}^n_0$, 
	\begin{equation*}
		\Big\|\big\{|I|^{-1/2}|\langle H_I^\xi,b\rangle|\big\}_{I\in\mathcal{D}^\omega}\Big\|_{\ell_{p,q}}\lesssim_{n,p,q}\|b\|_{\pmb{WB}_{p,q}(\mathbb{R}^n)}.
	\end{equation*}
\end{lem}
\begin{proof}
	Note that
	\begin{equation*}
		\frac{1}{|I|^{1/2}}|\langle H_I^\xi,b\rangle|=\frac{1}{|I|^{1/2}}\bigg|\int_I \overline{H_I^\xi(x)}\Big(b(x)-\big\langle \frac{\mathbbm{1}_I}{|I|},b\big\rangle\Big)dx\bigg|\le MO_1(b;I). 
	\end{equation*}
	For any given $I\in\mathcal{D}^\omega_k$, by Lemma \ref{Domegan1}, there exists some cube $Q(I)\in \bigcup\limits_{i=1}^{n+1} \mathcal{D}^{\omega(i)}$ satisfying
	\begin{equation*}
		I\subseteq Q(I)\subseteq c_n I,
	\end{equation*}
	where $c_n$ only depends on $n$. Thus
	\begin{equation}\label{IHbxiao}
		\begin{aligned}
			MO_1(b;I){}&\le \frac{1}{|I|}\int_I \Big|b(x)-\big\langle \frac{\mathbbm{1}_{Q(I)}}{|Q(I)|},b\big\rangle\Big| dx+\bigg|\big\langle \frac{\mathbbm{1}_{Q(I)}}{|Q(I)|},b\big\rangle-\big\langle \frac{\mathbbm{1}_I}{|I|},b\big\rangle\bigg|\\
			&\le \frac{2}{|I|}\int_I \Big|b(x)-\big\langle \frac{\mathbbm{1}_{Q(I)}}{|Q(I)|},b\big\rangle\Big| dx\\
			&\lesssim_{n} MO_1(b;Q(I)).
		\end{aligned}
	\end{equation}
	From Lemma \ref{Domegan2}, it implies that $Q(I)$ contains only a finite number of dyadic cubes in $\mathcal{D}_k^\omega$, and the number only depends on $n$.
	Then from \eqref{IHbxiao} we have
	\begin{equation}\label{IHbxiao222}
		\begin{aligned}
			\Big\|\big\{|I|^{-1/2}|\langle H_I^\xi,b\rangle|\big\}_{I\in\mathcal{D}^\omega}\Big\|_{\ell_{p,q}}
			&\lesssim_{n}\Big\|\big\{MO_1(b;Q(I))\big\}_{I\in\mathcal{D}^{\omega}}\Big\|_{\ell_{p,q}}\\
			&\lesssim_{n}\bigg\|\big\{MO_1(b;Q)\big\}_{Q\in\bigcup\limits_{i=1}^{n+1}\mathcal{D}^{\omega^{(i)}}}\bigg\|_{\ell_{p,q}}\\
			&\lesssim_{n,p,q}\sum_{i=1}^{n+1}\Big\|\big\{MO_1(b;Q)\big\}_{Q\in\mathcal{D}^{\omega(i)}}\Big\|_{\ell_{p,q}}
			=\|b\|_{\pmb{WB}_{p,q}(\mathbb{R}^n)}.
		\end{aligned}
	\end{equation}
\end{proof}

\begin{lem}\label{TECHRS}
	Let $0<p<\infty$, $0<q\le \infty$ and $\omega\in(\{0,1\}^n)^\mathbb{Z}$. If $b\in \pmb{WB}_{p,q}(\mathbb{R}^n)$, then
	\begin{equation*}
		\Big\|\big\{MO_2(b;I)\big\}_{I\in\mathcal{D}^\omega}\Big\|_{\ell_{p,q}}\lesssim_{n,p,q} \|b\|_{\pmb{WB}_{p,q}(\mathbb{R}^n)}.
	\end{equation*}
\end{lem}
\begin{proof}
	Note that for any $I\in\mathcal{D}^\omega$, by the Minkovski inequality and the H\"{o}lder inequality,
	\begin{equation*}
		\begin{aligned}
			MO_2(b;I){}&\le \bigg(\frac{1}{|I|}\int_I \Big|b(x)-\big\langle \frac{\mathbbm{1}_{4I}}{|4I|},b\big\rangle\Big|^2 dx\bigg)^{1/2}+\bigg|\big\langle \frac{\mathbbm{1}_{4I}}{|4I|},b\big\rangle-\big\langle \frac{\mathbbm{1}_I}{|I|},b\big\rangle\bigg|\\
			&\le \bigg(\frac{1}{|I|}\int_I \Big|b(x)-\big\langle \frac{\mathbbm{1}_{4I}}{|4I|},b\big\rangle\Big|^2 dx\bigg)^{1/2}+\frac{1}{|I|}\int_I \Big|b(x)-\big\langle \frac{\mathbbm{1}_{4I}}{|4I|},b\big\rangle\Big| dx\\
			&\le 2\bigg(\frac{1}{|I|}\int_I \Big|b(x)-\big\langle \frac{\mathbbm{1}_{4I}}{|4I|},b\big\rangle\Big|^2 dx\bigg)^{1/2}\lesssim_n MO_2(b;4I).
		\end{aligned}
	\end{equation*}
	It implies that
	\begin{equation*}
		\Big\|\big\{MO_2(b;I)\big\}_{I\in\mathcal{D}^\omega}\Big\|_{\ell_{p,q}}\lesssim_n \Big\|\big\{MO_2(b;4I)\big\}_{I\in\mathcal{D}^\omega}\Big\|_{\ell_{p,q}}.
	\end{equation*}
	Besides, by \cite[Proposition 4.1]{RSe},
	\begin{equation*}
		\Big\|\big\{MO_2(b;4I)\big\}_{I\in\mathcal{D}^\omega}\Big\|_{\ell_{p,q}}\lesssim_{n,p,q}\Big\|\big\{MO_1(b;4I)\big\}_{I\in\mathcal{D}^\omega}\Big\|_{\ell_{p,q}}.
	\end{equation*}
In the same way as in Lemma \ref{IHbxiaolemma}, one has
	\begin{equation*}
		\begin{aligned}
			\Big\|\big\{MO_1(b;4I)\big\}_{I\in\mathcal{D}^\omega}\Big\|_{\ell_{p,q}}
			&\lesssim_{n,p,q}\|b\|_{\pmb{WB}_{p,q}(\mathbb{R}^n)}.
		\end{aligned}
	\end{equation*}
	Therefore, 
	\begin{equation*}
		\Big\|\big\{MO_2(b;I)\big\}_{I\in\mathcal{D}^\omega}\Big\|_{\ell_{p,q}}\lesssim_{n,p,q} \|b\|_{\pmb{WB}_{p,q}(\mathbb{R}^n)}.
	\end{equation*}
	This finishes the proof.
\end{proof}
The following lemma reveals the relationship between weak Besov spaces and homogeneous Sobolev spaces.
\begin{lem}\label{WBNW1N}
	Let $n\ge 2$. Suppose that $b$ is a locally integrable complex-valued function. Then $b\in \pmb{WB}_{n,\infty}(\mathbb{R}^n)$ if and only if $b\in \dot{W}^1_n(\mathbb{R}^n)$. Moreover,
	\begin{equation*}
		\|b\|_{\pmb{WB}_{n,\infty}(\mathbb{R}^n)}\approx_{n} \|b\|_{\dot{W}^1_n(\mathbb{R}^n)}.
	\end{equation*}
\end{lem}
\begin{rem}
	Let $n\ge 2$. Rochberg and Semmes first showed $\dot{W}^1_n(\mathbb{R}^n)\subseteq \pmb{WB}_{n,\infty}(\mathbb{R}^n)$  in \cite[Theorem 2.2]{RS1988}, and Connes, Sullivan
	and Teleman proved the converse inclusion $\pmb{WB}_{n,\infty}(\mathbb{R}^n)\subseteq \dot{W}^1_n(\mathbb{R}^n)$ in \cite[Appendix]{CST1994}. Alternatively, Lemma \ref{WBNW1N} can be obtained by \cite[Theorem 1]{FRANK}. 
\end{rem}
In \cite{RS1988}, Rochberg and Semmes also showed implicitly the following lemma, which is the relation between weak Besov spaces and homogeneous Besov spaces.
\begin{lem}\label{RS4.9}
	Let $p>n$. Suppose that $b$ is a locally integrable complex-valued function. Then $b\in\pmb{WB}_{p,p}(\mathbb{R}^n)$ if and only if $b\in \pmb{B}_p(\mathbb{R}^n)$. Moreover,
	\begin{equation*}
		\|b\|_{\pmb{WB}_{p,p}(\mathbb{R}^n)}\approx_{n,p} \|b\|_{\pmb{B}_p(\mathbb{R}^n)}.
	\end{equation*}
\end{lem}

From Lemma \ref{WBNW1N} and Lemma \ref{RS4.9}, we obtain the following real interpolation.
\begin{lem}\label{RS4.10}
	Let $p_1>p>n\geq 2$ and $1\leq q\leq \infty$. Then
	\begin{equation*}
		\pmb{WB}_{p,q}(\mathbb{R}^n)=\left(\dot{W}^1_n(\mathbb{R}^n), \pmb{B}_{p_1}(\mathbb{R}^n)\right)_{\theta, q},
	\end{equation*}
	where $1/p=(1-\theta)/n +\theta/{p_1}$.
\end{lem}

\bigskip

\section{Proof of the sufficiency of Theorem \ref{corollary1.8}}\label{Application 3}

We start with the following lemma which shows that $\|b\|_{\pmb{B}_p^{\omega, 2^n}(\mathbb{R}^n)}$ can be dominated by $\|b\|_{\pmb{B}_p(\mathbb{R}^n)}$. The converse to this lemma can be found in Proposition \ref{pdayun}.

\begin{lemma}\label{Comparison}
	Let $1\leq p<\infty$. If $b\in\pmb{B}_p(\mathbb{R}^n)$, then $b\in\pmb{B}_p^{\omega, 2^n}(\mathbb{R}^n)$. Moreover, in this case, $\|b\|_{\pmb{B}_p^{\omega, 2^n}(\mathbb{R}^n)}\lesssim_{n} \|b\|_{\pmb{B}_p(\mathbb{R}^n)}$.
\end{lemma}	

\begin{proof}
	Without loss of generality, assume $\omega=0$. For any given $J\in\mathcal{D}^0$ and $\eta\in\{0,1\}^n_0$,
	\begin{equation*}
		\begin{aligned}
			\frac{|\langle H_J^\eta,b\rangle|}{|J|^{1/2}}
			{}&=\frac{1}{|J|^{1/2}}\bigg|\biggl\langle H_J^\eta,b-\biggl\langle \frac{\mathbbm{1}_{J}}{|J|},b\biggr\rangle\biggr\rangle\bigg|\\
			&\le \frac{1}{|J|}\int_J\biggl|b(x)-\biggl\langle \frac{\mathbbm{1}_J}{|J|},b\biggr\rangle\biggr| dx\\
			&\le \frac{1}{|J|^2}\int_{J\times J}|b(x)-b(y)|dxdy.
		\end{aligned}
	\end{equation*}
	Given $t\in\mathbb{Z}$, $I\in \mathcal{D}^0_t$ and $\eta\in\{0,1\}^n_0$, by the H\"{o}lder inequality, one has
	\begin{equation*}
		\begin{aligned}
			\sum_{J\in\mathcal{D}^0(I)}\frac{|\langle H_J^\eta,b\rangle|^p}{|J|^{p/2}}{}&\le  \sum_{J\in\mathcal{D}^0(I)} \frac{1}{|J|^2}\int_{J\times J}|b(x)-b(y)|^pdxdy\\
			&=\sum_{s=t}^\infty\sum_{J\in\mathcal{D}^0_s(I)}\frac{1}{(2^{n(t-s)}|I|)^2}\int_{J\times J}|b(x)-b(y)|^pdxdy\\
			&=\frac{1}{|I|^2}\int_{\mathbb{R}^n\times\mathbb{R}^n}K_I(x,y)|b(x)-b(y)|^pdxdy,
		\end{aligned}
	\end{equation*}
	where
	\begin{equation*}
		K_I(x,y)=\sum_{s=t}^\infty\sum_{J\in\mathcal{D}^0_s(I)}2^{2n(s-t)}\mathbbm{1}_J(x)\mathbbm{1}_J(y).
	\end{equation*}
	Clearly,  if $x\notin I$ or $y\notin I$, then $K_I(x,y)=0$. On the other hand, suppose that $x,y\in I$, and $|x-y|>\sqrt{n}\ell(J)$ for some $J\in \mathcal{D}^0_s(I)$. Then $\exists 1\le k\le n$ such that $|x_k-y_k|>\ell(J)$, where $x_k$ is the $k$-th coordinate of $x$. We then deduce that $\mathbbm{1}_J(x)\mathbbm{1}_J(y)=0$. Hence
	\begin{equation*}
		K_I(x,y)\le \mathbbm{1}_I(x)\mathbbm{1}_I(y)\sum_{s=t}^{t+\lfloor\log_2(\sqrt{n}\ell(I)/|x-y|)\rfloor}2^{2n(s-t)}\le \frac{(4n)^n}{4^n-1}\frac{|I|^2}{|x-y|^{2n}}\mathbbm{1}_I(x)\mathbbm{1}_I(y),
	\end{equation*}
	where $\lfloor\cdot\rfloor$ is the floor function.
	
	Therefore, for a given $t\in\mathbb{Z}$, we sum up all $I\in \mathcal{ D}^0_t$, and obtain
	\begin{equation*}
		\sum_{I\in\mathcal{D}^0_t}\sum_{J\in\mathcal{D}^0(I)}\sum_{\eta}\frac{|\langle H_J^\eta,b\rangle|^p}{|J|^{p/2}}\le\frac{(4n)^n(2^n-1)}{4^n-1}\int_{\mathbb{R}^n\times\mathbb{R}^n}\frac{|b(x)-b(y)|^p}{|x-y|^{2n}}dxdy.
	\end{equation*}
	Finally letting $t\to -\infty$, we have
	\begin{equation*}
		\|b\|^p_{\pmb{B}_p^{0, 2^n}(\mathbb{R}^n)}\lesssim_{n}\|b\|^p_{\pmb{B}_p(\mathbb{R}^n)}.
	\end{equation*}

\end{proof}	

\subsection{Proof of the sufficiency of Theorem \ref{corollary1.8}}
The subsection is devoted to the proof of the sufficiency of Theorem \ref{corollary1.8}. We divide the proof into two cases:
 \begin{enumerate}
 	\item $2\le p<\infty$, $0<\alpha\le 1$ when $n\ge 1$,
 	\item $1< p<2$, $1/2\le \alpha\le 1$ when $n=1$.
 \end{enumerate}
  The second one is much subtler. In addition, in the second case, we will prove a stronger result, that is, the sufficiency part of Theorem \ref{corollary1.8} holds for $1< p<2$, $1/p-1/2< \alpha\le 1$ when $n=1$.

\begin{proof}[Proof of the sufficiency of Theorem \ref{corollary1.8}]
	From  Proposition \ref{T0est} and Lemma \ref{Comparison}, we have
	\begin{equation*}
		\|[\pi_{T(1)}^{\omega}, M_b]\|_{S_p(L_2(\mathbb{R}^n))}\lesssim_{n,p}\|T(1)\|_{BMO(\mathbb{R}^n)} \|b\|_{\pmb{B}_p(\mathbb{R}^n)}
	\end{equation*}
	and
	\begin{equation*}
		\|[(\pi_{T^*(1)}^{\omega})^*, M_b]\|_{S_p(L_2(\mathbb{R}^n))}\lesssim_{n,p} \|T^*(1)\|_{BMO(\mathbb{R}^n)} \|b\|_{\pmb{B}_p(\mathbb{R}^n)}.
	\end{equation*}
	Hence by Theorem \ref{CZdec}, it remains to estimate $\|[S_\omega^{ij}, M_b]\|_{S_p(L_2(\mathbb{R}^n))}$ for any $i, j\in \mathbb{N}\cup\{0\}$. By the triangle inequality
	\begin{equation*}
		\begin{aligned}
			\|[S_\omega^{ij},M_b]\|_{S_p(L_2(\mathbb{R}^n))}{}&\le \|[S_\omega^{ij},\pi_b]\|_{S_p(L_2(\mathbb{R}^n))}+\|[S_\omega^{ij},\varLambda_b]\|_{S_p(L_2(\mathbb{R}^n))}
			+\|[S_\omega^{ij},R_b]\|_{S_p(L_2(\mathbb{R}^n))}.
		\end{aligned}
	\end{equation*}
	Here the operators $\pi_b$, $\varLambda_b$ and $R_b$ are associated with the dyadic system $\mathcal{ D}^\omega$.
	
	From Theorem \ref{lem2.1} and Lemma \ref{TLambdab}, we know that
	\begin{equation*}
		\begin{aligned}
			\|\pi_b\|_{S_p(L_2(\mathbb{R}^n))}\lesssim_{n, p}\|b\|_{\pmb{B}_p^{\omega, 2^n}(\mathbb{R}^n)},\quad \|\varLambda_b\|_{S_p(L_2(\mathbb{R}^n))}\lesssim_{n, p}\|b\|_{\pmb{B}_p^{\omega, 2^n}(\mathbb{R}^n)}.
		\end{aligned}
	\end{equation*}
	Meanwhile, recall that $S_\omega^{ij}\in B(L_2(\mathbb{R}^n))$ is contractive. Thus, using Lemma \ref{Comparison}, one gets
	\begin{equation*}
		\begin{aligned}
			\|[S_\omega^{ij},\pi_b]\|_{S_p(L_2(\mathbb{R}^n))}&\lesssim \|S_\omega^{ij}\|\|\pi_b\|_{S_p(L_2(\mathbb{R}^n))}
			\lesssim_{n,p} \|b\|_{\pmb{B}_p^{\omega, 2^n}(\mathbb{R}^n)}\lesssim_{n} \|b\|_{\pmb{B}_p(\mathbb{R}^n)}.
		\end{aligned}
	\end{equation*}
	Similarly,
	$$ \|[S_\omega^{ij},\varLambda_b]\|_{S_p(L_2(\mathbb{R}^n))}\lesssim_{n,p} \|b\|_{\pmb{B}_p(\mathbb{R}^n)}.$$
	
	For any $i, j\in \mathbb{N}\cup\{0\}$, we will show that $\|[S^{ij}_\omega,R_b]\|_{S_p(L_2(\mathbb{R}^n))}$ increases with polynomial growth with respect to $i$ and $j$ uniformly on $\omega$. Then from Theorem \ref{CZdec} and the triangle inequality, the desired result will follow.
	
	Without loss of generality, we assume $\omega=0$.
	Let $\varPhi=[S^{ij}_0,R_b]$. Then
	\begin{equation}\label{Ndef}
		\begin{aligned}
			\varPhi(f)=\sum_{K\in\mathcal{D}^0}\sum_{\substack{I,J\in\mathcal{D}^0;I,J\subseteq K\\ \ell(I)=2^{-i}\ell(K)\\\ell(J)=2^{-j}\ell(K)}}\sum_{\xi,\eta\in\{0,1\}^n_0}a^{\xi\eta}_{IJK}\langle H^\xi_I,R_b(f)\rangle H^\eta_J-\sum_{k\in \mathbb{Z}}b_{k-1}d_k(S^{ij}_0(f)).
		\end{aligned}
	\end{equation}
	Note that for any $k\in\mathbb{Z}$ and $\xi\in\{0,1\}^n_0$, if $I\in\mathcal{D}^0_k$, then $d_{k+1}H^\xi_I=H^\xi_I$. Hence, for any $I\in \mathcal{ D}_k^0$,
	\begin{equation*}
		\begin{aligned}
			\langle H^\xi_I,R_b(f)\rangle{}&=\biggl\langle H^\xi_I,\sum_{l\in\mathbb{Z}}b_{l-1}d_lf\biggr\rangle
			=\biggl\langle d_{k+1}H^\xi_I,\sum_{l\in\mathbb{Z}}b_{l-1}d_lf\biggr\rangle\\
			&=\biggl\langle H^\xi_I,d_{k+1}\biggl(\sum_{l\in\mathbb{Z}}b_{l-1}d_lf\biggr)\biggr\rangle
			=\langle H^\xi_I,b_{k}d_{k+1}f\rangle\\
			&=\langle H^\xi_I,b_{k}\mathbbm{1}_Id_{k+1}f\rangle
			=\biggl\langle H^\xi_I,\biggl\langle \frac{\mathbbm{1}_I}{|I|},b\biggr\rangle\mathbbm{1}_Id_{k+1}f\biggr\rangle\\
			&=\biggl\langle \frac{\mathbbm{1}_I}{|I|},b\biggr\rangle\langle H^\xi_I,d_{k+1}f\rangle=\biggl\langle \frac{\mathbbm{1}_I}{|I|},b\biggr\rangle\langle H^\xi_I,f\rangle.
		\end{aligned}
	\end{equation*}
	For the second term in \eqref{Ndef}, one has
	\begin{equation*}
		\begin{aligned}
			\sum_{k\in \mathbb{Z}}b_{k-1}d_k(S^{ij}_0(f)){}&=\sum_{k\in \mathbb{Z}}b_{k-1}d_k\biggl(\sum_{K\in\mathcal{D}^0}\sum_{\substack{I,J\in\mathcal{D}^0;I,J\subseteq K\\ \ell(I)=2^{-i}\ell(K)\\\ell(J)=2^{-j}\ell(K)}}\sum_{\xi,\eta}a_{IJK}^{\xi\eta}\langle H^\xi_I,f\rangle H^\eta_J\biggr)\\&=\sum_{k\in \mathbb{Z}}\sum_{K\in\mathcal{D}^0_{k-1-j}}\sum_{\substack{I,J\in\mathcal{D}^0;I,J\subseteq K\\ \ell(I)=2^{-i}\ell(K)\\\ell(J)=2^{-j}\ell(K)}}\sum_{\xi,\eta}a_{IJK}^{\xi\eta}b_{k-1}\langle H^\xi_I,f\rangle H^\eta_J\\
			&=\sum_{k\in \mathbb{Z}}\sum_{K\in\mathcal{D}^0_{k-1-j}}\sum_{\substack{I,J\in\mathcal{D}^0;I,J\subseteq K\\ \ell(I)=2^{-i}\ell(K)\\\ell(J)=2^{-j}\ell(K)}}\sum_{\xi,\eta}a_{IJK}^{\xi\eta}\biggl\langle \frac{\mathbbm{1}_J}{|J|},b\biggr\rangle\langle H^\xi_I,f\rangle  H^\eta_J\\
			&=\sum_{K\in\mathcal{D}^0}\sum_{\substack{I,J\in\mathcal{D}^0;I,J\subseteq K\\ \ell(I)=2^{-i}\ell(K)\\\ell(J)=2^{-j}\ell(K)}}\sum_{\xi,\eta}a_{IJK}^{\xi\eta}\biggl\langle \frac{\mathbbm{1}_J}{|J|},b\biggr\rangle \langle H^\xi_I,f\rangle H^\eta_J.
		\end{aligned}
	\end{equation*}
	Therefore,
	\begin{equation}\label{BKdef}
		\begin{aligned}
			\varPhi(f){}&=\sum_{K\in\mathcal{D}^0}\sum_{\substack{I,J\in\mathcal{D}^0;I,J\subseteq K\\ \ell(I)=2^{-i}\ell(K)\\\ell(J)=2^{-j}\ell(K)}}\sum_{\xi,\eta}a_{IJK}^{\xi\eta}\biggl(\biggl\langle \frac{\mathbbm{1}_I}{|I|},b\biggr\rangle-\biggl\langle \frac{\mathbbm{1}_J}{|J|},b\biggr\rangle\biggr)\langle H^\xi_I,f\rangle H^\eta_J
			=:\sum_{K\in\mathcal{D}^0}B_K(f).
		\end{aligned}
	\end{equation}
	Let $b_{IJ}= \langle \frac{\mathbbm{1}_I}{|I|},b\rangle-\langle \frac{\mathbbm{1}_J}{|J|},b\rangle$.
	Since $B_{K_1}$, $B_{K_2}$ have orthogonal ranges when $K_1\neq K_2$, we see
	\[B_{K_1}^*B_{K_2}=0,\quad \forall K_1\neq K_2, K_1,K_2\in\mathcal{D}^0,\]
	which yields $\varPhi^*\varPhi=\sum\limits_{K\in \mathcal{ D}^0}B_K^*B_K$.
	Note that $\forall f\in L_2(\mathbb{R}^n)$, 
	\begin{equation}\label{BkBK}
		B_K^*B_K(f)=\sum_{\substack{I,\tilde{I},J\in\mathcal{D}^0;I,\tilde{I},J\subseteq K\\ \ell(I)=\ell(\tilde{I})=2^{-i}\ell(K)\\\ell(J)=2^{-j}\ell(K)}}\sum_{\xi,\tilde{\xi},\eta}a_{IJK}^{\xi\eta}\overline{a_{\tilde{I}JK}^{\tilde{\xi}\eta}}\overline{b_{\tilde{I}J}} b_{IJ}\la H^\xi_I,f\rangle H^{\tilde{\xi}}_{\tilde{I}},
	\end{equation}   
	which implies that $\varPhi^*\varPhi$ is a block diagonal matrix with blocks $B_K^*B_K$ for all $K\in \mathcal{ D}^0$.
	Consequently, we have
	\begin{equation}\label{Nf1}
		\begin{aligned}
			\|\varPhi\|_{S_p(L_2(\mathbb{R}^n))}^p&=\|\varPhi^*\varPhi\|_{S_{p/2}(L_2(\mathbb{R}^n))}^{p/2}
			=\sum_{k\in\mathbb{Z}}\sum_{K\in\mathcal{D}^0_k}\|B^*_KB_K\|_{S_{p/2}(L_2(\mathbb{R}^n))}^{p/2}.
		\end{aligned}
	\end{equation}
Denote by 
\begin{equation*}
	[B_K^*B_K]=\biggl((B_K^*B_K)_{(\tilde{I},\tilde{\zeta}),(I,\zeta)}\biggr)_{\tilde{I},I\in \mathcal{D}^0;\tilde{I},I\subseteq K , \ell(\tilde{I})=\ell(I)=2^{-i}\ell(K), \tilde{\zeta},\zeta\in\{0,1\}^n_0}
\end{equation*}
the matrix form of $B_K^*B_K$ with respect to the basis $\{H_I^\zeta\}_{I\in\mathcal{D}^0;I\subseteq K , \ell(I)=2^{-i}\ell(K), \zeta\in\{0,1\}^n_0}$, where $(B_K^*B_K)_{(\tilde{I},\tilde{\zeta}),(I,\zeta)}=\langle  H_{\tilde{I}}^{\tilde{\zeta}},B_K^*B_K H_I^\zeta\rangle$. We also denote the  $2^{in}(2^n-1)\times 2^{in}(2^n-1)$ matrix by
\begin{equation}\label{W}
	W^{K,J,\eta}=\bigg(W_{(\tilde{I},\tilde{\zeta}),(I,\zeta)}^{K,J,\eta}\bigg)_{\tilde{I},I\in \mathcal{D}^0;\tilde{I},I\subseteq K , \ell(\tilde{I})=\ell(I)=2^{-i}\ell(K), \tilde{\zeta},\zeta\in\{0,1\}^n_0},
\end{equation}
where $W_{(\tilde{I},\tilde{\zeta}),(I,\zeta)}^{K,J,\eta}=a_{IJK}^{\zeta\eta}\overline{a_{\tilde{I}JK}^{\tilde{\zeta}\eta}}\overline{b_{\tilde{I}J}} b_{IJ}$. Then by \eqref{BkBK} one has
\begin{equation*}
	[B_K^*B_K]=\sum_{\substack{J\in\mathcal{D}^0;J\subseteq K\\ \ell(J)=2^{-j}\ell(K)}}\sum_{\eta} W^{K,J,\eta}.
\end{equation*}

	Now we divide the proof into two cases: \\
	$\bullet$ $2\le p<\infty$, $0<\alpha\le 1$ when $n\ge 1$;\\
	$\bullet$ $1< p<2$, $1/p-1/2< \alpha\le 1$ when $n=1$.\\
	
	\noindent (1) We first consider the case $n\ge 1$, $2\le p<\infty$ and $0<\alpha\le 1$. By using the triangle inequality, one has
	\begin{equation*}
		\begin{aligned}
			\|B^*_KB_K\|_{S_{p/2}(L_2(\mathbb{R}^n))}{}&=\|[B_K^*B_K]\|_{S_{p/2}(\mathbb{M}_{2^{in}(2^n-1)})}\le \sum_{\substack{J\in\mathcal{D}^0;J\subseteq K\\ \ell(J)=2^{-j}\ell(K)}}\sum_{\eta}\big\|W^{K,J,\eta}\big\|_{S_{p/2}(\mathbb{M}_{2^{in}(2^n-1)})}.
		\end{aligned}
	\end{equation*}
Notice that
\begin{equation}\label{WVV}
	\begin{aligned}
		\big\|W^{K,J,\eta}\big\|_{S_{p/2}(\mathbb{M}_{2^{in}(2^n-1)})}{}
		&=\sum_{\substack{{I}\in\mathcal{D}^0;{I}\subseteq K \\ \ell({I})=2^{-i}\ell(K)\\ {\zeta}\in\{0,1\}^n_0}}a_{{I}JK}^{{\zeta}\eta}\overline{a_{{I}JK}^{{\zeta}\eta}}b_{{I}J}\overline{b_{{I}J}} \\
		&=\sum_{\substack{{I}\in\mathcal{D}^0;{I}\subseteq K \\ \ell({I})=2^{-i}\ell(K)\\ {\zeta}\in\{0,1\}^n_0}}\Big|a_{{I}JK}^{{\zeta}\eta} b_{{I}J}\Big|^2.
	\end{aligned}
\end{equation}
Besides, note that
$|a_{IJK}^{\xi\eta}|\le 2^{-(i+j)n/2}$ in \eqref{aijkxieta}. Hence by (\ref{Nf1}) and (\ref{WVV}),
\begin{equation}\label{Nf}
	\begin{aligned}
		\|\varPhi\|_{S_p(L_2(\mathbb{R}^n))}^p{}&\le \sum_{k\in\mathbb{Z}}\sum_{K\in\mathcal{D}^0_k}\biggl(\sum_{\substack{J\in\mathcal{D}^0;J\subseteq K\\ \ell(J)=2^{-j}\ell(K)}}\sum_{\eta}\big\|W^{K,J,\eta}\big\|_{S_{p/2}(\mathbb{M}_{2^{in}(2^n-1)})}\biggr)^{p/2}\\
		&\le\sum_{k\in\mathbb{Z}}\sum_{K\in\mathcal{D}^0_k}\biggl(\sum_{\substack{J\in\mathcal{D}^0;J\subseteq K\\ \ell(J)=2^{-j}\ell(K)}}\sum_{\eta}\sum_{\substack{I\in\mathcal{D}^0;I\subseteq K \\ \ell(I)=2^{-i}\ell(K)}}\sum_{ \xi}\Big|a_{IJK}^{\xi\eta} b_{IJ}\Big|^2\biggr)^{p/2}\\
		&\le(2^n-1)^p\sum_{k\in\mathbb{Z}}\sum_{K\in\mathcal{D}^0_k}\biggl(2^{-(i+j)n}\sum_{\substack{J\in\mathcal{D}^0;J\subseteq K\\ \ell(J)=2^{-j}\ell(K)}}\sum_{\substack{I\in\mathcal{D}^0;I\subseteq K\\ \ell(I)=2^{-i}\ell(K)}}|b_{IJ}|^2\biggr)^{p/2}.
	\end{aligned}
\end{equation}
	Fix $k\in\mathbb{Z}$ and $K\in\mathcal{D}_k^0$. Since $b_{IJ}=b_{IK}-b_{JK}$, by the triangle inequality and the Cauchy-Schwarz inequality, we have
	\begin{equation}
		\begin{aligned}
			{}&\sum_{\substack{J\in\mathcal{D}^0;J\subseteq K\\ \ell(J)=2^{-j}\ell(K)}}\sum_{\substack{I\in\mathcal{D}^0;I\subseteq K\\ \ell(I)=2^{-i}\ell(K)}}|b_{IJ}|^2\\
			&\leq \sum_{\substack{J\in\mathcal{D}^0;J\subseteq K\\ \ell(J)=2^{-j}\ell(K)}}\sum_{\substack{I\in\mathcal{D}^0;I\subseteq K\\ \ell(I)=2^{-i}\ell(K)}} 2(|b_{IK}|^2+|b_{JK}|^2)\\
			&\le 2^{jn+1}\sum_{\substack{I\in\mathcal{D}^0;I\subseteq K\\ \ell(I)=2^{-i}\ell(K)}}|b_{IK}|^2+2^{in+1}\sum_{\substack{J\in\mathcal{D}^0;J\subseteq K\\ \ell(J)=2^{-j}\ell(K)}}|b_{JK}|^2.
		\end{aligned}
	\end{equation}
	Note that $b_{IK}\cdot \mathbbm{1}_I=(b_{k+i}-b_{k})\cdot \mathbbm{1}_I$, and sum all $I$ and $J$, one has
	\begin{equation}\label{2ijnbij}
		\begin{aligned}
			{}&2^{-(i+j)n}\sum_{\substack{J\in\mathcal{D}^0;J\subseteq K\\ \ell(J)=2^{-j}\ell(K)}}\sum_{\substack{I\in\mathcal{D}^0;I\subseteq K\\ \ell(I)=2^{-i}\ell(K)}}|b_{IJ}|^2\\
			&\le 2^{kn+1}\bigg(\sum_{\substack{I\in\mathcal{D}^0;I\subseteq K\\ \ell(I)=2^{-i}\ell(K)}}\|b_{IK}\cdot \mathbbm{1}_I\|_{L_2(\mathbb{R}^n)}^2+\sum_{\substack{J\in\mathcal{D}^0;J\subseteq K\\ \ell(J)=2^{-j}\ell(K)}}\|b_{JK}\cdot \mathbbm{1}_J\|_{L_2(\mathbb{R}^n)}^2\bigg)\\
			&= 2^{kn+1}\bigg(\|(b_{k+i}-b_{k})\mathbbm{1}_K\|_{L_2(\mathbb{R}^n)}^2+\|(b_{k+j}-b_{k})\mathbbm{1}_K\|_{L_2(\mathbb{R}^n)}^2\bigg)\\
			&\le 
			2^{1+2nk/p}\biggl(i\sum_{l=k+1}^{k+i}\|d_lb\cdot \mathbbm{1}_K\|_{L_p(\mathbb{R}^n)}^2+j\sum_{l=k+1}^{k+j}\|d_l b\cdot \mathbbm{1}_K\|_{L_p(\mathbb{R}^n)}^2\biggr).
		\end{aligned}
	\end{equation}
	Hence using the convex inequality, we obtain
	\begin{equation*}
		\begin{aligned}
			{}&\|\varPhi\|_{S_p(L_2(\mathbb{R}^n))}^p\\
			&\le (2^n-1)^p2^{p}\sum_{k\in\mathbb{Z}}\sum_{K\in\mathcal{D}^0_k}2^{nk}\biggl(i^{p-1}\sum_{l=k+1}^{k+i}\|d_lb\cdot\mathbbm{1}_K\|_{L_p(\mathbb{R}^n)}^p+j^{p-1}\sum_{l=k+1}^{k+j}\|d_lb\cdot\mathbbm{1}_K\|_{L_p(\mathbb{R}^n)}^p\biggr)\\
			&=(2^n-1)^p2^{p}(i^{p}+j^{p})\sum_{k\in\mathbb{Z}}2^{nk}\|d_kb\|_{L_p(\mathbb{R}^n)}^p\\
			&\lesssim_{n,p}(i^{p}+j^{p})\|b\|_{\pmb{B}_p^{0, 2^n}(\mathbb{R}^n)}^p\lesssim_{n} (i^{p}+j^{p})\|b\|_{\pmb{B}_p(\mathbb{R}^n)}^p.
		\end{aligned}
	\end{equation*}
	Since the above estimation is independent of the choice of $\omega$, one has
	$$ \|[S_\omega^{ij}, R_b]\|_{S_p(L_2(\mathbb{R}^n))}\lesssim_{n, p} (i^p+j^p)^{1/p}  \|b\|_{\pmb{B}_p(\mathbb{R}^n)}, $$
	which yields
	$$ \|[S_\omega^{ij}, M_b]\|_{S_p(L_2(\mathbb{R}^n))}\lesssim_{n, p} (i^p+j^p+1)^{1/p}  \|b\|_{\pmb{B}_p(\mathbb{R}^n)}. $$
	Therefore by Lemma \ref{weakconver} and the triangle inequality,
	\begin{equation*}
		\begin{aligned}
			{}&\|[T,M_b]\|_{S_p(L_2(\mathbb{R}^n))}\\
			&=\biggl\|\biggl[C_1(T)\mathbb{E}_\omega\sum_{i,j=0\atop \max\{i,j\}>0}^\infty \tau(i,j)S_\omega^{ij}+C_2(T)\mathbb{E}_\omega S_\omega^{00}+\mathbb{E}_\omega \pi^\omega_{T(1)}+\mathbb{E}_\omega (\pi^\omega_{T^*(1)})^*,M_b\biggr]\biggr\|_{S_p(L_2(\mathbb{R}^n))}\\
			&\lesssim_T \sum_{i,j=0}^\infty \tau(i,j)\mathbb{E}_\omega\|[S_\omega^{ij},M_b]\|_{S_p(L_2(\mathbb{R}^n))}+\mathbb{E}_\omega\|[ \pi^\omega_{T(1)}+ (\pi^\omega_{T^*(1)})^*,M_b]\|_{S_p(L_2(\mathbb{R}^n))} \\
			&\lesssim_{n,p, T} \big(1+\|T(1)\|_{BMO(\mathbb{R}^n)}+\|T^*(1)\|_{BMO(\mathbb{R}^n)}\big) \|b\|_{\pmb{B}_p(\mathbb{R}^n)}.
		\end{aligned}
	\end{equation*}
	This finishes the proof for $2\leq p<\infty$, $0<\alpha\le 1$ when $n\ge 1$.\\
	
\noindent (2) Next, we consider $n=1$, $1<p<2$ and $1/p-1/2<\alpha\le 1$. Fix $k\in\mathbb{Z}$ and $K\in\mathcal{D}_k^0$. Let the eigenvalues of $B_K^*B_K$ be
\begin{equation*}
	s_1(K)\ge s_2(K)\ge \cdots\ge s_{2^{i}}(K)\ge 0.
\end{equation*}
Then it is clear that
\begin{equation}\label{smTrb}
	0\le s_m(K)\le \frac{1}{m}\mathrm{Tr}(B_K^*B_K),\quad \forall 1\le m\le 2^{i}.
\end{equation}
Note that $|a_{IJK}^{\xi\eta}|\le 2^{-(i+j)/2}$ in \eqref{aijkxieta}, by \eqref{WVV} and \eqref{2ijnbij},
\begin{equation}\label{Trbkbk}
	\begin{aligned}
		\mathrm{Tr}(B_K^*B_K){}&=\|B_K^*B_K\|_{S_1(\mathbb{M}_{2^{i}})}\\
		&\le \sum_{\substack{J\in\mathcal{D}^0;J\subseteq K\\ \ell(J)=2^{-j}\ell(K)}}\sum_{\eta}\big\|W^{K,J,\eta}\big\|_{S_{1}(\mathbb{M}_{2^{i}})}\\
		&=\sum_{\substack{J\in\mathcal{D}^0;J\subseteq K\\ \ell(J)=2^{-j}\ell(K)}}\sum_{\eta}\sum_{\substack{{I}\in\mathcal{D}^0;{I}\subseteq K \\ \ell({I})=2^{-i}\ell(K)\\ {\zeta}\in\{0,1\}^1_0}}\Big|a_{{I}JK}^{{\zeta}\eta} b_{{I}J}\Big|^2\\
		&\le 2^{-(i+j)}\sum_{\substack{J\in\mathcal{D}^0;J\subseteq K\\ \ell(J)=2^{-j}\ell(K)}}\sum_{\substack{{I}\in\mathcal{D}^0;{I}\subseteq K \\ \ell({I})=2^{-i}\ell(K)}}| b_{{I}J}|^2\\
		&\le 2^{k+1}\big(\|(b_{k+i}-b_{k})\mathbbm{1}_K\|_{L_2(\mathbb{R})}^2+\|(b_{k+j}-b_{k})\mathbbm{1}_K\|_{L_2(\mathbb{R})}^2\big).
	\end{aligned}
\end{equation}
Thus by the convex inequality,
\begin{equation*}
	\begin{aligned}
		\|B_K^*B_K\|_{S_{p/2}(L_2(\mathbb{R}))}^{p/2}{}&=\sum_{m=1}^{2^i}s_m(K)^{p/2}\\
		&\le \sum_{m=1}^{2^i}\frac{1}{m^{p/2}}(\mathrm{Tr}(B_K^*B_K))^{p/2}\\
		&\le \frac{2^{i(1-p/2)}}{1-p/2}\bigg(2^{k+1}\Big(\|(b_{k+i}-b_{k})\mathbbm{1}_K\|_{L_2(\mathbb{R})}^2+\|(b_{k+j}-b_{k})\mathbbm{1}_K\|_{L_2(\mathbb{R})}^2\Big)\bigg)^{p/2}\\
		&\le \frac{2^{i(1-p/2)}}{1-p/2}\cdot 2^{(k+1)p/2}\Big(\|(b_{k+i}-b_{k})\mathbbm{1}_K\|_{L_2(\mathbb{R})}^p+\|(b_{k+j}-b_{k})\mathbbm{1}_K\|_{L_2(\mathbb{R})}^p\Big).
	\end{aligned}
\end{equation*}
It implies that
\begin{equation*}
	\begin{aligned}
		\|\varPhi\|_{S_p(L_2(\mathbb{R}))}^p{}&=\sum_{k\in\mathbb{Z}}\sum_{K\in\mathcal{D}^0_k}\|B^*_KB_K\|_{S_{p/2}(L_2(\mathbb{R}))}^{p/2}\\
		&\le \frac{2^{i(1-p/2)+p/2}}{1-p/2}\sum_{k\in\mathbb{Z}}\sum_{K\in\mathcal{D}^0_k}|K|^{-p/2}\Big(\|(b_{k+i}-b_{k})\mathbbm{1}_K\|_{L_2(\mathbb{R})}^p+\|(b_{k+j}-b_{k})\mathbbm{1}_K\|_{L_2(\mathbb{R})}^p\Big).
	\end{aligned}
\end{equation*}
Note also that
\begin{equation*}
	\begin{aligned}
		\|(b_{k+i}-b_{k})\mathbbm{1}_K\|_{L_2(\mathbb{R})}^2{}&
		=\bigg\|\sum_{l=1}^{i}d_{k+l}b\cdot \mathbbm{1}_K\bigg\|_{L_2(\mathbb{R})}^2\\
		&=\bigg\langle \sum_{l=1}^{i}\sum_{L\in\mathcal{D}^0_{k+l-1}}\sum_{\zeta}\langle H_L^\zeta,b\rangle H_L^\zeta\cdot \mathbbm{1}_K,\sum_{l=1}^{i}\sum_{L\in\mathcal{D}^0_{k+l-1}}\sum_{\zeta}\langle H_L^\zeta,b\rangle H_L^\zeta\cdot \mathbbm{1}_K\bigg\rangle\\
		&=\sum_{l=1}^{i}\sum_{\substack{L\in\mathcal{D}^0_{k+l-1}\\ L\subseteq K}}\sum_{\zeta}\big|\langle H_L^\zeta,b\rangle\big|^2.
	\end{aligned}
\end{equation*}
Thus by the convex inequality again,
\begin{equation*}
	\begin{aligned}
		{}&\sum_{k\in\mathbb{Z}}\sum_{K\in\mathcal{D}^0_k}|K|^{-p/2}\|(b_{k+i}-b_{k})\mathbbm{1}_K\|_{L_2(\mathbb{R})}^p\\
		&\le \sum_{k\in\mathbb{Z}}\sum_{K\in\mathcal{D}^0_k}|K|^{-p/2}\sum_{l=1}^{i}\sum_{\substack{L\in\mathcal{D}^0_{k+l-1}\\ L\subseteq K}}\sum_{\zeta}\big|\langle H_L^\zeta,b\rangle\big|^p\\
		&=\sum_{l=1}^{i}2^{-(l-1)p/2}\sum_{k\in\mathbb{Z}}\sum_{K\in\mathcal{D}^0_k}\sum_{\substack{L\in\mathcal{D}^0_{k+l-1}\\ L\subseteq K}}\sum_{\zeta}|L|^{-p/2}\big|\langle H_L^\zeta,b\rangle\big|^p\\
		&\le \frac{1}{1-2^{-p/2}}\|b\|_{\pmb{B}_p^{0, 2}(\mathbb{R})}^p.
	\end{aligned}
\end{equation*}
Similarly,
\begin{equation*}
	\begin{aligned}
		{}&\sum_{k\in\mathbb{Z}}\sum_{K\in\mathcal{D}^0_k}|K|^{-p/2}\|(b_{k+j}-b_{k})\mathbbm{1}_K\|_{L_2(\mathbb{R})}^p
		\le \frac{1}{1-2^{-p/2}}\|b\|_{\pmb{B}_p^{0, 2}(\mathbb{R})}^p.
	\end{aligned}
\end{equation*}
Hence we obtain
\begin{equation*}
	\|\varPhi\|_{S_p(L_2(\mathbb{R}))}^p\le \frac{2^{i(1-p/2)+p/2+1}}{(1-2^{-p/2})(1-p/2)}\|b\|_{\pmb{B}_p^{0, 2}(\mathbb{R})}^p\lesssim_{p} 2^{i(1-p/2)}\|b\|_{\pmb{B}_p(\mathbb{R})}^p.
\end{equation*}
Since the above estimate is independent of the choice of $\omega$, one has
$$ \|[S_\omega^{ij}, R_b]\|_{S_p(L_2(\mathbb{R}))}\lesssim_{p} 2^{i(1/p-1/2)} \|b\|_{\pmb{B}_p(\mathbb{R})}, $$
which yields
\begin{equation*}
	\begin{aligned}
		\|[S_\omega^{ij}, M_b]\|_{S_p(L_2(\mathbb{R}))}{}&\lesssim_{p} 2^{i(1/p-1/2)} \|b\|_{\pmb{B}_p(\mathbb{R})}.
	\end{aligned}
\end{equation*}
Since $1/p-1/2<\alpha\le 1$, we get
\begin{equation*}
	\begin{aligned}
		\sum_{i,j=0}^\infty &\tau(i,j)\|[S_\omega^{ij}, M_b]\|_{S_p(L_2(\mathbb{R}))}\\
		&\lesssim_{p} \sum_{i,j=0}^\infty (1+\max\{i,j\})^{2(1+\alpha)}2^{\max\{i,j\}\big(1/p-1/2-\alpha\big)}\|b\|_{\pmb{B}_p(\mathbb{R})}<\infty.
	\end{aligned}
\end{equation*}
Therefore by Lemma \ref{weakconver} and the triangle inequality,
\begin{equation*}
	\begin{aligned}
		{}&\quad\|[T,M_b]\|_{S_p(L_2(\mathbb{R}))}
		\lesssim_{p, T} \big(1+\|T(1)\|_{BMO(\mathbb{R})}+\|T^*(1)\|_{BMO(\mathbb{R})}\big) \|b\|_{\pmb{B}_p(\mathbb{R})}.
	\end{aligned}
\end{equation*}
This completes the proof of the sufficiency of Theorem \ref{corollary1.8}.
\end{proof}

\subsection{Comparison between  Theorem \ref{thm1.5} and Theorem \ref{corollary1.8}}
From our proof of Theorem \ref{corollary1.8}, we see that when $p\geq 2$, one always has
$$ \|[T,M_b]\|^p_{S_p(L_2(\mathbb{R}^n))}\lesssim_{n, p} \int_{\mathbb{R}^n\times\mathbb{R}^n}\frac{|b(x)-b(y)|^p}{|x-y|^{2n}}dxdy. $$
However, this does not contradict with Theorem \ref{thm1.5} for $p\leq n$ and $n\geq 2$ due to the following fact.
\begin{proposition}\label{0pn}
	Let $n\ge 1$ and $1\leq p\leq n$. Assume that $b$ is a locally integrable complex-valued function. Then $b$ is constant if
	$$  \int_{\mathbb{R}^n\times\mathbb{R}^n}\frac{|b(x)-b(y)|^p}{|x-y|^{2n}}dxdy<\8. $$
\end{proposition}
\begin{proof}
	We proceed with the proof by contradiction. Assume that $b$ is not constant. Then there exists $\varphi\in C_c^\infty(\mathbb{R}^n)$ such that $b\ast\varphi\in C^\infty(\mathbb{R}^n)$ is not constant either.
	By changing the variables, we have
	\[\int_{\mathbb{R}^n\times\mathbb{R}^n}\frac{|b(x)-b(y)|^p}{|x-y|^{2n}}dxdy=\int_{\mathbb{R}^n}\frac{\|b(x+t)-b(x)\|^p_{L_p(\mathbb{R}^n)}}{|t|^{2n}}dt.\]
	 Since by the Young inequality
	\[\|\varphi\ast b(x+t)-\varphi\ast b(x)\|^p_{L_p(\mathbb{R}^n)}\le \|\varphi\|^p_{L_1(\mathbb{R}^n)}\|b(x+t)-b(x)\|^p_{L_p(\mathbb{R}^n)},\]
	we get
	\[\int_{\mathbb{R}^n}\frac{\|\varphi\ast b(x+t)-\varphi\ast b(x)\|^p_{L_p(\mathbb{R}^n)}}{|t|^{2n}}dt<\infty.\]
	Hence we can assume that $b\in C^\infty(\mathbb{R}^n)$, otherwise we replace $b$ with $b\ast\varphi$.
	
	Since $b$ is not constant, there exists $\tilde{x}=(\tilde{x}_1,\cdots,\tilde{x}_n)\in\mathbb{R}^n$, such that $\nabla b(\tilde{x})\neq 0$. Let $U$ be a unitary matrix in $\mathbb{M}_n$ such that $\nabla b(\tilde{x})\cdot U=(|\nabla b(\tilde{x})|, 0, \cdots, 0)$. We substitute $\tilde{b}(y):=b(y\cdot U)$ for $b$.
	So we can also assume that there exists $\tilde{x}\in\mathbb{R}^n$ with $\nabla b(\tilde{x})=(M, 0, \cdots, 0)$ and $M>0$.
	
	Since $b\in C^\infty(\mathbb{R}^n)$, $\exists\ \delta>0$ such that $\forall \ |y-\tilde{x}|<2\delta$ with
	$$  |\nabla b(y)-\nabla b(\tilde{x})|\leq \dfrac{M}{4}.  $$
	Thus for any $ |x-\tilde{x}|<\delta$ and $|t|<\delta$ with $|t_1|>\frac{|t|}{2}$, by the mean value theorem,
	\begin{equation*}
		\begin{aligned}
			|b(x+t)-b(x)|&=|\nabla b(x+\theta\cdot t)\cdot t| \quad (0<\theta<1)\\
			&\geq |\nabla b(\tilde{x})\cdot t|-|\big(\nabla b(x+\theta\cdot t)-\nabla b(\tilde{x})\big)\cdot t|\\
			&\geq M|t_1|-\dfrac{M|t|}{4}\geq \dfrac{M|t_1|}{2}.
		\end{aligned}
	\end{equation*}
	This yields that
	\begin{equation}\label{phibb66}
		\|b(x+t)-b(x)\|^p_{L_p(\mathbb{R}^n)}\gtrsim _{n, p}\delta^n M^p |t_1|^p.
	\end{equation}
	Consequently, one has
	\begin{equation}\label{bbcont2n}
		\begin{aligned}
			\int_{\mathbb{R}^n}\frac{\|b(x+t)-b(x)\|^p_{L_p(\mathbb{R}^n)}}{|t|^{2n}}dt&\gtrsim_{n, p} \int\limits_{t\in \mathbb{R}^n, |t|<\delta \atop |t_1|>\frac{|t|}{2}} \dfrac{|t_1|^p}{|t|^{2n}}dt\\
			&\gtrsim_{n, p} \int_{0}^\delta \dfrac{r^p}{r^{2n}}\cdot r^{n-1} dr\\
			&= \int_{0}^\delta \dfrac{1}{r^{n+1-p}}dr=\8.
		\end{aligned}
	\end{equation}
	This leads to a contradiction.
\end{proof}

\bigskip

\section{Complex median method}\label{proofdivide}	

This section is devoted to the proof of the complex median method, i.e. Theorem \ref{divideS}. We proceed the proof with some fundamental lemmas. In the sequel, we will always assume that $(\Omega,\mathcal{F},\mu)$ is a measure space. Besides, suppose that $I\in \mathcal{F}$ is of finite measure, and $b$ is always a measurable function on $I$.
\begin{lem}\label{S1S2S}
	There exists a line $l\subset \mathbb{C}$ that divides $\mathbb{C}$ into two closed half-planes $S_1$ and $S_2$ whose intersection is $l$, such that 
	\begin{equation*}
		\mu(\{x\in I:b(x)\in S_i\})\ge \frac{1}{2}\mu(I),\quad i\in\{1, 2\}.
	\end{equation*}
\end{lem}

\begin{proof}
	For any $x\in \mathbb{R}$, let $l_x$ be the line that is perpendicular to the real axis at $x$. The left side of $l_x$ on the complex plane, including $l_x$, is denoted by $P_x$ (see Figure \ref{picture1}). Let 
	\begin{equation*}
		f(x)=\frac{\mu(\{y\in I:b(y)\in P_x\})}{\mu(I)}, \quad x\in \mathbb{R}.
	\end{equation*}
	It is not hard to check that:\\
    $\bullet$ $f$ is increasing,\\
    $\bullet$ $f$ is right-continuous,\\
	$\bullet$ $\lim_{x\to +\infty}f(x)=1$, $\lim_{x\to -\infty}f(x)=0$.

	Define 
	\begin{equation*}
		\alpha=\inf \Big\{x\in \mathbb{R}: f(x)\ge \frac{1}{2}\Big\}.
	\end{equation*}
	On the one hand, since $f$ is right-continuous, we have 
	\begin{equation}\label{S1122}
		f(\alpha)\ge \frac{1}{2}.
	\end{equation}
	On the other hand, note that for any $\varepsilon>0$, $f(\alpha-\varepsilon)<\frac{1}{2}$. Namely, 
	\begin{equation*}
		\mu(\{x\in I:b(x)\in P_{\alpha-\varepsilon}\})<\frac{1}{2}\mu(I).
	\end{equation*}
	Let $\varepsilon\to 0$, then
	\begin{equation*}
		\mu(\{x\in I:b(x)\in P_{\alpha} \backslash l_{\alpha}\})\le \frac{1}{2}\mu(I).
	\end{equation*}
	Denote by $Q_x$ the right side of $l_x$ on the complex plane, including $l_x$. Thus
	\begin{equation}\label{S2122}
		\mu(\{x\in I:b(x)\in Q_{\alpha}\})\ge \frac{1}{2}\mu(I).
	\end{equation}
	Let $l=l_{\alpha}$, $S_1=P_{\alpha}$ and $S_2=Q_{\alpha}$. Then $l$ divides $\mathbb{C}$ into two parts $S_1$ and $S_2$ whose intersection is $l$. Besides, from \eqref{S1122} and \eqref{S2122} we get
	\begin{equation*}
		\mu(\{x\in I:b(x)\in S_i\})\ge \frac{1}{2}\mu(I),\quad i\in\{1, 2\}.
	\end{equation*}	
\end{proof}

\begin{figure}[H] 
	\centering 
	\includegraphics[width=0.6\textwidth]{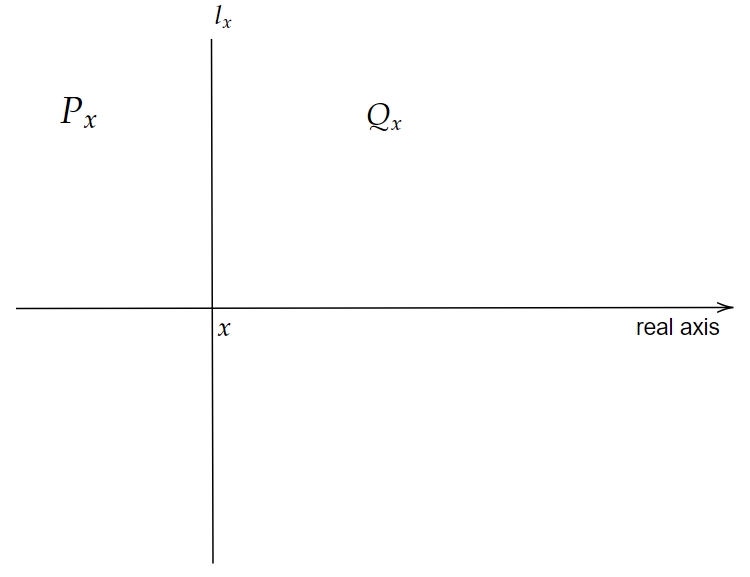} 
	\caption{} 
	\label{picture1} 
\end{figure}

The following lemma is derived by the above one.
\begin{lemma}\label{S1S2S3S4S}
	There exist a line $l$ and two rays $l_1,l_2$ in $\mathbb{C}$ satisfying $l\perp l_1$ and $l\perp l_2$, and they divide $\mathbb{C}$ into four closed quadrants $S_1$, $S_2$, $S_3$ and $S_4$ (see Figure \ref{picture2}), such that
	\begin{equation*}
		\mu(\{x\in I:b(x)\in S_i\})\ge \frac{1}{4}\mu(I),\quad i\in \{1,2,3,4\}.
	\end{equation*}
\end{lemma}

\begin{proof}
	By Lemma \ref{S1S2S}, there exists a line $l$ which divides $\mathbb{C}$ into two closed half-planes $U_1$ and $U_2$, such that 
	\begin{equation*}
		\mu(\{x\in I:b(x)\in U_j\})\ge \frac{1}{2}\mu(I),\quad j\in\{1,2\}.
	\end{equation*} 
	By rotating and translating the axes, we assume that $l$ is the real axis. In terms of $U_1$, repeating the proof of Lemma \ref{S1S2S}, we prove that there exists a ray $l_1$ satisfying $l\perp l_1$, where the origin of $l_1$ is $\alpha_1\in l$, such that $l_1$ divides $U_1$ into two parts $S_1$ and $S_2$ whose intersection is $l_1$. Moreover,
	\begin{equation*}
		\mu(\{x\in I:b(x)\in S_i\})\ge \frac{1}{2}\mu(\{x\in I:b(x)\in U_1\})\ge \frac{1}{4}\mu(I),\quad i\in\{1,2\}.
	\end{equation*}	
	Similarly, in terms of $U_2$, there exists a ray $l_2$ satisfying $l\perp l_2$, where the origin of $l_2$ is $\alpha_2\in l$, such that $l_2$ divides $U_2$ into two parts $S_3$ and $S_4$ whose intersection is $l_2$. Moreover,
	\begin{equation*}
		\mu(\{x\in I:b(x)\in S_i\})\ge \frac{1}{2}\mu(\{x\in I:b(x)\in U_2\})\ge \frac{1}{4}\mu(I),\quad i\in\{3,4\}.
	\end{equation*}	
\end{proof}

\begin{figure}[H]
	\centering 
	\begin{minipage}[b]{0.45\textwidth} 
		\centering 
		\includegraphics[width=1.0\textwidth]{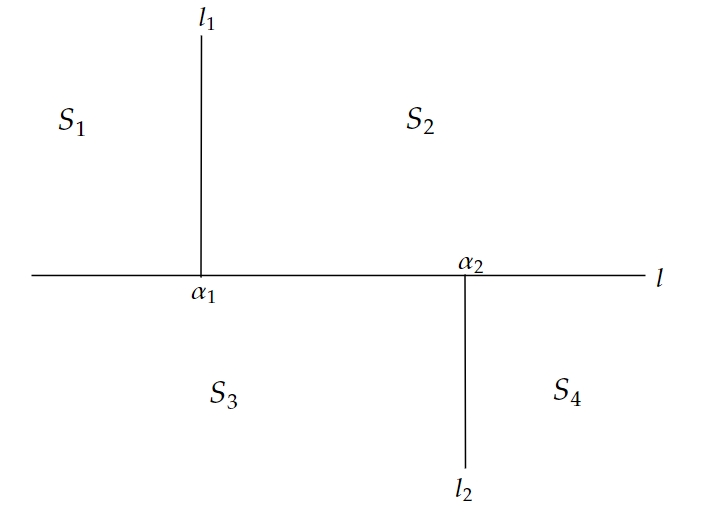} 
		\caption{}
		\label{picture2}
	\end{minipage}
	\begin{minipage}[b]{0.45\textwidth} 
		\centering 
		\includegraphics[width=1.0\textwidth]{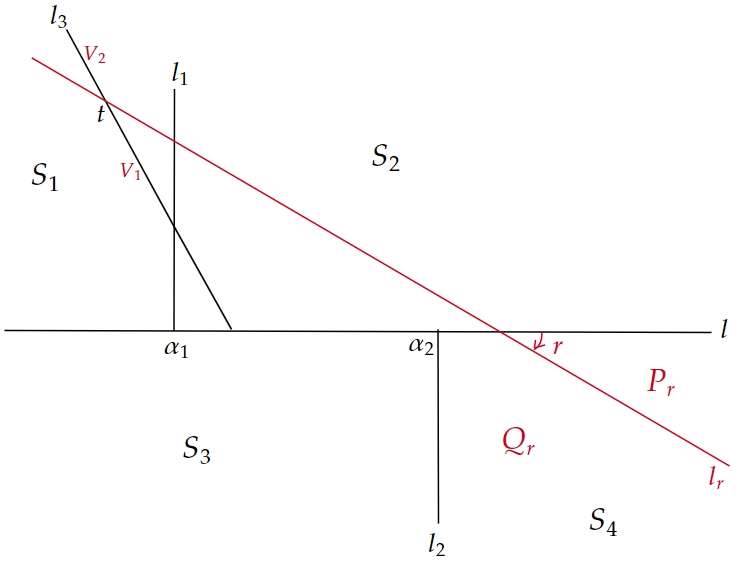}
		\caption{}
		\label{picture7}
	\end{minipage}
\end{figure}

In what follows, we always assume that $\alpha_1\neq \alpha_2$ since Theorem \ref{divideS} clearly holds when $\alpha_1=\alpha_2$. In fact, when $\alpha_1=\alpha_2$, since $l_1\perp l$, we let 
\begin{equation*}
	L_1=l_1,\quad L_2=l,
\end{equation*}
and
\begin{equation*}
	T_j=S_j,\quad j\in\{1,2,3,4\}.
\end{equation*}
Then $L_1$ and $L_2$ are what we need. Besides, we only need to consider the case where $\alpha_1$ is on the left of $\alpha_2$ on the line $l$ as $\alpha_1$ and $\alpha_2$ are symmetric. 

Next, we consider the following particular case which will simplify the proof of Theorem \ref{divideS}.	
\begin{lem}\label{middlecases}
	 Let $S_1$, $S_2$, $S_3$ and $S_4$ be four closed quadrants as in Lemma \ref{S1S2S3S4S}. If there is a ray $l_3$ with origin in $l$ such that 
	\begin{equation*}
		\mu(\{x\in I:b(x)\in l_3\cap S_1\})\ge \frac{1}{2}\mu(\{x\in I:b(x)\in S_1\})
	\end{equation*}
	\begin{equation*}
		(\text{or}\quad	\mu(\{x\in I:b(x)\in l_3\cap S_4\})\ge \frac{1}{2}\mu(\{x\in I:b(x)\in S_4\}),\,\,)
	\end{equation*}
	then there exist two orthogonal lines $l_4$ and $l_5$ such that $l_4$ and $l_5$ divide $\mathbb{C}$ into four closed quadrants $T_1$, $T_2$, $T_3$, $T_4$ (see Figure \ref{picture3}). Moreover,
	\begin{equation*}
		\mu(\{x\in I:b(x)\in T_i\})\ge \frac{1}{16}\mu(I),\quad i\in \{1,2,3,4\}.
	\end{equation*}
\end{lem}

\begin{proof}
	We only consider the case that $l_3\cap S_1\neq \emptyset$. The proof of the case that $l_3\cap S_4\neq \emptyset$ is the same. From the proof of Lemma \ref{S1S2S}, there exists a point $t\in l_3$, such that $t$ divides $l_3\cap S_1$ into two parts $V_{1}$ and $V_{2}$, such that
	\begin{equation*}
		\mu(\{x\in I:b(x)\in V_i\})\ge \frac{1}{2}\mu(\{x\in I:b(x)\in l_{3}\cap S_1\})\ge \frac{1}{16}\mu(I),\quad i\in\{1,2\}.
	\end{equation*}
	Without loss of generality, we assume that $l$ is the real axis. Then by our assumption that $\alpha_1$ is on the left of $\alpha_2$ on the line $l$, we see $\alpha_1< \alpha_2$. 
	
	Let $l_r$ be the line passing through $t$, such that the clockwise angle from $l$ to $l_r$ is $r$, where $-\frac{\pi}{2}< r \le \frac{\pi}{2}$ (see Figure \ref{picture7}). The line $l_r$ divides $S_4$ into two closed parts, and denote by $P_r$ the upper part and by $Q_r$ the lower part (see Figure \ref{picture7}). Let
	\begin{equation*}
		f(r)=
		\begin{cases}
			\mu(\{x\in I:b(x)\in P_r\})\big/\mu(\{x\in I:b(x)\in S_4\}),&  0< r \le \frac{\pi}{2},\\
			0, & -\frac{\pi}{2}<r< 0,
		\end{cases}
	\end{equation*}
	and
	\begin{equation*}
		f(0)=\begin{cases}
			\mu(\{x\in I:b(x)\in P_0\})\big/\mu(\{x\in I:b(x)\in S_4\}), &\mathrm{Im}(t)=0,\\
			0, &\mathrm{Im}(t)>0.
			\end{cases}
		\end{equation*}
	It is not hard to check that:\\
     $\bullet$ $f$ is increasing,\\
     $\bullet$ $f$ is right-continuous,\\
	 $\bullet$ $f(\frac{\pi}{2})=1$,

	Define 
	\begin{equation*}
		\beta=\inf \Big\{r\in (-\frac{\pi}{2},\frac{\pi}{2}]: f(r)\ge \frac{1}{2}\Big\}.
	\end{equation*}
	So $0\leq \beta<\frac{\pi}{2}$. On the one hand, since $f(r)$ is right-continuous, we have 
	\begin{equation}\label{S112}
		f(\beta)\ge \frac{1}{2}.
	\end{equation}
	On the other hand, note that for any $0<\varepsilon<\beta+\frac{\pi}{2}$, $f(\beta-\varepsilon)<\frac{1}{2}$. Namely,
	\begin{equation*}
		\mu(\{x\in I:b(x)\in P_{\beta-\varepsilon}\})<\frac{1}{2}\mu(\{x\in I:b(x)\in S_4\}).
	\end{equation*} 
	Let $\varepsilon\to 0$, then
	\begin{equation*}
		\mu(\{x\in I:b(x)\in P_\beta \backslash l_\beta\})\le \frac{1}{2}\mu(\{x\in I:b(x)\in S_4\}).
	\end{equation*}
Thus this implies
	\begin{equation}\label{S212}
		\mu(\{x\in I:b(x)\in Q_\beta \})\ge \frac{1}{2}\mu(\{x\in I:b(x)\in S_4\}).
	\end{equation}
	Let $l_4=l_\beta$, $S_{41}=P_\beta$ and $S_{42}=Q_\beta$. From \eqref{S112} and \eqref{S212}, $l_4$ passes through $t$ and divides $S_4$ into two closed parts $S_{41}$ and $S_{42}$, such that 
	\begin{equation*}
		\mu(\{x\in I:b(x)\in S_{4i}\})\ge \frac{1}{2}\mu(\{x\in I:b(x)\in S_{4}\})\ge \frac{1}{8}\mu(I),\quad i\in\{1,2\}.
	\end{equation*}
	Now if $\mathrm{Im}(t)>0$, then $\beta\neq 0$ as $t\in l_4$. Suppose $c$ is the intersection point of $l_4$ and $l$ (see Figure \ref{picture3}). Besides, for any given $\tilde{c}\in (\alpha_1,\alpha_2)$, we can find another line $l_5$ passing through $\tilde{c}$, such that $l_5\perp l_4$. Then  $l_4$ and $l_5$ divide $\mathbb{C}$ into four closed quadrants $T_1$, $T_2$, $T_3$, $T_4$. Moreover, we have
	\begin{equation*}
		\begin{aligned}
			{}&\mu(\{x\in I:b(x)\in T_1\})\ge \mu(\{x\in I:b(x)\in V_2\})\ge \frac{1}{16}\mu(I),\\
			&\mu(\{x\in I:b(x)\in T_2\})\ge \mu(\{x\in I:b(x)\in S_{41}\})\ge \frac{1}{8}\mu(I),\\
			&\mu(\{x\in I:b(x)\in T_3\})\ge \mu(\{x\in I:b(x)\in V_1\})\ge \frac{1}{16}\mu(I),\\
			&\mu(\{x\in I:b(x)\in T_4\})\ge \mu(\{x\in I:b(x)\in S_{42}\})\ge \frac{1}{8}\mu(I).			
		\end{aligned}
	\end{equation*}
	If $\mathrm{Im}(t)=0$ and $\beta>0$, then Lemma \ref{middlecases} follows in a similar way. Finally, if $\mathrm{Im}(t)=0$ and $\beta=0$, then $l_4=l$ and choose $l_5$ to be the line through the point $t$ and orthogonal to $l$, and this proves Lemma \ref{middlecases}.
\end{proof}

\begin{figure}[H] 
	\centering 
	\includegraphics[width=0.6\textwidth]{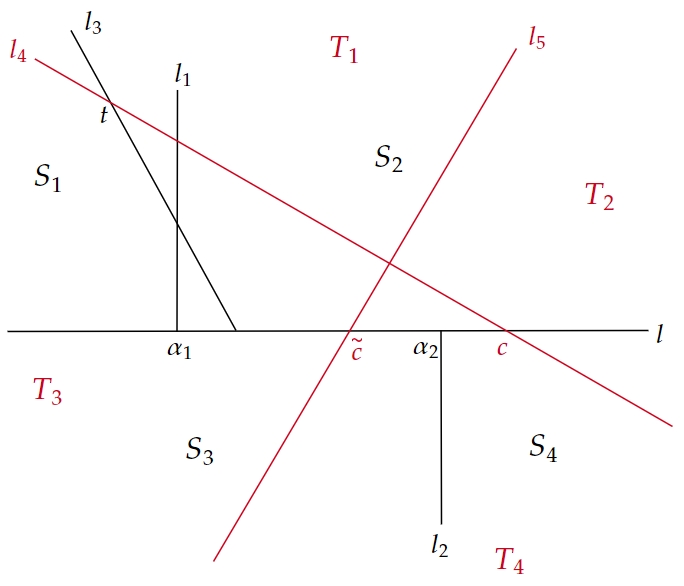} 
	\caption{} 
	\label{picture3} 
\end{figure}

\begin{rem}
	Indeed, we can find a line $l_6$ which passes through the point $t$ and is perpendicular to $l_3$, and then $l_3$ and $l_6$ satisfy Lemma \ref{middlecases}. This is an easy proof of Lemma \ref{middlecases}. But we still keep the previous complicated proof of Lemma \ref{middlecases} since the argument of this complicated proof plays a vital role in the later proof of Theorem \ref{divideS}.
	
	\end{rem}

Now we begin to prove Theorem \ref{divideS}.	
\begin{proof}[Proof of Theorem \ref{divideS}]
	By Lemma \ref{S1S2S3S4S}, there exist a line $l$ and two rays $l_1,l_2$ in $\mathbb{C}$ satisfying $l\perp l_1$ and $l\perp l_2$, and they divide $\mathbb{C}$ into four closed quadrants $S_1$, $S_2$, $S_3$ and $S_4$. Moreover,
	\begin{equation}\label{SiS4}
		\mu(\{x\in I:b(x)\in S_i\})\ge \frac{1}{4}\mu(I),\quad i\in \{1,2,3,4\}.
	\end{equation}
	We denote by $\alpha_1$, $\alpha_2$ the origins of $l_1$ and $l_2$ respectively.
	
	\smallskip
	
	By rotating and translating the axes, we assume that $l$ is the real axis, and $\alpha_1<\alpha_2$. Besides, by the proof of Lemma \ref{S1S2S}, there exists a point $A\le \alpha_1$ in $l$ and a ray $l_A\perp l$, whose origin is $A$, such that $l_A$ divides $S_1$ into two closed parts $R_{1}$ and $R_{2}$ (see Figure \ref{picture4}). Moreover, 
	\begin{equation}\label{muRiS1}
		\mu(\{x\in I:b(x)\in R_{i}\})\ge \frac{1}{2}\mu(\{x\in I:b(x)\in S_1\}),\quad i\in\{1,2\}.
	\end{equation}
	Similarly, there exists a point $B\ge \alpha_2$ in $l$ and a ray $l_B\perp l$, whose origin is $B$, such that $l_B$ divides $S_4$ into two closed parts $R_3$ and $R_4$. Moreover, 
	\begin{equation*}
		\mu(\{x\in I:b(x)\in R_{i}\})\ge \frac{1}{2}\mu(\{x\in I:b(x)\in S_4\}),\quad i\in\{3,4\}.
	\end{equation*}
	Now for any point $A\le x\le B$ and any angle $r\in[0,\pi]$, we
	let $l(x,r)\subset S_1\cup S_2$ be the ray whose origin is $x$, such that the clockwise angle from $l$ to $l(x,r)$ is $r$. Then $l(x, r)$ divides $S_1$ into two closed parts, and denote by $Q_{11}(x,r)$ the lower part and by $Q_{12}(x,r)$ the upper part (see Figure \ref{picture5}).
	
	Similarly, for any point $A\le {x}\le B$ and any angle $\tilde{r}\in[0,\pi]$, let $\tilde{l}({x},\tilde{r})\subset S_3\cup S_4$ be the ray whose origin is ${x}$, such that the clockwise angle from $l$ to $\tilde{l}({x},\tilde{r})$ is $\tilde{r}$. Then $\tilde{l}({x},\tilde{r})$ divides $S_4$ into two closed parts, and denote by $Q_{42}({x},\tilde{r})$ the lower part and by $Q_{41}({x},\tilde{r})$ the upper part (see Figure \ref{picture5}).
	
	\begin{figure}[H]
		\centering 
		\begin{minipage}[b]{0.45\textwidth} 
			\centering 
			\includegraphics[width=1.0\textwidth]{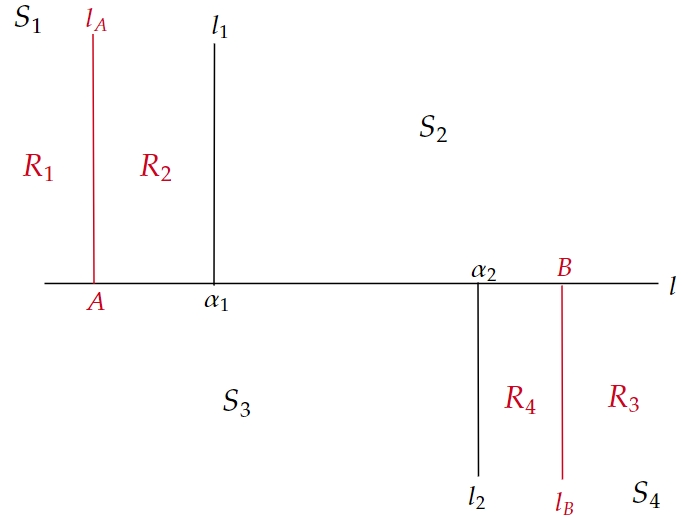} 
			\caption{}
			\label{picture4}
		\end{minipage}
		\begin{minipage}[b]{0.45\textwidth} 
			\centering 
			\includegraphics[width=1.0\textwidth]{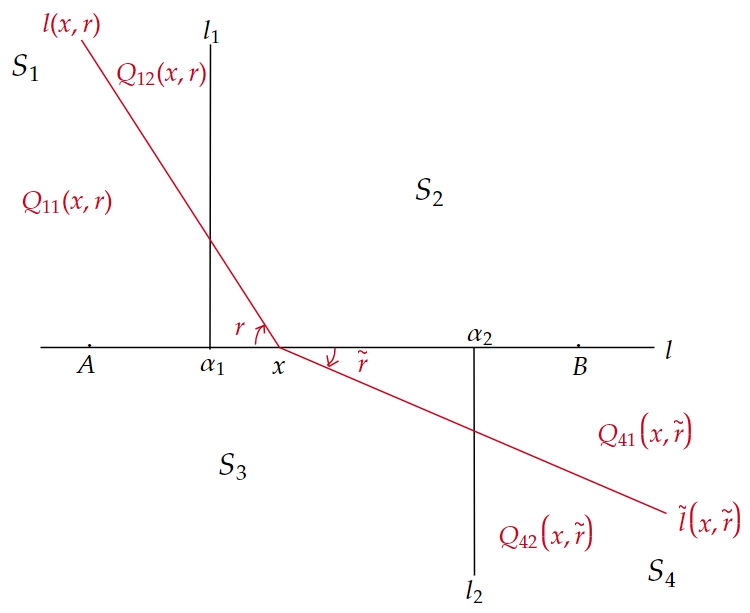}
			\caption{}
			\label{picture5}
		\end{minipage}
	\end{figure}

	Then let
	\begin{equation*}
		f(x,r)=\frac{\mu(\{y\in I:b(y)\in Q_{11}(x,r)\})}{\mu(\{y\in I:b(y)\in S_1\})},\quad g(x,r)=\frac{\mu(\{y\in I:b(y)\in Q_{12}(x,r)\})}{\mu(\{y\in I:b(y)\in S_1\})}
	\end{equation*}
	and
	\begin{equation*}
		\tilde{f}(x,r)=\frac{\mu(\{y\in I:b(y)\in Q_{41}(x,r)\})}{\mu(\{y\in I:b(y)\in S_4\})},\quad \tilde{g}(x,r)=\frac{\mu(\{y\in I:b(y)\in Q_{42}(x,r)\})}{\mu(\{y\in I:b(y)\in S_4\})}.
	\end{equation*}
	It is not hard to check that:\\
	$\bullet$ when the point $x$ is fixed, in terms of the angle $r$, $f$ and $\tilde{f}$ are both increasing and right-continuous functions, $g$ and $\tilde{g}$ are both decreasing and left-continuous functions,\\
	$\bullet$ when the angle $r$ is fixed, in terms of the point $x$, $f$ and $\tilde{g}$ are both increasing and right-continuous functions, $g$ and $\tilde{f}$ are both decreasing and left-continuous functions.

	Besides, define 
	\begin{equation*}
		\begin{aligned}
			r_1(x){}&=\inf\Big\{r:f(x,r)\ge \frac{1}{4}, \ g(x,r)\ge \frac{1}{4}\Big\},\\
			r_2(x)&=\sup\Big\{r:f(x,r)\ge \frac{1}{4}, \ g(x,r)\ge \frac{1}{4}\Big\},\\
			r_3(x)&=\inf\Big\{r:\tilde{f}(x,r)\ge \frac{1}{4}, \  \tilde{g}(x,r)\ge \frac{1}{4}\Big\},\\
			r_4(x)&=\sup\Big\{r:\tilde{f}(x,r)\ge \frac{1}{4}, \  \tilde{g}(x,r)\ge \frac{1}{4}\Big\}.
		\end{aligned}
	\end{equation*}
Note that each $r_i$ is well-defined, since $\Big\{r:f(x,r)\ge \frac{1}{4}, \ g(x,r)\ge \frac{1}{4}\Big\}$ and $\Big\{r:\tilde{f}(x,r)\ge \frac{1}{4}, \  \tilde{g}(x,r)\ge \frac{1}{4}\Big\}$ are not empty by using the same way as in Lemma \ref{S1S2S}. Indeed, these two nonempty sets are even closed intervals. In fact, if $r_1(x)<r_2(x)$, for any $x\in [A,B]$, by the definition of $r_1(x)$, there exists a positive sequence $\{\varepsilon_n\}_{n\ge 1}$ with $\varepsilon_n\to 0$, such that $f(x,r_1(x)+\varepsilon_n)\ge \frac{1}{4}$. Since $f$ is right-continuous with respect to the angle, we then deduce $$f(x,r_1(x))\ge \frac{1}{4}.$$
	Since $r_1(x)< r_2(x)$, one has $$f(x,r_2(x))\ge \frac{1}{4}.$$
	Similarly, we derive that 
	$$g(x,r_2(x))\ge \frac{1}{4} \quad \text{and} \quad g(x,r_1(x))\ge \frac{1}{4}.$$
	Thus for any $x\in [A,B]$,
	\begin{equation}\label{fgr2r1}
		\Big\{r:f(x,r)\ge \frac{1}{4}, \ g(x,r)\ge \frac{1}{4}\Big\}=[r_1(x),r_2(x)].
	\end{equation}
    If $r_1(x)=r_2(x)$, then \eqref{fgr2r1} is trivial.
	Similarly, for any $x\in [A,B]$,
	\begin{equation}\label{fgr3r4}
		\Big\{r:\tilde{f}(x,r)\ge \frac{1}{4}, \ \tilde{g}(x,r)\ge \frac{1}{4}\Big\}=[r_3(x),r_4(x)].
	\end{equation}

	 Furthermore, in the later proof we will always assume that $r_2,r_4< \pi$ and $r_1,r_3>0$ (because the case $r_2=\pi$ or $r_4=\pi$ or $r_1=0$ or $r_3=0$ is the special case as in Lemma \ref{middlecases}, and we omit the details).
	For any given $A\le x_1\le x_2\le B$, since $g$ is decreasing with respect to the angle, by the definition of $r_2(x_1)$, for any $0<\varepsilon<\pi-r_2(x_1)$, we have $g(x_1,r_2(x_1)+\varepsilon)<\frac{1}{4}$. Besides, note that $g$ is decreasing with respect to $x$, then
	\begin{equation*}
		g(x_2,r_2(x_1)+\varepsilon)\le g(x_1,r_2(x_1)+\varepsilon)<\frac{1}{4},
	\end{equation*}
	this implies that $r_2(x_2)\le r_2(x_1)+\varepsilon$. Then by letting $\varepsilon\to 0$, one has $r_2(x_2)\le r_2(x_1)$. Thus $r_2$ is decreasing. Similarly, we obtain:\\
	$\bullet$ $r_1$, $r_2$ are decreasing,\\
	$\bullet$ $r_3$, $r_4$ are increasing.

	\smallskip
	
	In the following we will divide the remaining of the proof into three cases.
	
	\subsubsection*{Case 1: There exists $x_0\in [A,B]$, such that $r_1(x_0)=r_2(x_0)$}	
	Since $f$ is increasing with respect to the angle, by the definition of $r_1(x_0)$, for any $0<\varepsilon<r_1(x_0)$, $f(x_0,r_1(x_0)-\varepsilon)<\frac{1}{4}$. Namely,
	\begin{equation*}
		\mu(\{x\in I:b(x)\in Q_{11}(x_0,r_1(x_0)-\varepsilon)\})<\frac{1}{4}\mu(\{x\in I:b(x)\in S_1\}).
	\end{equation*} 
	Let $\varepsilon\to 0$, then 
	\begin{equation}\label{Q11S1}
		\mu\big(\big\{x\in I:b(x)\in Q_{11}(x_0,r_1(x_0))\big\backslash l(x_0,r_1(x_0))\big\}\big)\le \frac{1}{4}\mu(\{x\in I:b(x)\in S_1\}).
	\end{equation} 
	Similarly, we get
	\begin{equation}\label{Q12S1}
		\mu\big(\big\{x\in I:b(x)\in Q_{12}(x_0,r_2(x_0))\big\backslash l(x_0,r_2(x_0))\big\}\big)\le \frac{1}{4}\mu(\{x\in I:b(x)\in S_1\}).
	\end{equation} 
	Thus from \eqref{Q11S1} and \eqref{Q12S1} one has
	\begin{equation*}
		\mu(\{x\in I:b(x)\in l(x_0,r_1(x_0))\cap S_1 \})\ge \frac{1}{2}\mu(\{x\in I:b(x)\in S_1\}).
	\end{equation*} 
	Hence the desired result is obtained by Lemma \ref{middlecases}.
	
	\bigskip
	
	\subsubsection*{Case 2: There exists $x_0\in [A,B]$, such that $r_3(x_0)=r_4(x_0)$}			
	The proof is the same as in Case 1.
	
	\bigskip
	
	\subsubsection*{Case 3: For any $x\in [A,B]$, $r_1(x)<r_2(x)$ and $r_3(x)<r_4(x)$}
	Our aim is to prove that there exists $y\in [A,B]$, such that
	\begin{equation}\label{achieve}
		\Big\{r:f(y,r)\ge \frac{1}{4}, \ g(y,r)\ge \frac{1}{4}\Big\}\cap \Big\{r:\tilde{f}(y,r)\ge \frac{1}{4}, \  \tilde{g}(y,r)\ge \frac{1}{4}\Big\}\neq \emptyset.
	\end{equation}

	\smallskip
	
	We need the following definition: for any $i\in\{1,2,3,4\}$, define
	\begin{equation*}
		r_i(c-)=\lim_{x\nearrow c}r_i(x),\quad \forall c\in (A,B],
	\end{equation*}
    and
	\begin{equation*}
		r_i(c+)=\lim_{x\searrow c}r_i(x),\quad \forall c\in [A,B).
	\end{equation*}
	Note that $r_i$ is always monotone, thus the above definition makes sense. 
	
	\smallskip
	
	Now we show the following important properties:\\
	$\bullet$ For any $c\in (A,B]$, one has 
	\begin{equation}\label{property1}
		r_1(c-)\le r_2(c)
	\end{equation}
     and 
     \begin{equation}\label{property3}
     	r_3(c)\le r_4(c-),
     \end{equation}
     $\bullet$ For any $c\in [A,B)$, one has
     \begin{equation}\label{property2}
     	r_1(c)\le r_2(c+)
     \end{equation}
     and 
     \begin{equation}\label{property4}
     	r_3(c+)\le r_4(c).
     \end{equation}

	We only prove \eqref{property1}. If $r_1(c-)>r_2(c)$, then for any given $a$ satisfying
	\begin{equation*}
		r_1(c-)>a>r_2(c),
	\end{equation*}
	by the definition of $r_1(c-)$, there exists a sequence $\{x_n\}$ satisfying $x_n\nearrow c$ such that
	\begin{equation*}
		r_1(x_n)>a>r_2(c).
	\end{equation*}
	Since $g$ is decreasing with respect to the angle, by the definition of $r_1(x_n)$, we have $g(x_n,a)\ge \frac{1}{4}$. Then from the fact that $g$ is left-continuous with respect to $x$, one has
	\begin{equation*}
		g(c,a)\ge \frac{1}{4}.
	\end{equation*}
	Besides, since $f$ is increasing with respect to the angle, we have
	\begin{equation*}
		f(c,a)\ge f(c,r_2(c))\ge \frac{1}{4}.
	\end{equation*}
	Thus by the definition of $r_2(c)$, one gets $a\le r_2(c)$. It contradicts $a>r_2(c)$. 
	
	\smallskip

	\bigskip
	
	Now we come back to the proof of \eqref{achieve}. By \eqref{fgr2r1} and \eqref{fgr3r4}, \eqref{achieve} is equivalent to proving that there exists $y\in [A,B]$, such that
	\begin{equation}\label{achieve2}
		[r_1(y),r_2(y)]\cap [r_3(y),r_4(y)]\neq \emptyset. 
	\end{equation}

    \smallskip
    
	If $r_1(x)\le r_4(x)$ for all $x\in [A,B]$, since $l_A$ divides $S_1$ into two closed parts $R_{1}$ and $R_{2}$ satisfying \eqref{muRiS1}, we have $r_2(A)\ge \frac{\pi}{2}$. Besides, note that $A\le \alpha_2$, which implies that $r_3(A)\le \frac{\pi}{2}$. Thus we only have the following four cases:\\
	$\bullet$ $r_3(A)\le r_1(A)\le r_4(A)\le r_2(A)$,\\
	$\bullet$ $r_3(A)\le r_1(A)\le r_2(A)\le r_4(A)$,\\
	$\bullet$ $r_1(A)\le r_3(A)\le r_4(A)\le r_2(A)$,\\
	$\bullet$ $r_1(A)\le r_3(A)\le r_2(A)\le r_4(A)$.\\
	Hence
	\begin{equation*}
		[r_1(A),r_2(A)]\cap [r_3(A),r_4(A)]\neq \emptyset.
	\end{equation*}
    and \eqref{achieve2} is obtained by letting $y=A$.
    
	Now suppose that there exists $x\in [A, B]$ such that $r_1(x)>r_4(x)$.  Then we define 
	\begin{equation*}
		c_0=\sup\{x\in [A,B]:r_1(x)-r_4(x)>0\}.
	\end{equation*}
	We will consider three cases according to the value of $c_0$.
	
	\subsubsection*{Subcase 3.1:  $A<c_0<B$} 
	From the definition of $c_0$ we know that
	\begin{equation}\label{r2r3c0-c0+}
		r_1(c_0-)\ge r_4(c_0-),\quad \text{and}\quad r_1(c_0+)\le r_4(c_0+).
	\end{equation} 
	Besides, for any $c_0<x\le B$, 
	\begin{equation*}
		r_1(x)\le r_4(x).
	\end{equation*}
	There are three subcases in this situation:\\
	(1) if there exists $c_0<x_0\le B$, such that $r_1(x_0)=r_4(x_0)$, then \eqref{achieve2} is obvious by letting $y=x_0$,\\
	(2) if for all $c_0<x\le B$,
		\begin{equation*}
			r_1(x)<r_4(x),
		\end{equation*} 
		but there exists $c_0<x_0\le B$, such that $r_2(x_0)\ge r_3(x_0)$, then \eqref{achieve2} is derived by letting $y=x_0$,\\
	(3) if for any $c_0<x\le B$,
		\begin{equation*}
			r_1(x)<r_4(x)\quad \text{and} \quad r_2(x)<r_3(x),
		\end{equation*} 
		this implies that
		\begin{equation*}
			r_2(c_0+)\le r_3(c_0+).
		\end{equation*}
		On the one hand, from \eqref{property2} and \eqref{property4} one has
		\begin{equation*}
			r_1(c_0)\le r_2(c_0+)\le r_3(c_0+)\le r_4(c_0).
		\end{equation*}
		On the other hand, from \eqref{property1}, \eqref{property3} and \eqref{r2r3c0-c0+} we get
		\begin{equation*}
			r_3(c_0)\le r_4(c_0-)\le r_1(c_0-)\le r_2(c_0).
		\end{equation*}
		Hence \eqref{achieve2} is obtained by letting $y=c_0$.

	\smallskip
	
	\subsubsection*{Subcase 3.2: $c_0=B$} 
	This implies that 
	\begin{equation*}
		r_1(B-)\ge r_4(B-).
	\end{equation*}
	From \eqref{property1} and \eqref{property3} we know that
	\begin{equation*}
		r_3(B)\leq r_4(B-)\le r_1(B-)\leq r_2(B).
	\end{equation*}
	Note that
	\begin{equation*}
		r_1(B)\le \frac{\pi}{2}\le r_4(B).
	\end{equation*}
	Hence \eqref{achieve2} is derived by letting $y=B$.
	
	\smallskip
	
	\subsubsection*{Subcase 3.3:  $c_0=A$} 
	This implies that for any $A<x\le B$,
	\begin{equation*}
		r_1(x)\le r_4(x).
	\end{equation*}
	We need to consider the following three subcases:\\
	(1) if there exists $A<x_0\le B$, such that $r_2(x_0)\ge r_4(x_0)$, then
		\begin{equation*}
			r_1(x_0)\le r_4(x_0)\le r_2(x_0).
		\end{equation*}
		Thus \eqref{achieve2} is obtained by letting $y=x_0$,\\
	(2) if for all $A<x\le B$, 
		\begin{equation*}
			r_2(x)<r_4(x),
		\end{equation*}
		but there exists $A<x_0\le B$, such that $r_2(x_0)\ge r_3(x_0)$, then \eqref{achieve2} is derived by letting $y=x_0$,\\
	(3) if for any $A<x\le B$, 
		\begin{equation*}
			r_2(x)<r_4(x)\quad \text{and} \quad r_2(x)<r_3(x),
		\end{equation*}
		this implies that 
		\begin{equation*}
			r_2(A+)\le r_3(A+).
		\end{equation*}
		Thus from \eqref{property2} and \eqref{property4} we have
		\begin{equation*}
			r_1(A)\leq r_2(A+)\le r_3(A+)\leq r_4(A).
		\end{equation*}
		Note that 
		\begin{equation*}
			r_3(A)\le \frac{\pi}{2}\le r_2(A).
		\end{equation*}
		Hence \eqref{achieve2} is obtained by letting $y=A$.

	Now we have proved \eqref{achieve2}. Thus we choose $r_0(y)\in [r_1(y),r_2(y)]\cap [r_3(y),r_4(y)]\neq \emptyset$, and let the line $L_1$ be the extension of the ray $l({y,r_0(y)})$. Besides, for any given $\tilde{y}\in (\alpha_1,\alpha_2)$, let $L_2$ be the line passing through $\tilde{y}$, such that $L_1\perp L_2$. Then $L_1$ and $L_2$ divide $\mathbb{C}$ into four closed quadrants $T_1$, $T_2$, $T_3$, $T_4$ (see Figure \ref{picture6}). Moreover, from \eqref{SiS4} we have
	\begin{equation*}
		\mu(\{x\in I:b(x)\in T_i\})\ge \frac{1}{4}\mu(\{x\in I:b(x)\in S_i\})\ge \frac{1}{16}\mu(I),\quad i\in \{1,2,3,4\}.
	\end{equation*}
\end{proof}

\begin{figure}[H] 
	\centering 
	\includegraphics[width=0.6\textwidth]{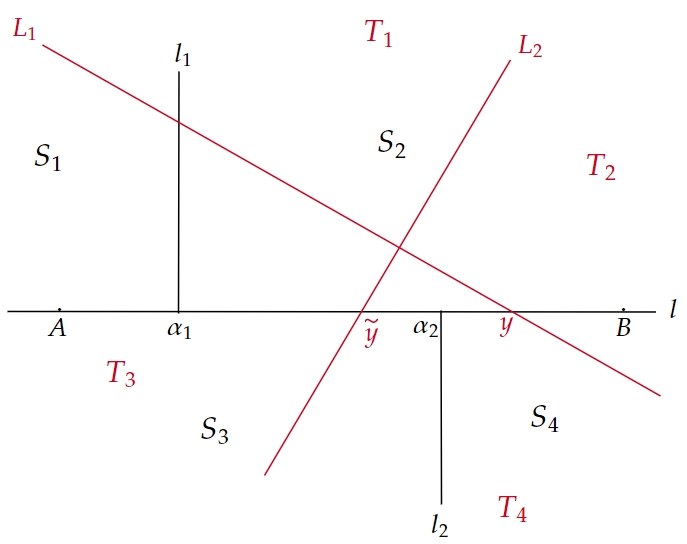} 
	\caption{} 
	\label{picture6} 
\end{figure}

\begin{rem}
	One would try to show Theorem \ref{divideS} first for simple functions, and then use a limit argument for general functions. However, it seems that the proof for simple functions is already complicated. In addition, we do not know how to deal with the limit argument either.
	\end{rem}

\bigskip

\section{Proof of the necessity of Theorem \ref{corollary1.8}}\label{schattenconv}

In order to prove the necessity of Theorem \ref{corollary1.8}, we need to describe the Besov space  $\pmb{B}_p(\mathbb{R}^n)$ in terms of the Schatten class membership of commutators involving singular integral operators. In this section, we deal with singular integral operators associated with non-degenerate kernels in Hyt\"{o}nen's sense. We refer the reader to \cite{TH3} for more details about non-degenerate kernels. At first, we show that when $1<p<\8$, $[T, M_b]\in  S_p(L_2(\mathbb{R}^n))$ implies $b\in \pmb{B}^{\omega,2^n}_p(\mathbb{R}^n)$ (see Lemma \ref{conv1}), where $\pmb{B}^{\omega,2^n}_p(\mathbb{R}^n)$ is defined in \eqref{Bpw2nrnm}. Our proof is based on the complex median method, i.e. Theorem \ref{divideS}. Then we show that for $1<p<\8$, $\pmb{B}_p(\mathbb{R}^n)$ is the intersection of several dyadic martingale Besov spaces associated with different translated dyadic systems (see Proposition \ref{pdayun}). This enables us to transfer the martingale setting to the Euclidean setting. We start with non-degenerate kernels.

\subsection{Non-degenerate kernels}
We first give the definition of non-degeneracy of kernels.
\begin{definition}\label{nondege}
	Let $T\in B(L_2(\mathbb{R}^n))$ be a singular integral operator with kernel $K(x,y)$ satisfying standard kernel estimates \eqref{standard}. $K$ is called non-degenerate, if one of the following conditions holds:
	
	(a) for every $y\in \mathbb{R}^n$ and $r>0$, there exists $x\in B(y,r)^c$ such that
	\begin{equation}\label{nondege1}
		|K(x,y)|\ge \frac{1}{c_0r^n},
	\end{equation}
	where $c_0$ is a fixed positive constant.
	
	(b) if $K$ is a homogeneous kernel with
	\begin{equation}\label{nondege2}
		K(x,y)=\frac{\Omega(x-y)}{|x-y|^n},
	\end{equation}
	where $\Omega\in L_1(\mathbb{S}^{n-1})\backslash\{0\}$ and $\Omega(tx)=\Omega(x)$ for all $t>0$ and $x\in \mathbb{R}^n$ (here $\mathbb{S}^{n-1}$ is the sphere of $\mathbb{R}^n$), then there exists a Lebesgue point $\theta_0\in \mathbb{S}^{n-1}$ of $\Omega$ such that
	\begin{equation*}
		\Omega(\theta_0)\neq 0.
	\end{equation*}
	
\end{definition}	

\begin{remark}
	Let $T$ be a Calder\'{o}n-Zygmund transform with the convolution kernel $K(x,y)=\phi(x-y)$. Then the non-degeneracy condition (a) in Definition \ref{nondege} is simplified into the following form: for every $r>0$, there exists $x\in B(0,r)^c$ with
	\begin{equation*}
		|\phi(x)|\ge \frac{1}{c_0r^n},
	\end{equation*}
	We refer to \cite{TH3} for more details.
\end{remark}

In the rest of this section, the kernel $K$ will be always assumed to satisfy \eqref{standard} and to be non-degenerate. From the above definition, we obtain the following property for $K$.
\begin{lem}\label{ball}
	For every $A\ge 3$ and every ball $B=B(x_0,r)$, there is a disjoint ball $\tilde{B}=B(y_0,r)$ at distance $\mathrm{dist}(B,\tilde{B})\approx Ar$ such that
	\begin{equation}\label{kx0y0an}
		|K(y_0,x_0)|\approx \frac{1}{A^nr^n},
	\end{equation}
	and for all $x_1\in B$ and $y_1\in \tilde{B}$, 
	\begin{equation}\label{kx1y1x0y0}
		|K(y_1,x_1)-K(y_0,x_0)|\lesssim\frac{1}{A^{n+\alpha}r^n}.
	\end{equation}
	Furthermore, if $K(y_0,x_0)$ is real and $A$ is sufficiently large,
	then there exists a positive number $ \varrho$ which depends on $A$, $\alpha$ and $n$ and is much less than $1$ such that
	\begin{equation}\label{88888}
		|\mathrm{Im}(K(y_1,x_1))|\le \varrho\mathrm{Re}(K(y_1,x_1))
	\quad\text{and}\quad
		|K(y_1,x_1)|\le 2\mathrm{Re}(K(y_1,x_1))
	\end{equation}
	for any $x_1\in B$ and $y_1\in \tilde{B}$, where $\mathrm{Re}\big(K(y_1,x_1)\big)$ and $\mathrm{Im}\big(K(y_1,x_1)\big)$ are the real and imaginary parts of $K(y_1,x_1)$ respectively.
\end{lem}

\begin{proof}
	(1) Assume that $K$ is as in Definition \ref{nondege} (a). For a fixed ball $B=B(x_0,r)$ and $A\ge 3$, thanks to the standard estimate of $K$ and \eqref{nondege1}, there exists a point $y_0\in B(x_0,Ar)^c$ such that
	\begin{equation*}
		\frac{1}{c_0(Ar)^n}\le |K(y_0,x_0)|\le \frac{C}{|x_0-y_0|^n}.
	\end{equation*}
	This implies that 
	\begin{equation*}
		Ar\le |x_0-y_0|\le (c_0C)^{\frac{1}{n}}Ar,\quad |K(y_0,x_0)|\approx \frac{1}{A^nr^n}.
	\end{equation*}
	Let $\tilde{B}=B(y_0,r)$. Since $A\ge 3$, we have $\mathrm{dist}(B,\tilde{B})\approx |x_0-y_0|$. Thus
	$\mathrm{dist}(B,\tilde{B})\approx Ar$.
	
	(2) Assume that $K$ is as in Definition \ref{nondege} (b). For a fixed ball $B=B(x_0,r)$ and $A\ge 3$, let $y_0=x_0+Ar\theta_0$ and $\tilde{B}=B(y_0,r)$. It is not hard to check that $\mathrm{dist}(B,\tilde{B})\approx Ar$ and
	\begin{equation*}
		|K(y_0,x_0)|=\frac{|\Omega(y_0-x_0)|}{|y_0-x_0|^n}=\frac{|\Omega(Ar\theta_0)|}{|Ar\theta_0|^n}=\frac{|\Omega(\theta_0)|}{(Ar)^n}\approx \frac{1}{(Ar)^n}.
	\end{equation*}
This finishes the proof of \eqref{kx0y0an}. 

Now we show \eqref{kx1y1x0y0} and \eqref{88888}. For $x_1\in B$ and $y_1\in \tilde{B}$, from \eqref{standard} we have  
	\begin{equation*}
		\begin{aligned}
			|K(y_1,x_1)-K(y_0,x_0)|{}&\le |K(y_1,x_1)-K(y_0,x_1)|+|K(y_0,x_1)-K(y_0,x_0)|     \\
			&\lesssim \frac{|y_1-y_0|^\alpha}{|y_1-x_1|^{n+\alpha}}+\frac{|x_1-x_0|^\alpha}{|x_1-y_0|^{n+\alpha}}\\
			&\lesssim \frac{r^{\alpha}}{(Ar)^{n+\alpha} }+\frac{r^{\alpha}}{(Ar)^{n+\alpha} }\approx \frac{1}{A^{n+\alpha}r^n}.
		\end{aligned}
	\end{equation*}
	When $A$ is sufficiently large, $K(y_1,x_1)$ will be very close to $K(y_0,x_0)$. Hence if $K(y_0,x_0)$ is real, we deduce that
	\begin{equation*}
		A^nr^n|\mathrm{Re}(K(y_1,x_1))-K(y_0,x_0)|\lesssim\frac{1}{A^{\alpha}}
	\end{equation*}
and
\begin{equation*}
	A^nr^n|\mathrm{Im}(K(y_1,x_1))|\lesssim\frac{1}{A^{\alpha}}.
\end{equation*}
Note that by \eqref{kx0y0an} that
\begin{equation*}
	A^nr^n|K(y_0,x_0)|\approx 1.
\end{equation*}	
Therefore, we deduce that
	\begin{equation*}
		|\mathrm{Im}(K(y_1,x_1))|\le \varrho\mathrm{Re}(K(y_1,x_1))
	\quad \text{and} \quad
		|K(y_1,x_1)|\le 2\mathrm{Re}(K(y_1,x_1))
	\end{equation*}
	for sufficiently small $\varrho$.
\end{proof}

\subsection{A key lemma}
This subsection establishes a key lemma for the proof of the necessity of Theorem \ref{corollary1.8}. We will need the following lemma due to Rochberg and Semmes in \cite{RSe}.
\begin{lemma}\label{RSNWO}
	Let $1<p<\infty$. Assume that $\{e_I\}_{I\in\mathcal{D}}$ and $\{f_I\}_{I\in\mathcal{D}}$ are function sequences in $L_2(\mathbb{R}^n)$ satisfying $\mathrm{supp}e_I, \,\mathrm{supp}f_I\subseteq I$ and
	$\|e_I\|_{\infty},\,\|f_I\|_{\infty}\le |I|^{-\frac{1}{2}}$. For any bounded compact operator $V$ on $L_2(\mathbb{R}^n)$, one has
	\begin{equation*}
		\begin{aligned}
			\sum_{I\in\mathcal{D}}\big|\langle e_I,V(f_I)\rangle\big|^p\lesssim_{n,p} \|V\|^p_{S_p(L_2(\mathbb{R}^n))}.
		\end{aligned}
	\end{equation*}
\end{lemma}

We discover another different proof of Lemma \ref{RSNWO} by virtue of martingale paraproducts and also extend it to the semicommutative setting (see \cite[Theorem 10.1]{WZ2024}).

Now we come to the following lemma, which is vital for the proof of the necessity of Theorem \ref{corollary1.8}. It describes the relation between $\|b\|_{\pmb{B}_p^{\omega,2^n}(\mathbb{R}^n)}$ and $\|[T,M_b]\|_{S_p(L_2(\mathbb{R}^n))}$.

\begin{lemma}\label{conv1}
	Let $1<p<\infty$. Suppose that $b$ is a locally integrable complex-valued function. If $[T,M_b]\in S_p(L_2(\mathbb{R}^n))$, then for any $\omega\in(\{0,1\}^n)^\mathbb{Z}$,
	\begin{equation*}
		\begin{aligned}
			\|b\|_{\pmb{B}_p^{\omega,2^n}(\mathbb{R}^n)}\lesssim_{n,p,T}\|[T,M_b]\|_{S_p(L_2(\mathbb{R}^n))}.
		\end{aligned}
	\end{equation*}	
\end{lemma}	

\begin{proof}
Without loss of generality, we assume that $\omega=0$. Recall that
	\begin{equation*}
		\begin{aligned}
			\|b\|^p_{\pmb{B}_p^{0,2^n}(\mathbb{R}^n)}{}&=\sum_{I\in\mathcal{D}^0}\sum_{i=1}^{2^n-1}|I|^{-\frac{p}{2}}|\langle h_I^i,b\rangle|^p.
		\end{aligned}
	\end{equation*}
	For any given $k\in \mathbb{Z}$ and $I\in \mathcal{D}_k^0$, let $c(I)$ be the center of $I$. Let $B=B(c(I),2^{-k}\sqrt{n})$, then $I\subset B$. From Lemma \ref{ball}, for any given $A$ which is much greater than $n$, there is a disjoint ball $\tilde{B}=B(y_0,2^{-k}\sqrt{n})$ at distance $\mathrm{dist}(B,\tilde{B})\approx_n 2^{-k}A$, such that
	\begin{equation}\label{minus1}
		|K(y_0,c(I))|\approx_{n,T} \frac{1}{A^n|I|}.
	\end{equation}
	Then we choose a cube $\hat{I}\in \mathcal{D}^0_k$ such that $\hat{I}\subset \tilde{B}$ and $y_0\in \hat{I}$. It is not hard to find that $\mathrm{dist}(I,\hat{I})\approx_n 2^{-k}A$. In the following, we will always assume that $A$ is a sufficiently large number.
	
	By Theorem \ref{divideS}, there exist $\theta={\theta}(\hat{I},b)\in [0,2\pi)$ and $\alpha_{\hat{I}}(b)\in\mathbb{C}$ such that
	if we denote
	\begin{equation}\label{FSI}
		\begin{aligned}
		F_s^I=\bigg\{{x}\in \hat{I}: -\frac{\pi}{4}+\frac{(s-1)\pi}{2}\le \mathrm{arg}\Big(e^{\mathrm{i}\theta}\big(\alpha_{\hat{I}}(b)-b({x})\big)\Big)\le -\frac{\pi}{4}+\frac{s\pi}{2} \quad \text{or}\quad b({x})=\alpha_{
		\hat{I}}(b)\bigg\},
	    \end{aligned}
	\end{equation}
	where $s\in\{1,2,3,4\}$, and $\mathrm{arg}(z)$ is the argument of a complex number $z$, then $|F_s^I|\ge \frac{1}{16}|\hat{I}|$, and $F_1^I \cup F_2^I \cup F_3^I \cup F_4^I=\hat{I}$.  Note that for any $1\le i\le 2^n-1$,
	\begin{equation*}
		\begin{aligned}
			|I|^{-\frac{1}{2}} |\langle h_I^i,b\rangle|{}
			&=|I|^{-\frac{1}{2}}\bigg|\int_I \overline{h_I^i(x)}b(x)dx\bigg|\\
			&=|I|^{-\frac{1}{2}}\bigg|\int_I \overline{h_I^i(x)}\big(b(x)-\alpha_{\hat{I}}(b)\big)dx\bigg|\\
			&\le \frac{1}{|I|}\int_I |b(x)-\alpha_{\hat{I}}(b)|dx\\
			&=\frac{1}{|I|}\sum_{q=1}^{2^n}\int_{I(q)} |b(x)-\alpha_{\hat{I}}(b)|dx,
		\end{aligned}
	\end{equation*}
    where $I(q)$ is the $q$-th subinterval of $I$.
	Similarly, we define
\begin{equation}\label{ESI}
		E_s^I=\bigg\{x\in I: -\frac{\pi}{4}+\frac{(s-1)\pi}{2}\le \mathrm{arg}\Big(e^{\mathrm{i}{\theta}}\big(b({x})-\alpha_{\hat{I}}(b)\big)\Big)\le -\frac{\pi}{4}+\frac{s\pi}{2} \quad \text{or}\quad b({x})=\alpha_{
			\hat{I}}(b)\bigg\},
	\end{equation}
	then for any $1\le i\le 2^n-1$,
	\begin{equation}\label{M1M2}
		\begin{aligned}
			|I|^{-\frac{1}{2}} |\langle h_I^i,b\rangle|{}
			&\le\frac{1}{|I|}\sum_{q=1}^{2^n}\int_{I(q)} |b(x)-\alpha_{\hat{I}}(b)|dx\\
			&\le \frac{1}{|I|}\sum_{q=1}^{2^n}\sum_{s=1}^4\int_{I(q)\cap E_s^I} |b(x)-\alpha_{\hat{I}}(b)|dx\\
			&=:\sum_{s=1}^4M_s^I.
		\end{aligned}
	\end{equation}
	Note that for any $s\in \{1,2,3,4\}$, if $x\in I(q)\cap E_s^I$ and $\hat{x}\in F_s^I$, then
	\begin{equation}\label{bxal}
		\begin{aligned}
			|b(x)-\alpha_{\hat{I}}(b)|{}&\le \big|e^{\mathrm{i}{\theta}}\big(b(x)-\alpha_{\hat{I}}(b)\big)\big|+\big|e^{\mathrm{i}{\theta}}\big(\alpha_{\hat{I}}(b)-b(\hat{x})\big)\big|\\
			&\le 2\big|e^{\mathrm{i}{\theta}}\big(b(x)-b(\hat{x})\big)\big|
			= 2|b(x)-b(\hat{x})|.
		\end{aligned}
	\end{equation}
	Thus by \eqref{minus1} and \eqref{bxal} one has
	\begin{equation*}
		\begin{aligned}
			M_s^I{}&\approx \frac{1}{|I|}\sum_{q=1}^{2^n}\int_{I(q)\cap E_s^I} |b(x)-\alpha_{\hat{I}}(b)|dx \cdot \frac{|F_s^I|}{|I|}\\
			&\approx_{n,T} \frac{A^n}{|I|}\sum_{q=1}^{2^n}\int_{I(q)\cap E_s^I} |b(x)-\alpha_{\hat{I}}(b)||K(y_0,c(I))||F_s^I|dx\\
			&= \frac{A^n}{|I|}\sum_{q=1}^{2^n}\int_{I(q)\cap E_s^I}\int_{F_s^I} |b(x)-\alpha_{\hat{I}}(b)||K(y_0,c(I))|d\hat{x}dx\\
			&\le \frac{2A^n}{|I|}\sum_{q=1}^{2^n}\int_{I(q)\cap E_s^I}\int_{F_s^I} |b(x)-b(\hat{x})||K(y_0,c(I))|d\hat{x}dx.
		\end{aligned}
	\end{equation*}
	For any $s\in \{1,2,3,4\}$, $x\in I(q)\cap E_s^I$ and $\hat{x}\in F_s^I$, from Lemma \ref{ball} we have
	\begin{equation*}
		\begin{aligned}
			|K(\hat{x},x)-K(y_0,c(I))|\lesssim_{n,T}  \frac{1}{A^{n+\alpha}|I|}.
		\end{aligned}
	\end{equation*}
	Hence $|K(\hat{x},x)|\ge \frac{1}{2}|K(y_0,c(I))|$ thanks to \eqref{minus1}. Thus
	\begin{equation*}
		\begin{aligned}
			M_s^I{}&\lesssim_{n,T} \frac{4A^n}{|I|}\sum_{q=1}^{2^n}\int_{I(q)\cap E_s^I}\int_{F_s^I} |b(x)-b(\hat{x})||K(\hat{x},x)|d\hat{x}dx.
		\end{aligned}
	\end{equation*}
	
	We first estimate $M_1^I$.
	Let $\theta_1\in [0,2\pi)$ such that $e^{i\theta_1}K(y_0,c(I))$ is  positive. Thus from Lemma \ref{ball} we obtain
	\begin{equation}\label{theta1K}
		|K(\hat{x},x)| \le 2\text{Re}\big(e^{i\theta_1}K(\hat{x},x)\big).
	\end{equation}
	Moreover, note that for any $x\in I(q)\cap E_1^I$ and $\hat{x}\in F_1^I$, the arguments of $e^{\mathrm{i}{\theta}}\big(b({x})-\alpha_{\hat{I}}(b)\big)$ and $e^{\mathrm{i}{\theta}}\big(\alpha_{\hat{I}}(b)-b(\hat{x})\big)$ both belong to $[-\frac{\pi}{4},\frac{\pi}{4}]$, thus
	\begin{equation*}
		-\frac{\pi}{4}\le \mathrm{arg}\Big(e^{\mathrm{i}{\theta}}\big(b(x)-b(\hat{x})\big)\Big)\le \frac{\pi}{4}.
	\end{equation*}
	This implies that
	\begin{equation}\label{b-b1}
		\begin{aligned}
			\big|e^{\mathrm{i}{\theta}}\big(b(x)-b(\hat{x})\big)\big|{}&
			\le 2\mathrm{Re}\Big(e^{\mathrm{i}{\theta}}\big(b(x)-b(\hat{x})\big)\Big)
		\end{aligned}
	\end{equation} 
	and
	\begin{equation}\label{imreb1}
		\Big|\mathrm{Im}\Big(e^{\mathrm{i}{\theta}}\big(b(x)-b(\hat{x})\big)\Big)\Big|\le \mathrm{Re}\Big(e^{\mathrm{i}{\theta}}\big(b(x)-b(\hat{x})\big)\Big).
	\end{equation}
	Thus from \eqref{theta1K} and \eqref{b-b1} we deduce that
	\begin{equation*}
		\begin{aligned}
			M_1^I{}&
			\lesssim_{n,T}\frac{4A^n}{|I|}\sum_{q=1}^{2^n}\int_{I(q)\cap E_1^I}\int_{F_1^I} \big|e^{\mathrm{i}{\theta}}\big(b(x)-b(\hat{x})\big)\big||K(\hat{x},x)|d\hat{x}dx\\
			&\le \frac{8A^n}{|I|}\sum_{q=1}^{2^n}\int_{I(q)\cap E_1^I}\int_{F_1^I} \mathrm{Re}\Big(e^{\mathrm{i}{\theta}}\big(b(x)-b(\hat{x})\big)\Big)|K(\hat{x},x)|d\hat{x}dx\\
			&\le  \frac{16 A^n}{|I|}\sum_{q=1}^{2^n}\int_{I(q)\cap E_1^I}\int_{F_1^I} \mathrm{Re}\Big(e^{\mathrm{i}{\theta}}\big(b(x)-b(\hat{x})\big)\Big)\cdot\text{Re}\big(e^{i\theta_1}K(\hat{x},x)\big)d\hat{x}dx.
		\end{aligned}
	\end{equation*}
	Notice that from Lemma \ref{ball} one has 
	\begin{equation}\label{eitheta1}
		\big|\mathrm{Im}\big(e^{i\theta_1}K(\hat{x},x)\big)\big|\le \varrho\mathrm{Re}\big(e^{i\theta_1}K(\hat{x},x)\big),
	\end{equation}
	where $\varrho$ is a positive number which depends on $A$, $\alpha$ and $n$ and is much less than $1$.
	Then from \eqref{imreb1} and \eqref{eitheta1} we derive that
	\begin{equation*}
		\begin{aligned}
			\mathrm{Re}\Big(e^{\mathrm{i}{\theta}}\big(b(x)-b(\hat{x})\big)\Big)\cdot\text{Re}\big(e^{i\theta_1}K(\hat{x},x)\big)
			&\le 2\mathrm{Re}\Big(e^{\mathrm{i}{\theta}}\big(b(x)-b(\hat{x})\big)\Big)\cdot\text{Re}\big(e^{i\theta_1}K(\hat{x},x)\big)\\
			& \quad \quad \quad \quad-2\mathrm{Im}\Big(e^{\mathrm{i}{\theta}}\big(b(x)-b(\hat{x})\big)\Big)\cdot\text{Im}\big(e^{i\theta_1}K(\hat{x},x)\big)\\
			&=2\mathrm{Re}\Big(e^{\mathrm{i}{\theta}}(b(x)-b(\hat{x}))\cdot e^{i\theta_1}K(\hat{x},x)\Big).
		\end{aligned}
	\end{equation*}
	Hence
	\begin{equation*}\label{M1I}
		\begin{aligned}
			M_1^I{}&
			\lesssim_{n,T} \frac{32A^n}{|I|}\sum_{q=1}^{2^n}\int_{I(q)\cap E_1^I}\int_{F_1^I} \mathrm{Re}\Big(e^{\mathrm{i}{\theta}}(b(x)-b(\hat{x}))\cdot e^{i\theta_1}K(\hat{x},x)\Big)d\hat{x}dx\\
			&\le \frac{32A^n}{|I|}\sum_{q=1}^{2^n}\bigg|\int_{I(q)\cap E_1^I}\int_{F_1^I} e^{\mathrm{i}{\theta}}(b(x)-b(\hat{x}))\cdot e^{i\theta_1}K(\hat{x},x)d\hat{x}dx\bigg|\\
			&=\frac{32A^n}{|I|}\sum_{q=1}^{2^n}\bigg|\int_{I(q)\cap E_1^I}\int_{F_1^I} (b(x)-b(\hat{x}))K(\hat{x},x)d\hat{x}dx\bigg|.
		\end{aligned}
	\end{equation*}
Similarly, the other three terms $M_2^I$, $M_3^I$ and $M_4^I$ can be dealt with in the same way as $M_1^I$ by rotation, and we obtain
	\begin{equation}\label{M2I}
		\begin{aligned}
			M_s^I{}&
			\lesssim_{n,T} \frac{A^n}{|I|}\sum_{q=1}^{2^n}\bigg|\int_{I(q)\cap E_s^I}\int_{F_s^I} (b(x)-b(\hat{x}))K(\hat{x},x)d\hat{x}dx\bigg|,\quad s\in\{1,2,3,4\}.
		\end{aligned}
	\end{equation}
	Hence from \eqref{M1M2} and \eqref{M2I} one has
	\begin{equation}\label{second2}
		\begin{aligned}
			\|b\|^p_{\pmb{B}_p^{0,2^n}(\mathbb{R}^n)}{}
			&= \sum_{k\in\mathbb{Z}}\sum_{I\in\mathcal{D}^0_k}\sum_{i=1}^{2^n-1} |I|^{-\frac{p}{2}} |\langle h_I^i,b\rangle|^p\\
			&\le (2^n-1) \cdot 4^{p-1} \sum_{I\in\mathcal{D}^0}\sum_{s=1}^4 (M_s^I)^p\\
			&\lesssim_{n,p,T} A^{np}\sum_{s=1}^4\sum_{q=1}^{2^n}\sum_{I\in\mathcal{D}^0}\bigg|\Big\langle \frac{|I(q)|^{\frac{1}{2}}\mathbbm{1}_{F_s^I}}{|I|}, [T,M_b]\frac{\mathbbm{1}_{I(q)\cap E_s^I}}{|I(q)|^{\frac{1}{2}}}\Big\rangle\bigg|^p.
		\end{aligned}
	\end{equation}
	For any $I\in\mathcal{D}^0_k$, note that $\mathrm{dist}(I,\hat{I})=2^{-k}C_nA$, where $C_n$ is a constant only depending on $n$. Let $c(\hat{I})$ be the center of $\hat{I}$. We consider the cube $Q(I)$, where the center of $Q(I)$ is $\frac{c(I)+c(\hat{I})}{2}$, and the length of $Q(I)$ is $2^{-k+1}C_nA$. This implies that $I,\hat{I}\subset Q(I)$. Besides, from Lemma \ref{Domegan1} we know that there exists some cube $J(I)\in\bigcup\limits_{i=1}^{n+1} \mathcal{D}^{\omega(i)}$ such that
	\begin{equation*}
		Q(I)\subseteq J(I)\subseteq c_n Q(I).
	\end{equation*}
	Notice that $I\subseteq J(I)$ and
	\begin{equation*}
		\ell(I)\leq \ell(J(I))\le c_n\ell(Q(I))\le 2c_nC_nA\ell(I).
	\end{equation*}
    Now for any $s\in\{1,2,3,4\}$ and $1\le q\le 2^n$, let
	\begin{equation*}
		e_{J(I),s,q}=\frac{|I(q)|^{\frac{1}{2}}\mathbbm{1}_{F_s^I}}{|I|} \quad \text{and} \quad f_{J(I),s,q}=\frac{\mathbbm{1}_{I(q)\cap E_s^I}}{|I(q)|^{\frac{1}{2}}}.
	\end{equation*}
	Then  $\mathrm{supp}e_{J(I),s,q},\,\mathrm{supp}f_{J(I),s,q}\subseteq J(I)$ and $\|e_{J(I),s,q}\|_{\infty},\,\|f_{J(I),s,q}\|_{\infty}\le C|J(I)|^{-\frac{1}{2}}$, where 
	the constant $C$ only depends on $n$ and $A$. Note that from \eqref{second2} one has
	\begin{align*}
		\|b\|^p_{\pmb{B}_p^{0,2^n}(\mathbb{R}^n)}&\lesssim_{n,p,T}A^{np}\sum_{s=1}^4\sum_{q=1}^{2^n}\sum_{I\in\mathcal{D}^0}\Big|\big\langle e_{J(I),s,q}, [T,M_b](f_{J(I),s,q})\big\rangle\Big|^p.
	\end{align*}
	From Lemma \ref{Domegan2}, each $J(I)$ contains only a finite number of dyadic cubes in $\mathcal{D}_k^0$, and this number only depends on $n$ and $A$. Therefore, from Lemma \ref{RSNWO} we get
	\begin{equation*}
		\begin{aligned}
			\|b\|_{\pmb{B}_p^{0,2^n}(\mathbb{R}^n)}
			\lesssim_{n,p,T} \|[T,M_b]\|_{S_p(L_2(\mathbb{R}^n))},
		\end{aligned}
	\end{equation*}
	as long as we fix $A$. (Here we apply Lemma \ref{RSNWO} for dyadic systems $ \mathcal{D}^{\omega(i)}$.)
\end{proof}

\subsection{Proof of the necessity of Theorem \ref{corollary1.8}}


As mentioned before, we will show that for $1<p<\infty$, the Besov space $\pmb{B}_p(\mathbb{R}^n)$ is the intersection of finite well-chosen martingale Besov spaces, where the number of chosen martingales only depends on $n$. We will use the dyadic covering result in Lemma \ref{Domegan1}.
\begin{prop}\label{pdayun}
	Let $1<p<\infty$. Let $\mathcal{D}^{\omega(1)}, \mathcal{D}^{\omega(2)}, \cdots, \mathcal{D}^{\omega(n+1)}$ be $n+1$ dyadic systems as in Lemma \ref{Domegan1}. Then
	$$  \pmb{B}_p(\mathbb{R}^n)=\bigcap_{i=1}^{n+1}  \pmb{B}_p^{\omega(i),2^n}(\mathbb{R}^n), $$
	where the norm for the intersection on the right hand side is the maximum of the involved norms.
\end{prop}

\begin{proof}
	By the standard limit argument, it suffices to show that if $b$ is a locally integrable complex-valued function, and $b\in \pmb{B}_p^{\omega(i),2^n}(\mathbb{R}^n)$ for any $1\le i\le n+1$, then
	\begin{equation*}
		\|b\|_{\pmb{B}_p(\mathbb{R}^n)}\approx_{n,p} \sum_{i=1}^{n+1} \|b\|_{\pmb{B}_p^{\omega(i),2^n}(\mathbb{R}^n)}.
	\end{equation*}
	Note that
	\begin{equation*}
		\begin{aligned}
			\|b\|^p_{\pmb{B}_p(\mathbb{R}^n)}{}&=\int_{\mathbb{R}^n\times\mathbb{R}^n} \frac{|b(x)-b(y)|^p}{|x-y|^{2n}}dxdy\\
			&= \sum_{k\in \mathbb{Z}}\int_{2^{-(k+2)}<|x-y|\le 2^{-(k+1)}} \frac{|b(x)-b(y)|^p}{|x-y|^{2n}}dxdy\\
			&\approx_{n} \sum_{k\in \mathbb{Z}}\int_{2^{-(k+2)}<|x-y|\le 2^{-(k+1)}} 2^{2nk}|b(x)-b(y)|^pdxdy\\
			&=\sum_{k\in \mathbb{Z}}\sum_{I\in \mathcal{D}^0_k}\frac{1}{|I|^2}\int_{I}\int_{2^{-(k+2)}<|x-y|\le 2^{-(k+1)}} |b(x)-b(y)|^pdxdy.
		\end{aligned}
	\end{equation*}
	For any given $k\in \mathbb{Z}$ and $I\in\mathcal{D}^0_k$, we consider the cube $2I$. From Lemma \ref{Domegan1}, we know that there exists $Q(I)\in\bigcup\limits_{i=1}^{n+1} \mathcal{D}^{\omega(i)}$ such that
	\begin{equation*}
		2I\subseteq Q(I)\subseteq c_n\cdot 2I.
	\end{equation*}
	Besides, note that
	\begin{equation*}
		\ell(Q(I))\le 2c_n\ell(I)=c_n2^{-k+1}.
	\end{equation*}
From Lemma \ref{Domegan2}, it implies that $Q(I)$ contains only a finite number of dyadic cubes in $\mathcal{D}_k^0$, and the number only depends on $n$.
	Thus
	\begin{equation*}
		\begin{aligned}
			{}&\frac{1}{|I|^2}\int_{I}\int_{2^{-(k+2)}<|x-y|\le 2^{-(k+1)}} |b(x)-b(y)|^pdxdy\\
			&\lesssim_n \frac{1}{|Q(I)|^2}\int_{Q(I)\times Q(I)} |b(x)-b(y)|^pdxdy.
		\end{aligned}
	\end{equation*}
	Then
	\begin{equation*}
		\begin{aligned}
			\|b\|^p_{\pmb{B}_p(\mathbb{R}^n)}{}&
			\lesssim_n \sum_{k\in \mathbb{Z}}\sum_{I\in \mathcal{D}^0_k}\frac{1}{|Q(I)|^2}\int_{Q(I)\times Q(I)} |b(x)-b(y)|^pdxdy\\
			&\lesssim_n\sum_{i=1}^{n+1}\sum_{Q\in \mathcal{D}^{\omega(i)}}\frac{1}{|Q|^2}\int_{Q\times Q} |b(x)-b(y)|^pdxdy.
		\end{aligned}
	\end{equation*}
	Note that for any given $1\le i\le n+1$ and $Q\in \mathcal{D}^{\omega(i)}$,
	\begin{equation*}
		\begin{aligned}
			  {}&\int_{Q\times Q} |b(x)-b(y)|^pdxdy\\
			  &=\int_{Q\times Q} \bigg|b(x)-\bigg\langle \frac{\mathbbm{1}_Q}{|Q|},b\bigg\rangle+\bigg\langle \frac{\mathbbm{1}_Q}{|Q|},b\bigg\rangle-b(y)\bigg|^pdxdy\\
			  &\lesssim_p |Q|\int_{Q} \bigg|b(x)-\bigg\langle \frac{\mathbbm{1}_Q}{|Q|},b\bigg\rangle\bigg|^pdx.
		\end{aligned}
	\end{equation*}
	Then from Proposition \ref{equbbk}, we obtain
	\begin{equation*}
		\begin{aligned}
			\|b\|^p_{\pmb{B}_p(\mathbb{R}^n)}
			&\lesssim_{n,p} \sum_{i=1}^{n+1}\sum_{Q\in \mathcal{D}^{\omega(i)}}\frac{1}{|Q|}\int_{Q} \bigg|b(x)-\bigg\langle \frac{\mathbbm{1}_Q}{|Q|},b\bigg\rangle\bigg|^pdx\\
			&=\sum_{i=1}^{n+1}\sum_{k\in\mathbb{Z}}\sum_{Q\in \mathcal{D}_k^{\omega(i)}}\frac{1}{|Q|}\int_{Q} |b(x)-b_k^{\omega(i)}(x)|^pdx\\
			&=\sum_{i=1}^{n+1}\sum_{k\in\mathbb{Z}} 2^{nk}\|b-b_k^{\omega(i)}\|^p_{L_p(\mathbb{R}^n)}\approx_{n,p} \sum_{i=1}^{n+1}\|b\|^p_{\pmb{B}_p^{\omega(i),2^n}(\mathbb{R}^n)},
		\end{aligned}
	\end{equation*}
where $$b_k^{\omega(i)}=\sum_{Q\in\mathcal{D}_k^{\omega(i)}}\bigg\langle \frac{\mathbbm{1}_Q}{|Q|},b\bigg\rangle \mathbbm{1}_Q.$$
Therefore, from Lemma \ref{Comparison} and the above inequality, we derive
    \begin{equation*}
	    \|b\|_{\pmb{B}_p(\mathbb{R}^n)}\approx_{n,p} \sum_{i=1}^{n+1} \|b\|_{\pmb{B}_p^{\omega(i),2^n}(\mathbb{R}^n)},
    \end{equation*}
    as desired.
\end{proof}

	 
Now we give the proof of the necessity of Theorem \ref{corollary1.8}.
\begin{proof}[Proof of the necessity of Theorem \ref{corollary1.8}]
 For any $\omega\in(\{0,1\}^n)^\mathbb{Z}$, from Lemma \ref{conv1} one has
	\begin{equation*}
		\|b\|_{\pmb{B}_p^{\omega,2^n}(\mathbb{R}^n)}\lesssim_{n,p,T} \|[T,M_{b}]\|_{S_p(L_2(\mathbb{R}^n))}<\infty.
	\end{equation*}
	Hence from Proposition \ref{pdayun} we obtain that $b\in \pmb{B}_p(\mathbb{R}^n)$ and
	\begin{equation*}
		\|b\|_{\pmb{B}_p(\mathbb{R}^n)}\lesssim_{n,p,T} \|[T,M_{b}]\|_{S_p(L_2(\mathbb{R}^n))}<\infty.
	\end{equation*}
	In particular, when $n\geq 2$ and $0<p\le n$, if $[T,M_{b}]\in S_p(L_2(\mathbb{R}^n))$, then $[T,M_{b}]\in S_n(L_2(\mathbb{R}^n))$ as $S_p(L_2(\mathbb{R}^n))\subset S_q(L_2(\mathbb{R}^n))$ for $0<p\leq q\leq \8$. This implies that $b\in \pmb{B}_n(\mathbb{R}^n)$. Hence by Proposition \ref{0pn}, we see that $b$ is constant.	
\end{proof}

\bigskip

\section{Weak-type Schatten class of commutators}\label{WEAK}

This section is devoted to the proof of Theorem \ref{corollary1.10}. We divide this section into two parts, which concern the sufficiency and the necessity respectively. By Lemma \ref{WBNW1N}, it suffices to estimate the $S_{n, \infty}$-norm of commutators in terms of the $\pmb{WB}_{n,\infty}$-norm for $n\geq 2$. In fact, we will prove similar results for all $n\leq p<\infty$, instead of just $p=n$. Finally, we will present the Schatten-Lorentz class of commutators from real interpolation.

\subsection{Proof of the sufficiency of Theorem \ref{corollary1.10}} We will show the sufficiency part of Theorem \ref{corollary1.10} for all $2\leq n\leq p<\infty$. We also consider the case $n=1$. We divide the proof into three cases: 
\begin{enumerate}
	\item $p> 2$ for $n\geq 1$,
	\item $p= 2$ for $n=1$ or $n=2$,
	\item $1< p< 2$ and $1/2\le \alpha\le 1$ for $n=1$.
\end{enumerate}
The last two ones are more difficult. In addition, for the third case, we will prove a stronger result, that is, the sufficiency part of Theorem \ref{corollary1.10} holds for $1< p< 2$ and $1/p-1/2< \alpha\le 1$ when $n=1$.
\begin{proof}[Proof of the sufficiency of Theorem \ref{corollary1.10}]

 Similarly to the proof of the sufficiency of Theorem \ref{corollary1.8}, in the following we will mainly show that $\|[S^{ij}_\omega,R_b]\|_{S_{p,\infty}(L_2(\mathbb{R}^n))}$ increases with polynomial growth with respect to $i$ and $j$ uniformly on $\omega\in(\{0,1\}^n)^\mathbb{Z}$.
 
  Without loss of generality, we assume $\omega=0$. For any given $i, j\in \mathbb{N}\cup\{0\}$, let $\varPhi=[S^{ij}_0,R_b]$. From the proof of the sufficiency of Theorem \ref{corollary1.8}, we know that $\varPhi^*\varPhi$ is a block diagonal matrix with blocks $B_K^*B_K$ for all $K\in \mathcal{ D}^0$.
Now fix $K\in\mathcal{D}^0$. Let the eigenvalues of $B_K^*B_K$ be
\begin{equation*}
	s_1(K)\ge s_2(K)\ge \cdots\ge s_{2^{in}(2^n-1)}(K)\ge 0.
\end{equation*}
Then by \eqref{Trbkbk} for $n\ge 1$, 
\begin{equation*}
	\begin{aligned}
		\mathrm{Tr}(B_K^*B_K){}
		&\le \frac{4(2^n-1)^2}{|K|}\|(b-b_{k})\mathbbm{1}_K\|_{L_2(\mathbb{R}^n)}^2:=M_K.
	\end{aligned}
\end{equation*}
From \eqref{smTrb}, it implies that 
\begin{equation}\label{smK}
	0\le s_m(K)\le \frac{M_K}{m},\quad \forall 1\le m\le 2^{in}(2^n-1).
\end{equation}
	Now we divide the proof into three cases:\\
	$\bullet$ $p> 2$ for $n\geq 1$;\\
	$\bullet$ $p= 2$ for $n=1$ or $n=2$;\\
	$\bullet$ $1< p< 2$ and $1/p-1/2< \alpha\le 1$ for $n=1$.\\
	
\noindent	(1) When $p> 2$ and $n\ge 1$, by \eqref{smK}, we obtain
\begin{equation}\label{phiphi}
	\begin{aligned}
		\|\varPhi^*\varPhi\|_{S_{p/2,\infty}(L_2(\mathbb{R}^n))}{}&=\bigg\|\Big\{s_1(K),s_2(K),\cdots,s_{2^{in}(2^n-1)}(K)\Big\}_{K\in\mathcal{D}^0}\bigg\|_{\ell_{p/2,\infty}}\\
		&\le \bigg\|\Big\{M_K,\frac{M_K}{2},\cdots,\frac{M_K}{2^{in}(2^n-1)}\Big\}_{K\in\mathcal{D}^0}\bigg\|_{\ell_{p/2,\infty}}\\
		&\lesssim_p\sum_{m=1}^{2^{in}(2^n-1)}\frac{1}{m}\big\|\{M_K\}_{K\in\mathcal{D}^0}\big\|_{\ell_{p/2,\infty}}.
	\end{aligned}
\end{equation}
From Lemma \ref{TECHRS}, one has
\begin{equation}\label{phiphi222}
	\begin{aligned}
		\big\|\{M_K\}_{K\in\mathcal{D}^0}\big\|_{\ell_{p/2,\infty}}{}
		&=4(2^n-1)^2\Big\|\big\{(MO_2(b;K))^2\big\}_{K\in\mathcal{D}^0}\Big\|_{\ell_{p/2,\infty}}\\
		&= 4(2^n-1)^2\Big\|\big\{MO_2(b;K)\big\}_{K\in\mathcal{D}^0}\Big\|_{\ell_{p,\infty}}^2
		\lesssim_{n,p}\|b\|_{\pmb{WB}_{p,\infty}(\mathbb{R}^n)}^{2}.
	\end{aligned}
\end{equation}
Thus
\begin{equation*}
	\begin{aligned}
		\|\varPhi^*\varPhi\|_{S_{p/2,\infty}(L_2(\mathbb{R}^n))}\lesssim_{n,p}\sum_{m=1}^{2^{in}(2^n-1)}\frac{1}{m}\|b\|_{\pmb{WB}_{p,\infty}(\mathbb{R}^n)}^{2}{}
		&\lesssim_n i\|b\|_{\pmb{WB}_{p,\infty}(\mathbb{R}^n)}^{2}.
	\end{aligned}
\end{equation*}
It implies that 
\begin{equation*}
	\|\varPhi\|_{S_{p,\infty}(L_2(\mathbb{R}^n))}=	\|\varPhi^*\varPhi\|_{S_{p/2,\infty}(L_2(\mathbb{R}^n))}^{1/2}\lesssim_{n,p}i^{1/2}\|b\|_{\pmb{WB}_{p,\infty}(\mathbb{R}^n)}. 
\end{equation*}
Since the above estimation is independent of the choice of $\omega$, one has
$$ \|[S_\omega^{ij}, R_b]\|_{S_{p,\infty}(L_2(\mathbb{R}^n))}\lesssim_{n, p} i^{1/2}  \|b\|_{\pmb{WB}_{p,\infty}(\mathbb{R}^n)}. $$
Thus from Theorem \ref{lem2.1}, Lemma \ref{TLambdab} and Lemma \ref{IHbxiaolemma},
\begin{equation*}
	\begin{aligned}
		\|[S_\omega^{ij},M_b]\|_{S_{p,\infty}(L_2(\mathbb{R}^n))}{}&\lesssim \|[S_\omega^{ij},\pi_b]\|_{S_{p,\infty}(L_2(\mathbb{R}^n))}+\|[S_\omega^{ij},\varLambda_b]\|_{S_{p,\infty}(L_2(\mathbb{R}^n))}
		+\|[S_\omega^{ij},R_b]\|_{S_{p,\infty}(L_2(\mathbb{R}^n))}\\
		&\lesssim_{n, p} (i^{1/2}+1)  \|b\|_{\pmb{WB}_{p,\infty}(\mathbb{R}^n)}.
	\end{aligned}
\end{equation*} 
Therefore by Lemma \ref{weakconver}, Proposition \ref{T0est} and Lemma \ref{IHbxiaolemma},
\begin{equation*}
	\begin{aligned}
		{}&\|[T,M_b]\|_{S_{p,\infty}(L_2(\mathbb{R}^n))}
		\lesssim_{n,p, T} \big(1+\|T(1)\|_{BMO(\mathbb{R}^n)}+\|T^*(1)\|_{BMO(\mathbb{R}^n)}\big) \|b\|_{\pmb{WB}_{p,\infty}(\mathbb{R}^n)}.
	\end{aligned}
\end{equation*}

\noindent (2) When $p= 2$ and $n=1$ or $n=2$, note that from \eqref{phiphi222},
\begin{equation*}
	\begin{aligned}
		\|\{M_K\}_{K\in\mathcal{D}^0}\|_{\ell_{1,\infty}}<\infty.
	\end{aligned}
\end{equation*}
We rearrange $\{M_K\}_{K\in\mathcal{D}^0}$ in an non-increasing order:
\begin{equation*}
	M_{K_1}\ge M_{K_2}\ge M_{K_3}\ge \cdots,
\end{equation*}
and define
\begin{equation}\label{matrixAB}
	\begin{aligned}
		A=\begin{pmatrix}
			M_{K_1} & \,  &   0 \\
			\,     &  M_{K_2} &   \,\\
			0      & \,  &   \ddots
		\end{pmatrix}
		\quad \text{and}\quad
		B=\begin{pmatrix}
			\ddots & \,  &   0 \\
			\,     &  t^{-1} &   \,\\
			0      & \,  &   \ddots
		\end{pmatrix}_{1\le t\le 2^{in}(2^n-1)}.
	\end{aligned}
\end{equation}
Then from \eqref{phiphi} one has
\begin{equation*}
	\begin{aligned}
		\|\varPhi^*\varPhi\|_{S_{1,\infty}(L_2(\mathbb{R}^n))}{}&
		\le \bigg\|\Big\{M_{K_s},\frac{M_{K_s}}{2},\cdots,\frac{M_{K_s}}{2^{in}(2^n-1)}\Big\}_{s\in\mathbb{N}}\bigg\|_{\ell_{1,\infty}}
		=\|A\otimes B\|_{S_{1,\infty}},
	\end{aligned}
\end{equation*}
where $\|A\otimes B\|_{S_{1,\infty}}$ is the $S_{1,\infty}$ quasi-norm of the matrix $A\otimes B$.
Since
\begin{equation*}
	\begin{aligned}
		\|A\|_{S_{1,\infty}}=\sup_{s\in\mathbb{N}}s\cdot M_{K_s},
	\end{aligned}
\end{equation*}
we have $0\le M_{K_s}\le s^{-1}\|A\|_{S_{1,\infty}}$.
Let
\begin{equation*}
	\begin{aligned}
		C=\|A\|_{S_{1,\infty}}\begin{pmatrix}
			\ddots & \,  &   0 \\
			\,     &  s^{-1} &   \,\\
			0      & \,  &   \ddots
		\end{pmatrix}_{s\in\mathbb{N}}.
	\end{aligned}
\end{equation*}
Then $0\le A\le C$ and
\begin{equation}\label{2step1}
	\begin{aligned}
		 \|\varPhi^*\varPhi\|_{S_{1,\infty}(L_2(\mathbb{R}^n))}\le\|A\otimes B\|_{S_{1,\infty}}\le \|C\otimes B\|_{S_{1,\infty}}.
	\end{aligned}
\end{equation}
Note that
\begin{equation*}
	\begin{aligned}
		C\otimes B{}&=\|A\|_{S_{1,\infty}}\begin{pmatrix}
			\ddots & \,  &   0 \\
			\,     &  (st)^{-1} &   \,\\
			0      & \,  &   \ddots
		\end{pmatrix}_{s\in\mathbb{N},1\le t\le 2^{in}(2^n-1)}.
	\end{aligned}
\end{equation*}
We rearrange $\{(st)^{-1}\}_{s\in\mathbb{N},1\le t\le 2^{in}(2^n-1)}$ in a non-increasing order, and let $a_{st}=(st)^{-1}$. 
Then
\begin{equation*}\label{2step2}
	\begin{aligned}
		\|C\otimes B\|_{S_{1,\infty}}=\|A\|_{S_{1,\infty}}\sup_{s_0\in\mathbb{N},1\le t_0\le 2^{in}(2^n-1)}a_{s_0t_0}\cdot\mathscr{P}\big(a_{s_0t_0}\big),
	\end{aligned}
\end{equation*}
where $\mathscr{P}\big(a_{s_0t_0}\big)$ is the position of $a_{s_0t_0}$ in the non-increasing sequence $\{a_{st}\}_{s\in\mathbb{N},1\le t\le 2^{in}(2^n-1)}$.

Our aim is to estimate $\mathscr{P}\big(a_{s_0t_0}\big)$. 
For any $s_0\in\mathbb{N}$ and $1\le t_0\le 2^{in}(2^n-1)$, Note that if there exists $a_{st}\ge a_{s_0t_0}$, then $t\le s_0t_0$. Then we define the following set: 
$$L_{s_0,t_0,t}=\{s\in\mathbb{N}:a_{st}\ge a_{s_0t_0}\},\quad \forall 1\le t\le \min\big\{2^{in}(2^n-1),s_0t_0\big\}. $$
Then $\#L_{s_0,t_0,t}\le \frac{s_0t_0}{t}$, and  
\begin{equation*}
	\begin{aligned}
		\mathscr{P}\big(a_{s_0t_0}\big)\le\sum_{t=1}^{\min\big\{2^{in}(2^n-1),s_0t_0\big\}} \#L_{s_0,t_0,t}\le\sum_{t=1}^{\min\big\{2^{in}(2^n-1),s_0t_0\big\}}\frac{s_0t_0}{t}
	    {}
	    &\lesssim i\cdot s_0t_0.
	\end{aligned}
\end{equation*}
It implies that
\begin{equation}\label{2step3}
	\begin{aligned}
		\|C\otimes B\|_{S_{1,\infty}}{}&=\|A\|_{S_{1,\infty}}\sup_{s_0\in\mathbb{N},1\le t_0\le 2^{in}(2^n-1)}a_{s_0t_0}\cdot\mathscr{P}\big(a_{s_0t_0}\big)\\
		&\lesssim \|A\|_{S_{1,\infty}}\sup_{s_0\in\mathbb{N},1\le t_0\le 2^{in}(2^n-1)}(s_0t_0)^{-1}\cdot i\cdot s_0t_0\\
		&=i\|A\|_{S_{1,\infty}}=i\|\{M_K\}_{K\in\mathcal{D}^0}\|_{\ell_{1,\infty}}.
	\end{aligned}
\end{equation}
Thus by \eqref{2step1}, \eqref{2step3} and \eqref{phiphi222} one gets
\begin{equation*}
	\|\varPhi\|_{S_{2,\infty}(L_2(\mathbb{R}^n))}=	\|\varPhi^*\varPhi\|_{S_{1,\infty}(L_2(\mathbb{R}^n))}^{1/2}\lesssim i^{1/2}\|\{M_K\}_{K\in\mathcal{D}^0}\|_{\ell_{1,\infty}}^{1/2} \lesssim i^{1/2}\|b\|_{\pmb{WB}_{2,\infty}(\mathbb{R}^n)}. 
\end{equation*}
Since the above estimation is independent of the choice of $\omega$, one has
$$ \|[S_\omega^{ij}, R_b]\|_{S_{2,\infty}(L_2(\mathbb{R}^n))}\lesssim   \|b\|_{\pmb{WB}_{2,\infty}(\mathbb{R}^n)}, $$
which yields
$$ \|[S_\omega^{ij}, M_b]\|_{S_{2,\infty}(L_2(\mathbb{R}^n))}\lesssim  \|b\|_{\pmb{WB}_{2,\infty}(\mathbb{R}^n)}. $$
Therefore by Lemma \ref{weakconver} and Remark \ref{Spqquasi},
\begin{equation*}
	\begin{aligned}
		{}&\|[T,M_b]\|_{S_{2,\infty}(L_2(\mathbb{R}^n))}
	     \lesssim_{T} \big(1+\|T(1)\|_{BMO(\mathbb{R}^n)}+\|T^*(1)\|_{BMO(\mathbb{R}^n)}\big) \|b\|_{\pmb{WB}_{2,\infty}(\mathbb{R}^n)}.
	\end{aligned}
\end{equation*}

\noindent (3) When $1<p< 2$, $1/p-1/2<\alpha\le 1$ and $n=1$, we use the same method as in the second case. From \eqref{phiphi} we know
\begin{equation*}
	\begin{aligned}
		\|\varPhi\|_{S_{p,\infty}(L_2(\mathbb{R}))}=\|\varPhi^*\varPhi\|_{S_{p/2,\infty}(L_2(\mathbb{R}))}^{1/2}{}
		&\le \bigg\|\Big\{M_{K_s},\frac{M_{K_s}}{2},\cdots,\frac{M_{K_s}}{2^{i}}\Big\}_{s\in\mathbb{N}}\bigg\|_{\ell_{p/2,\infty}}^{1/2}=\|A\otimes B\|_{S_{p/2,\infty}}^{1/2},
	\end{aligned}
\end{equation*}
where $A$ and $B$ are two matrices defined in \eqref{matrixAB} for $n=1$, and $\|A\otimes B\|_{S_{p/2,\infty}}$ is the $S_{p/2,\infty}$ quasi-norm of the matrix $A\otimes B$. Since
\begin{equation*}
	\begin{aligned}
		\|A\|_{S_{p/2,\infty}}=\sup_{s\in\mathbb{N}}s^{2/p}\cdot M_{K_s},
	\end{aligned}
\end{equation*}
we have $0\le M_{K_s}\le s^{-2/p}\|A\|_{S_{p/2,\infty}}$.
Let
\begin{equation*}
	\begin{aligned}
		C=\|A\|_{S_{p/2,\infty}}\begin{pmatrix}
			\ddots & \,  &   0 \\
			\,     &  s^{-2/p} &   \,\\
			0      & \,  &   \ddots
		\end{pmatrix}_{s\in\mathbb{N}}.
	\end{aligned}
\end{equation*}
Then $0\le A\le C$ and
\begin{equation}\label{2step1n=1}
	\begin{aligned}
		\|\varPhi\|_{S_{p,\infty}(L_2(\mathbb{R}))}\le\|A\otimes B\|_{S_{p/2,\infty}}^{1/2}\le \|C\otimes B\|_{S_{p/2,\infty}}^{1/2}.
	\end{aligned}
\end{equation}
Note that
\begin{equation*}
	\begin{aligned}
		C\otimes B{}&=\|A\|_{S_{p/2,\infty}}\begin{pmatrix}
			\ddots & \,  &   0 \\
			\,     &  s^{-2/p}t^{-1} &   \,\\
			0      & \,  &   \ddots
		\end{pmatrix}_{s\in\mathbb{N},1\le t\le 2^i}.
	\end{aligned}
\end{equation*}
We rearrange $\{s^{-2/p}t^{-1}\}_{s\in\mathbb{N},1\le t\le 2^i}$ in a non-increasing order, and let $a_{st}=s^{-2/p}t^{-1}$. 
Then
\begin{equation*}\label{2step2}
	\begin{aligned}
		\|C\otimes B\|_{S_{p/2,\infty}}=\|A\|_{S_{p/2,\infty}}\sup_{s_0\in\mathbb{N},1\le t_0\le 2^i}a_{s_0t_0}\cdot\mathscr{P}\big(a_{s_0t_0}\big)^{2/p},
	\end{aligned}
\end{equation*}
where $\mathscr{P}\big(a_{s_0t_0}\big)$ is the position of $a_{s_0t_0}$ in the non-increasing sequence $\{a_{st}\}_{s\in\mathbb{N},1\le t\le 2^i}$.

Our aim is to estimate $\mathscr{P}\big(a_{s_0t_0}\big)$. 
For any $s_0\in\mathbb{N}$ and $1\le t_0\le 2^i$, we define the following set: 
$$L_{s_0,t_0,t}=\{s\in\mathbb{N}:a_{st}\ge a_{s_0t_0}\},\quad \forall 1\le t\le 2^i. $$
Then $\#L_{s_0,t_0,t}\le s_0t_0^{p/2}t^{-p/2}$, and  
\begin{equation*}
	\begin{aligned}
		\mathscr{P}\big(a_{s_0t_0}\big)\le\sum_{t=1}^{2^i} \#L_{s_0,t_0,t}\le\sum_{t=1}^{2^i}s_0t_0^{p/2}t^{-p/2}
		{}&	\lesssim s_0t_0^{p/2}\cdot 2^{i(1-p/2)}.
	\end{aligned}
\end{equation*}
It implies that
\begin{equation}\label{2step3n=1}
	\begin{aligned}
		\|C\otimes B\|_{S_{p/2,\infty}}{}&=\|A\|_{S_{p/2,\infty}}\sup_{s_0\in\mathbb{N},1\le t_0\le 2^i}a_{s_0t_0}\cdot\mathscr{P}\big(a_{s_0t_0}\big)^{2/p}\\
		&\lesssim \|A\|_{S_{p/2,\infty}}\sup_{s_0\in\mathbb{N},1\le t_0\le 2^i}s_0^{-2/p}t_0^{-1}\cdot (s_0t_0^{p/2})^{2/p}\cdot 2^{i(2/p-1)}\\
		&=2^{i(2/p-1)}\|A\|_{S_{p/2,\infty}}=2^{i(2/p-1)}\|\{M_K\}_{K\in\mathcal{D}^0}\|_{\ell_{p/2,\infty}}.
	\end{aligned}
\end{equation}
Thus by \eqref{2step1n=1}, \eqref{2step3n=1} and \eqref{phiphi222} one gets
\begin{equation*}
	\|\varPhi\|_{S_{p,\infty}(L_2(\mathbb{R}))}\lesssim 2^{i(1/p-1/2)}\|\{M_K\}_{K\in\mathcal{D}^0}\|_{\ell_{p/2,\infty}}^{1/2} \lesssim_p 2^{i(1/p-1/2)}\|b\|_{\pmb{WB}_{p,\infty}(\mathbb{R})}. 
\end{equation*}
Since the above estimation is independent of the choice of $\omega$, one has
$$ \|[S_\omega^{ij}, R_b]\|_{S_{p,\infty}(L_2(\mathbb{R}))}\lesssim_p 2^{i(1/p-1/2)} \|b\|_{\pmb{WB}_{p,\infty}(\mathbb{R})}, $$
which yields
$$ \|[S_\omega^{ij}, M_b]\|_{S_{p,\infty}(L_2(\mathbb{R}))}\lesssim_{p} 2^{i(1/p-1/2)}  \|b\|_{\pmb{WB}_{p,\infty}(\mathbb{R})}. $$
Since $1/p-1/2<\alpha\le 1$, we get
\begin{equation*}
	\begin{aligned}
		\sum_{i,j=0}^\infty &\tau(i,j)\|[S_\omega^{ij}, M_b]\|_{S_{p,\infty}(L_2(\mathbb{R}))}\\
		&\lesssim_{p} \sum_{i,j=0}^\infty (1+\max\{i,j\})^{2(1+\alpha)}2^{i(1/p-1/2)-\alpha\max\{i,j\}}\|b\|_{\pmb{WB}_{p,\infty}(\mathbb{R})}<\infty.
	\end{aligned}
\end{equation*}
Therefore by Lemma \ref{weakconver} and Remark \ref{Spqquasi},
\begin{equation*}
	\begin{aligned}
		{}&\|[T,M_b]\|_{S_{p,\infty}(L_2(\mathbb{R}))}\lesssim_{p, T} \big(1+\|T(1)\|_{BMO(\mathbb{R})}+\|T^*(1)\|_{BMO(\mathbb{R})}\big) \|b\|_{\pmb{WB}_{p,\infty}(\mathbb{R})}.
	\end{aligned}
\end{equation*}
\end{proof}

\subsection{Proof of the necessity of Theorem \ref{corollary1.10}}
We will show that the necessity part of Theorem \ref{corollary1.10} indeed holds for  $1< p<\infty$ and $0<\alpha\le 1$. Our main ingredient is also the complex method.

	
\begin{proof}[Proof of the necessity of Theorem \ref{corollary1.10}]
Suppose that $1<p<\infty$. Fix $\omega\in(\{0,1\}^n)^\mathbb{Z}$. For any given $k\in \mathbb{Z}$ and $I\in \mathcal{D}_k^\omega$, 
on the one hand, from \eqref{M1M2} and \eqref{M2I} one has	
\begin{equation}\label{comp1}
	\begin{aligned}
		MO_1(b;I){}&\le \frac{2}{|I|}\sum_{q=1}^{2^n}\int_{I(q)} |b(x)-\alpha_{\hat{I}}(b)|dx\\
		&\lesssim_{n,T} \sum_{s=1}^4\frac{A^n}{|I|}\bigg|\int_{ E_s^I}\int_{F_s^I} (b(x)-b(\hat{x}))K(\hat{x},x)d\hat{x}dx\bigg|,
	\end{aligned}
\end{equation}
where $A$ is a sufficiently large number, and $E_s^I$ and $F_s^I$ are given in \eqref{ESI} and \eqref{FSI} respectively.
On the other hand, for any $1\le s\le 4$, let
\begin{equation*}
	\phi_{I,s}(f)
	=\bigg\langle\frac{\mathbbm{1}_{F_s^I}}{|I|},f\bigg\rangle \mathbbm{1}_{E_s^I}.
\end{equation*}
Then for any $\hat{x}\in\mathbb{R}^n$,
\begin{equation*}
	\begin{aligned}
		[T,M_b](\phi_{I,s})(f)(\hat{x}){}&=\int_{\mathbb{R}^n}(b(x)-b(\hat{x}))K(\hat{x},x)\phi_{I,s}(f)(x)dx\\
		&=\frac{1}{|I|}\int_{E_s^I}\int_{F_s^I}(b(x)-b(\hat{x}))K(\hat{x},x)f(y)dydx.
	\end{aligned}
\end{equation*}
It implies that 
\begin{equation}\label{comp2}
	\begin{aligned}
		\mathrm{Tr}([T,M_b]\phi_{I,s}){}&=\frac{1}{|I|}\int_{E_s^I}\int_{F_s^I}(b(x)-b(\hat{x}))K(\hat{x},x)d\hat{x}dx.
	\end{aligned}
\end{equation}
By \eqref{comp1} and \eqref{comp2} one has
\begin{equation*}
	MO_1(b;I)\lesssim_{n,T} \sum_{s=1}^4\big|\mathrm{Tr}([T,M_b]\phi_{I,s})\big|.
\end{equation*}
It implies that
\begin{equation*}
	\Big\|\big\{MO_1(b;I)\big\}_{I\in\mathcal{D}^\omega}\Big\|_{\ell_{p,\infty}}
	\lesssim_{n,p,T} \sum_{s=1}^4\bigg\|\Big\{\mathrm{Tr}([T,M_b]\phi_{I,s})\Big\}_{I\in\mathcal{D}^\omega}\bigg\|_{\ell_{p,\infty}}.
\end{equation*}
Now for any given $1\le s\le 4$ and any sequence $\{a_{I,s}\}_{I\in\mathcal{D}^\omega}$ satisfying $\|\{a_{I,s}\}_{I\in\mathcal{D}^\omega}\|_{\ell_{p',1}}=1$, one has
\begin{equation*}
	\begin{aligned}
		\bigg|\sum_{I\in\mathcal{D}^\omega}\mathrm{Tr}([T,M_b]\phi_{I,s})\cdot a_{I,s}\bigg|{}&
		=\bigg|\mathrm{Tr}\Big([T,M_b]\cdot\sum_{I\in\mathcal{D}^\omega}a_{I,s}\phi_{I,s}\Big)\bigg|\\
		&\lesssim_p \big\|[T,M_b]\big\|_{S_{p,\infty}(L_2(\mathbb{R}^n))}\bigg\|\sum_{I\in\mathcal{D}^\omega}a_{I,s}\phi_{I,s}\bigg\|_{S_{p',1}(L_2(\mathbb{R}^n))}.
	\end{aligned}
\end{equation*}
Note that
\begin{equation*}
	\begin{aligned}
		\sum_{I\in\mathcal{D}^\omega}a_{I,s}\phi_{I,s}(f){}
		&=\sum_{I\in\mathcal{D}^\omega}a_{I,s}\bigg\langle \frac{\mathbbm{1}_{F_s^I}}{|I|^{1/2}},f\bigg\rangle\frac{\mathbbm{1}_{E_s^I}}{|I|^{1/2}}.
	\end{aligned}
\end{equation*}
Using the same method as in the end of the proof of Lemma \ref{conv1},
 from Lemma \ref{WEAKnonNWOpre1} one has
\begin{equation*}
	\begin{aligned}
		\bigg\|\sum_{I\in\mathcal{D}^\omega}a_{I,s}\phi_{I,s}\bigg\|_{S_{p',1}(L_2(\mathbb{R}^n))}{}
		&\lesssim_{n,p} \sum_{i=1}^{n+1}\big\|\{a_{I,s}\}_{J(I)\in\mathcal{D}^{\omega(i)}}\big\|_{\ell_{p',1}}\\
		&\le \sum_{i=1}^{n+1}\big\|\{a_{I,s}\}_{I\in\mathcal{D}^{\omega}}\big\|_{\ell_{p',1}}=n+1.
	\end{aligned}
\end{equation*}
It implies that
\begin{equation*}
	\begin{aligned}
		\bigg|\sum_{I\in\mathcal{D}^\omega}\mathrm{Tr}([T,M_b]\phi_{I,s})\cdot a_{I,s}\bigg|{}&
	    \lesssim_{n,p}\big\|[T,M_b]\big\|_{S_{p,\infty}(L_2(\mathbb{R}^n))}.
	\end{aligned}
\end{equation*}
Thus by duality,
\begin{equation*}
	\bigg\|\Big\{\mathrm{Tr}([T,M_b]\phi_{I,s})\Big\}_{I\in\mathcal{D}^\omega}\bigg\|_{\ell_{p,\infty}}\lesssim_{n,p,T} \big\|[T,M_b]\big\|_{S_{p,\infty}(L_2(\mathbb{R}^n))},
\end{equation*}
and then
\begin{equation*}
	\Big\|\big\{MO_1(b;I)\big\}_{I\in\mathcal{D}^\omega}\Big\|_{\ell_{p,\infty}}\lesssim_{n,p,T} \big\|[T,M_b]\big\|_{S_{p,\infty}(L_2(\mathbb{R}^n))}.
\end{equation*}
Since the choice of $\omega$ is arbitrary, we conclude that
\begin{equation*}
	\|b\|_{\pmb{WB}_{p,\infty}(\mathbb{R}^n)}\lesssim_{n,p,T} \|[T,M_b]\|_{S_{p,\infty}(L_2(\mathbb{R}^n))},
\end{equation*}
as desired.
\end{proof}

We conclude this section by some remarks on the Schatten-Lorentz class of commutators. It is well-known that
if $0<p_0<p_1< \infty$, $0\le q_0,q_1,q\le\infty$ and $0<\theta<1$, then
	\begin{equation*}
		\big(S_{p_0,q_0}(L_2(\mathbb{R}^n)),S_{p_1,q_1}(L_2(\mathbb{R}^n))\big)_{\theta,q}=S_{p,q}(L_2(\mathbb{R}^n))
	\end{equation*} 
with equivalent quasi-norms, where $1/p=(1-\theta)/p_0+\theta/p_1$.

From real interpolation, Theorem \ref{corollary1.8}, Theorem \ref{corollary1.10} and Lemma \ref{RS4.10}, we obtain the following two remarks.
\begin{rem}\label{WEAKthm6.4pq}
	Suppose that $n\le p<\infty$, $0<\alpha\le 1$ when $n\ge 2$, or $2\le p<\infty$, $0<\alpha\le 1$ when $n= 1$, or $1< p<2$, $1/2\le \alpha\le 1$ when $n=1$.
	Suppose that $1\le q\le \infty$. Let $T\in B(L_2(\mathbb{R}^n))$ be a singular integral operator with kernel $K(x,y)$ satisfying \eqref{standard}. Let $b$ be a locally integrable complex-valued function. If $b\in \pmb{WB}_{p,q}(\mathbb{R}^n)$, then $ C_{T, b}\in S_{p,q}(L_2(\mathbb{R}^n))$ and 
	$$ \| C_{T, b}\|_{S_{p,q}(L_2(\mathbb{R}^n))}\lesssim_{n, p,q,T}\big(1+\|T(1)\|_{BMO(\mathbb{R}^n)}+\|T^*(1)\|_{BMO(\mathbb{R}^n)}\big)\|b\|_{\pmb{WB}_{p,q}(\mathbb{R}^n)}. $$
\end{rem}

\begin{rem}\label{WEAKthm6.4pq24}
	Suppose that $1< p<\infty$, $1\le q\le \infty$ and $0<\alpha\le 1$. Let $T\in B(L_2(\mathbb{R}^n))$ be a singular integral operator with a non-degenerate kernel $K(x,y)$ satisfying \eqref{standard}. Suppose that $b$ is a locally integrable complex-valued function. If $C_{T, b}\in S_{p,q}(L_2(\mathbb{R}^n))$, then $b\in \pmb{WB}_{p,q}(\mathbb{R}^n)$ and
	\begin{equation*}
		\|b\|_{\pmb{WB}_{p,q}(\mathbb{R}^n)}\lesssim_{n,p,q,T} \|C_{T, b}\|_{S_{p,q}(L_2(\mathbb{R}^n))}.
	\end{equation*}
\end{rem}

 Combine Remark \ref{WEAKthm6.4pq} and Remark \ref{WEAKthm6.4pq24}, and we get:
\begin{rem}
	Suppose that $n\le p<\infty$, $0<\alpha\le 1$ when $n\ge 2$, or $2\le p<\infty$, $0<\alpha\le 1$ when $n= 1$, or $1< p<2$, $1/2\le \alpha\le 1$ when $n=1$.
	Suppose that $1\le q\le \infty$. Let $T\in B(L_2(\mathbb{R}^n))$ be a singular integral operator with a nondegenerate kernel $K(x,y)$ satisfying \eqref{standard}. Let $b$ be a locally integrable complex-valued function. Then $ C_{T, b}\in S_{p,q}(L_2(\mathbb{R}^n))$ if and only if $b\in \pmb{WB}_{p,q}(\mathbb{R}^n)$. Moreover, in this case
	$$ \| C_{T, b}\|_{S_{p,q}(L_2(\mathbb{R}^n))}\approx_{n, p,q,T}\big(1+\|T(1)\|_{BMO(\mathbb{R}^n)}+\|T^*(1)\|_{BMO(\mathbb{R}^n)}\big)\|b\|_{\pmb{WB}_{p,q}(\mathbb{R}^n)}. $$
\end{rem}

\bigskip {\textbf{Acknowledgments.}} This work is part of Hao Zhang and Zhenguo Wei's theses defended respectively in June 2023 and June 2024 at Universit{\'e} de Franche-Comt{\'e}. We thank Professor Quanhua Xu for proposing this subject to us and for many helpful discussions and suggestions.

Hao Zhang would like to express his gratitude to Professor \'{E}ric Ricard for his kind invitation and hospitality, and helpful discussions during the visit to Caen in March and April 2023. During this visit, several results of this work were carried out.

We are very grateful to Professor \'{E}ric Ricard for sharing the new proof of Schatten class of martingale paraproducts in \cite{ERZH2024} with us, which inspires us a lot. We also thank Professor Tuomas Hyt\"{o}nen and Professor \'{E}ric Ricard for valuable comments and suggestions.

This work is partially supported by the French ANR project (No. ANR-19-CE40-0002).

\bibliographystyle{myrefstyle}
\bibliography{ref999}
\end{document}